\documentclass[12pt,CJK]{article}
\usepackage{amsfonts,amsmath,amssymb,amsthm,mathrsfs}
\usepackage{latexsym}
\usepackage{graphicx}
\usepackage{indentfirst}
\usepackage{color}
\usepackage{fancyhdr}
\usepackage[colorlinks,linktocpage,linkcolor=blue]{hyperref}

\theoremstyle{plain}
\newtheorem{thm}{Theorem}[section]
\newtheorem{definition}{Definition}
\newtheorem{lemma}[thm]{Lemma}

\newtheorem{corollary}[thm]{Corollary}
\newtheorem{remark}[thm]{Remark}
\numberwithin{equation}{section}

\theoremstyle{remark}

\def\Xint#1{\mathchoice
  {\XXint\displaystyle\textstyle{#1}}%
  {\XXint\textstyle\scriptstyle{#1}}%
  {\XXint\scriptstyle\scriptscriptstyle{#1}}%
  {\XXint\scriptscriptstyle\scriptscriptstyle{#1}}%
  \!\int}
\def\XXint#1#2#3{{\setbox0=\hbox{$#1{#2#3}{\int}$}
  \vcenter{\hbox{$#2#3$}}\kern-.5\wd0}}

\def\dashint{\Xint-}
\begin{document}
\allowdisplaybreaks
\pagestyle{myheadings}\markboth{$~$ \hfill {\rm Q. Xu,} \hfill $~$} {$~$ \hfill {\rm  } \hfill$~$}
%\author{Li Wang
%%\thanks{Email: lwang10@lzu.edu.cn.}
%\quad Qiang Xu
%\thanks{Corresponding author}
%\thanks{Email: xuqiang09@lzu.edu.cn.}
%\quad Peihao Zhao
%%\thanks{Email: phzhao@lzu.edu.cn.}
%\\
%School of Mathematics and Statistics, Lanzhou University, \\
%Gansu, 730000, PR China.
%\vspace{0.5cm}
%}

\author{
Li ~Wang~~~\\
School of Mathematics and Statistics, Lanzhou University, \\
Lanzhou, 710000, China.\vspace{0.2cm}\\
Qiang Xu
\thanks{Corresponding author.}
\thanks{Email: qiangxu@mis.mpg.de.}\\
Max Planck Institute for Mathematics in the Sciences, \\
Inselstrasse 22, 04103 Leipzig, Germany \vspace{0.2cm}\\
Peihao Zhao\\
School of Mathematics and Statistics, Lanzhou University, \\
Lanzhou, 710000, China. \vspace{0.2cm}
}

%

%\author{Weiren Zhao
%\thanks{Email: xuqiang@math.pku.edu.cn.}
%\thanks{This work was supported by the National Natural Science Foundation of China (Grant No. 11471147).}

\title{\textbf{Convergence rates on periodic homogenization
of p-Laplace type equations}}
\maketitle
\begin{abstract}
In this paper, we find some error estimates for periodic homogenization
of p-Laplace type equations under the same structure assumption on homogenized equations.
The main idea is that by adjusting the size of the difference quotient of the correctors
to make the convergence rate visible.
In order to reach our goal, the corresponding flux corrector with some properties
are developed.
Meanwhile, the shift-arguments is in fact applied down to $\varepsilon$ scale,
which leads to a new weighted type inequality for smoothing operator
with the weight satisfying Harnack's inequality in small scales.
As a result, it is possible to develop some large-scale estimates.
We finally mention that our approach brought in a systematic error
(this phenomenon will disappear in linear and non-degenerated cases),
which was fortunately a quantity $o(\varepsilon)$ here.
\end{abstract}

\tableofcontents

\section{Instruction and main results}
The purpose of the present paper is mainly to study the error estimates
for a class of quasilinear elliptic
equations, arising in the periodic homogenization theory.
More precisely, let $\Omega\subset\mathbb{R}^d$ be a bounded
domain, and consider the following elliptic equations in divergence form
depending on a parameter $\varepsilon > 0$,
\begin{eqnarray*}
(\text{D}_\varepsilon)\left\{\begin{aligned}
\mathcal{L}_{\varepsilon} u_\varepsilon \equiv
-\text{div}A(x/\varepsilon,\nabla u_\varepsilon) &= F &\qquad&\text{in}~~\Omega, \\
 u_\varepsilon &= g &\qquad& \text{on}~\partial\Omega.
\end{aligned}\right.
\end{eqnarray*}
%with the Dirichlet boundary condition
%\begin{equation}\label{pdeD:1}
% u = g
%\end{equation}
%or the Neumann boundary condition
%\begin{equation}\label{pdeN:1}
%n\cdot A(x/\varepsilon,\nabla u_\varepsilon) = h(x)
%\qquad \text{on}~\partial\Omega,
%\end{equation}
%where $n=(n_1,\cdots,n_d)$ is the outward unit normal
%vector to $\partial\Omega$.

Given three constants $\mu_0,\mu_1,\mu_2>0$, let us fix a function
$A:\mathbb{R}^d\times\mathbb{R}^d\to\mathbb{R}^d$ which satisfies the following conditions.
\begin{itemize}
\item For every $y, \xi\in\mathbb{R}^d$, we have
    the periodicity $A(y+z,\xi) = A(y,\xi)$ for $z\in\mathbb{Z}^d$, and $A(y,\cdot)$ is homogeneous with respect to the second variable, i.e.,
\begin{equation}\label{a:1}
  A(y,t\xi) = t^{p-1}A(y,\xi)  \quad\text{for~any}~t\geq 0.
\end{equation}
\item For any $y\in\mathbb{R}^d$, and $\xi,\xi^\prime\in\mathbb{R}^d$,
we impose the following coerciveness and growth conditions:
\begin{equation}\label{c:1}
\begin{aligned}
\mu_0 (|\xi|+|\xi^\prime|)^{p-2}|\xi-\xi^\prime|^2
\leq \big<A(y,\xi)-A(y,\xi^\prime),\xi-\xi^\prime\big>
\leq \mu_1(|\xi|+|\xi^\prime|)^{p-2}|\xi-\xi^\prime|^2.
\end{aligned}
\end{equation}
%where $\xi$ and $\xi^\prime$ are not equal to zero at the same time in the case of $1<p<2$.
\item The smoothness condition is also imposed, and there holds
\begin{equation}\label{a:4}
|A(y,\xi)-A(y^\prime,\xi)|\leq \mu_2|y-y^\prime||\xi|^{p-1}
\end{equation}
for all $y,y^\prime\in\mathbb{R}^d$ and $\xi\in\mathbb{R}^d$.
\end{itemize}
As a class of the examples, one may consider the weighted p-Laplace equations such as
$A(y,\xi) = a(y)|\xi|^{p-2}\xi$ (see \cite[pp.225]{P})
or the model $A_i(y,\xi) = a_{ij}(y)|\xi|^{p-2}\xi_j$
appeared in \cite[pp.117]{ABJLGP}. We also mention that the Lipschitz
continuity in $\eqref{a:4}$ can be weaken into a H\"older one, while
the present condition may highlight other key factors in the theory.

In terms of qualitative homogenization, our assumptions are merely special cases
considered in \cite{CDD,MD,FM}, so the process of homogenization of
$(\text{D}_\varepsilon)$ may be understood in their way, i.e., in the sense of G-convergence.
Let $F\in W^{-1,p^\prime}(\Omega)$ with $1/p+1/p^\prime = 1$ and assume
$u_\varepsilon$ is a weak solution to problem $(\text{D}_\varepsilon)$. It is known
that $u_\varepsilon \rightharpoonup u_0$ weakly in $W^{1,p}(\Omega)$, and
$A(x/\varepsilon,\nabla u_\varepsilon)\rightharpoonup\widehat{A}(\nabla u_0)$ weakly
in $L^p(\Omega;\mathbb{R}^d)$, as $\varepsilon\to 0$,
where $u_0$ is the solution to the effective (homogenized) equation
\begin{equation}\label{pde:1.3}
(\text{D}_0)\left\{\begin{aligned}
\mathcal{L}_{0} u_0 \equiv
-\text{div}\widehat{A}(\nabla u_0) &= F &\qquad&\text{in}~~\Omega, \\
 u_0 &= g &\qquad& \text{on}~\partial\Omega.
\end{aligned}\right.
\end{equation}
The function $\widehat{A}:\mathbb{R}^d\to\mathbb{R}^d$ is defined for every $\xi\in\mathbb{R}^d$ by
\begin{equation}\label{eq:1.1}
\widehat{A}(\xi) = \int_{Y=(-\frac{1}{2},\frac{1}{2}]^d} A(y,\xi+\nabla_y N(y,\xi))dy,
\end{equation}
where $N(y,\xi)$ is the so-called corrector, satisfying the cell problem
\begin{equation}\label{pde:1.2}
\left\{\begin{aligned}
&\text{div} A(y,\xi+\nabla_y N(y,\xi)) = 0 \quad\text{in}~ Y,\\
& N(\cdot,\xi)\in W^{1,p}_{per}(Y), \quad \dashint_{Y}N(\cdot,\xi) = 0,
\end{aligned}\right.
\end{equation}
where the definition $W^{1,p}_{per}(Y)$ may be found in \cite{MD,VSO},
and ``$\dashint_{Y}$'' is an average integral, defined in Subsection $\ref{subsec:1.1}$.
The existence theory may be found in \cite[Theorem 26.A]{Z}.

Concerning the quantitative estimates for homogenization problems, we hope the effective operator
admits the same structure as the assumption $\eqref{c:1}$,
\begin{equation}\label{c:2}
\tilde{\mu}_0 |\xi-\xi^\prime|^2\big(|\xi|+|\xi^\prime|\big)^{p-2}
\leq \big<\widehat{A}(\xi)-\widehat{A}(\xi^\prime),\xi-\xi^\prime\big>
\leq \tilde{\mu_1}(|\xi|+|\xi^\prime|)^{p-2}|\xi-\xi^\prime|^2
\end{equation}
for any $\xi,\xi^\prime\in\mathbb{R}^d$,
where $\tilde{\mu}_0,\tilde{\mu}_1$ may depend on $\mu_0,\mu_1,p$ and $d$,
which means that the homogenized operator $\mathcal{L}_0$
belongs to an analogy type of the operators as $\mathcal{L}_\varepsilon$ does, and
it is clear to be true in the special case $p=2$ (see for example \cite{ASC,ABJLGP,C,WXZ1}).
Under the assumptions $\eqref{a:1}$,$\eqref{c:1}$,
one just verified the growth condition of $\eqref{c:2}$ for the case $2\leq p<\infty$,
and the coerciveness of $\eqref{c:2}$ for the case $1<p\leq 2$
(see Lemma $\ref{lemma:2.5}$).
However, there are some examples (see Section $\ref{section:5}$) to
illustrate the coerciveness of $\eqref{c:2}$ for $2<p<\infty$,
and the growth condition for the case $1<p<2$.

On the other hand, it is more or little known that the quantitative estimates
in homogenization theory relies on a higher regularity of the weak solution to the effective
equation (a similar statement may be found in \cite[pp.477]{VSO}). Therefore,
the purpose of this hypothesis is to ensure that the weak solution to
$(\text{D}_0)$ owns a reasonable regularity (see Lemma $\ref{lemma:3.3}$),
otherwise we have to directly assume the effective solution $u_0$ admits these regularities.

\subsection{Motivation and relation to previous works}
Before recounting the main results, it is better to explain our initial ideas and source of them.
We begin the introduction from the two-scale expansion argument (see \cite{ABJLGP,VSO}), which is
\begin{equation*}
 u_\varepsilon = u_0(x) + \varepsilon u_1(x,y) + \varepsilon^2 u_2(x,y) + \cdots
\qquad\quad y=x/\varepsilon.
\end{equation*}
By a formal computation, we derive that
\begin{equation*}
 \underbrace{\text{div}_y A(y,\nabla u_0 + \nabla_y u_1) = 0}_{(E_1)}
\qquad\text{and}\qquad
 \text{div}_y A(y,\nabla u_1 + \nabla_y u_2) = F - \text{div}_x A(y,\nabla u_0 + \nabla_y u_1),
\end{equation*}
while the solvability of the second equation implies
\begin{equation*}
\underbrace{\text{div}_x \dashint_{Y} A(y,\nabla u_0 + \nabla_y u_1) dy = F}_{(E_2)}.
\end{equation*}
Obviously, the equations $(E_1)$ and $(E_2)$ inspired people to build the equation $\eqref{pde:1.2}$ and
the formula $\eqref{eq:1.1}$, respectively.
Thus, the image of $\nabla u_0$ determines $u_1$ and the cardinal number
of the set $\{\nabla u_0(x):x\in\Omega\}\subset\mathbb{R}^d$ reveals how many
equations like $\eqref{pde:1.2}$ need to be solved, if we want the formula
$\eqref{eq:1.1}$ to be fully understood.
In this sense, one may quickly realize the benefit of the linear equation,
which tells us $u_1$ is linearly dependent of $\nabla u_0$, and therefore we just solve
a finite number of equations that the correctors satisfy.

Hence,
from the view of reducing the number of equations like $\eqref{pde:1.2}$, we find
that the condition $\eqref{a:6}$ plays a similar role. And more importantly,
it also works for nonlinear cases,
which eventually help us to construct some examples to verify
the assumption $\eqref{c:2}$ (see Section $\ref{section:5}$).
Also, this view might be used to explain the motivation of introducing numerical correctors.
We refer the reader to \cite{EMZ,G,MD} and their references therein for this direction.

On the other hand, the non-degenerated case ($p=2$)
is very special. Although it does little to decrease
the number of the equations that the corrector satisfies,
it has been proved that $u_1$ is in fact Lipschitz continuous with respect to $\nabla u_0$
(see for example \cite[Lemma 2.1]{WXZ1}). Thus, we may consider the quantities
$w_\varepsilon = u_\varepsilon - v$ with $v = u_0 + \varepsilon N(y,\varphi)$, and
\begin{equation*}
\nabla w_\varepsilon = \nabla u_\varepsilon - \nabla v,
%\nabla v =  \nabla_yN(y,\varphi) +\varepsilon \nabla_{\xi} N(y,\varphi)\nabla \varphi
\end{equation*}
where $\varphi\in H_0^1(\Omega;\mathbb{R}^d)$ may be fixed later.
As in the linear case, one may derive a convergence rate from estimating
the quantity $\|\nabla w_\varepsilon\|_{L^2(\Omega)}$ by energy methods (see \cite[Theorem 1.1]{WXZ1}),
in which we believe that
the formula below reveals
an important information of convergence rates for elliptic equations with the divergence structure,
\begin{equation}\label{pri:1.1}
\begin{aligned}
\|\nabla w_\varepsilon\|_{L^2(\Omega)}^2
\leq  \bigg|\int_\Omega
&\Big[\underbrace{\widehat{A}(\nabla u_0)
- \widehat{A}(\varphi)}_{\text{T}_1:~\text{involving~the~shift~argument}} \\
& +\underbrace{\widehat{A}(\varphi) - A(x/\varepsilon,\varphi
+\nabla_yN(x/\varepsilon,\varphi))}_{\text{T}_2:~\text{giving~the~flux~tensor}}\\
& +\underbrace{A(x/\varepsilon,\varphi+\nabla_yN(x/\varepsilon,\varphi))
- A(x/\varepsilon,\nabla v)}_{\text{T}_3:~\text{producing~the remainder~terms~of~$\varepsilon$~order}}\Big] \cdot \nabla w_\varepsilon dx\bigg|
\end{aligned}
\end{equation}
(where $y = x/\varepsilon$, and see \cite[Lemma 3.1]{WXZ1}).
So, in order to better describe our ideas,
we would like to introduce the relevant calculations on the right-hand side of $\eqref{pri:1.1}$.
In general,
the computations of $T_3$ is not complicated as it appears,
which can be reduced to those of the other two terms.
As we have claimed above, the term $T_1$ is in fact related to the following estimate
\begin{equation}\label{pri:1.2}
  \|S_\varepsilon(\varphi) - \varphi\|_{L^2(\Sigma_{2\varepsilon})}
\leq  C\varepsilon \|\nabla\varphi\|_{L^2(\Sigma_{\varepsilon})}
\end{equation}
(a key inequality in the shift argument),
where $C$ depends only on $d$ (see Lemma $\ref{lemma:2.7}$), and
the set $\Sigma_{r}$ is referred to as the
``co-layer'' part of $\Omega$ (see Subsection $\ref{subsec:1.1}$).
If the smoothing operator $S_\varepsilon$ (see Definition $\ref{def:1}$) is
replaced by Steklov averaging operator, the reader may clearly find the relationship with
a shift argument developed by V. Zhikov and S. Pastukhova (see \cite{ZVVPSE,ZVVPSE1}), while
applying smoothing operator $S_\varepsilon$ to error estimates was
first suggested by Z. Shen in $\cite{S5}$, and this improvement
proved to be more flexibility in analysis.
Moreover, let $\varphi = \psi_{r}\nabla u_0$ in $\eqref{pri:1.2}$,
where $\psi_r\in C_0^1(\Omega)$ is a cut-off function defined in $\eqref{notation:2}$.
Then, the computations on $T_1$ is consequently translated into the layer type and co-layer type
estimates, which means we need to estimate
\begin{equation*}
  \|\nabla u_0\|_{L^2(\Omega\setminus\Sigma_{2r})}
\qquad\text{and}\qquad
  \varepsilon\|\nabla^2 u_0\|_{L^2(\Sigma_{r})}
\end{equation*}
for the homogenized solution to $(\text{D}_0)$, respectively.
If only a small rate of convergence is observed,
then Meyer's estimate coupled with some
interior estimates for higher derivatives is sufficient to serve this purpose, which
in fact will benefit a large scale estimate later on.
These results have been systematically stated in \cite{S4,S5} for linear systems,
as well as in \cite{WXZ1} for a non-degenerated equation ($p=2$).
Then we proceed to show two points
in the calculations of $T_2$ in $\eqref{pri:1.1}$,
which are all based upon the cell problem $\eqref{pde:1.2}$.
The first one derived from $\eqref{eq:1.1}$ is
$\int_Y T_2 dy = 0$, and the second one is $\text{div}_y(T_2) = 0$.
Using both of them, it is not hard to construct the so-called flux corrector, and
a later calculation will benefit from its antisymmetric property. Interestingly,
this construction does not depend on whether the original model we studied is linear or not.

So far, we have mainly introduced the methods used in non-degenerated
equations ($p=2$), as well as, our views on correctors and error estimates.
Thus, in terms of the nonlinear problem $(\text{D}_\varepsilon)$,
a natural question is how to extend the above theory established
for the special case $p=2$ to more general cases $1<p<\infty$.

The first challenge we confront with is that
the first order corrector $u_1$ is only H\"older continuous with respect to $\nabla u_0$
(i.e., the corrector $N(y,\cdot)$ is uniformly
H\"older continuous with respect to the second variable for a.e.
$y\in\mathbb{R}^d$, and see Lemma $\ref{lemma:2.1}$).
The idea is that we use the difference quotient as the substitute
for the derivatives of the corrector, and the first-order approximating corrector
is constructed in the form of
\begin{equation}\label{eq:1.2}
 V(x) = \nabla u_0(x)  + D^{h}\big(\varepsilon N(\cdot/\varepsilon,\varphi)\big)(x),
\end{equation}
where $D^{h} = (D_1^h,\cdots,D_d^h)$ is a difference quotient of size $h$
(see Subsection $\ref{subsec:1.1}$).  Then we consider the quantity
\begin{equation*}
  \nabla u_\varepsilon - V,
\end{equation*}
which will play a similar role as the quantity $\nabla w_\varepsilon$ did in $\eqref{pri:1.1}$.
Apparently, the main differences are caused
by the replacement of the derivative by the difference quotient.
We are also aware of the following fact: for any $x\in\Sigma_\varepsilon$,
there holds
\begin{equation}\label{f:1.1}
  D^h_i (f)(x) = \nabla_i f(x) + o(1),
\end{equation}
as $h$ goes to zero, provided $f\in C^1(\Omega)$. As usual,
the calculations like ``$o(1)\times \varepsilon = o(\varepsilon)$'' and
``$o(1)\times C = o(1)$ for a constant $C$'' are
agreed to be true,
and $o(\varepsilon)$ is referred as to the higher order infinitesimal of $\varepsilon$,
as $\varepsilon$ goes to zero. Therefore, our main results may be shown by ignoring
the quantity $o(\varepsilon)$, which may be regarded as a kind of systematic error.
Thus, the main problem is to find the accuracy of $h$ in the formula $\eqref{eq:1.2}$.

%Along this line of thinking,
%it seems that everything is not much different from the above discussion for the case $p=2$.
%In fact, the main problem here is the accuracy of $h$ in the formula $\eqref{eq:1.2}$.
%As we have mentioned before, the convergence rate can only be observed on certain size of $h$, and
%this involves a subtle analysis and meticulous calculations.

\subsection{Main results and some comments}
For the ease of the statement, we impose the following index throughout the paper,
\begin{equation}\label{notation:1}
 \alpha = \left\{
            \begin{array}{ll}
              1/(3-p), & \hbox{$1<p\leq 2$;} \\
              2/p, & \hbox{$2<p<\infty$,}
            \end{array}
          \right.
\quad \beta = 1/p - 1/q,
\quad  \gamma = \left\{
            \begin{array}{ll}
              (p-1)/(3-p), & \hbox{$1<p\leq 2$;} \\
              ~2/p, & \hbox{$2\leq p<\infty$,}
            \end{array}
          \right.
\end{equation}
where $p<q\leq p(1+\delta)$ and $\delta\in(0,1)$
is usually very small, which is actually determined by Meyer's estimates (see Remark $\ref{remark:4.1}$).
Here $\beta\in(0,\frac{\delta}{1+\delta}]$, and $\alpha,\gamma\in(0,1]$ are used to describe the H\"older's continuity of $N(y,\cdot)$
and the corresponding flux corrector $E(y,\cdot)$ with respect to the second variable, respectively.
We now state the main results of the paper.

\begin{thm}[convergence rates]\label{thm:1.1}
Let $\Omega\subset\mathbb{R}^d$ be a bounded Lipschitz domain
with $r_0 = \emph{diam}(\Omega)$, and $0<\varepsilon\ll r_0$. Assume that
$A$ satisfies the conditions $\eqref{a:1}, \eqref{c:1}, \eqref{a:4}$.
Then there exist $\theta\in(0,1)$ and a constant $C$,
depending on $\mu_1,\mu_2,\mu_3,p$ and $d$, such that
\begin{equation}\label{pri:1.3}
\|\nabla N(\cdot,\xi)\|_{L^\infty(Y)}
+[\nabla N(\cdot,\xi)]_{C^{0,\theta}(Y)}
\leq C\Big(\dashint_{Y}|N(\cdot,\xi)|^p\Big)^{1/p}
\leq C |\xi|
\end{equation}
holds for any $\xi\in\mathbb{R}^d$. Moreover,
suppose that $\widehat{A}$ satisfies the condition $\eqref{c:2}$,
and $u_\varepsilon, u_0\in W^{1,p}(\Omega)$
satisfy $\mathcal{L}_\varepsilon u_\varepsilon = \mathcal{L}_0 u_0 =0$ in
$\Omega$ and $u_\varepsilon = u_0$ on $\partial\Omega$, with assumption that
\begin{equation}\label{a:5}
M := \|\nabla u_0\|_{C^{0,\vartheta}(\Omega)}
\quad \text{with~some~}\vartheta\in(0,1),
\quad\text{and}\quad
\Omega_{M,\varepsilon} = \{x\in\Omega:
|\nabla u_0|\leq M\varepsilon^{\vartheta}\}.
\end{equation}
Let $\varphi_\varepsilon = S_{\varepsilon}(\psi_{4\varepsilon}\nabla u_0)$, and
the first-order approximating corrector is in the form of $\eqref{eq:1.2}$, i.e.,
\begin{equation*}
 V(x) = \nabla u_0(x)  + D^{h}\big(\varepsilon N(\cdot/\varepsilon,\varphi_\varepsilon)\big)(x),
\end{equation*}
where we prefer
$h=\varepsilon^\tau$ and $1<\tau<1/(1-\alpha)$ in advance,
and $\psi_{4\varepsilon}\in C_0^1(\Omega)$ is cut-off function satisfying $\eqref{notation:2}$. Then,
by ignoring a systematic error $o(\varepsilon)$, we have the following conclusions.
\begin{itemize}
\item For the case $1<p<2$, one may derive that
\begin{equation}\label{pri:1.5}
\|\nabla u_\varepsilon - V\|_{L^p(\Omega)}
\lesssim |\Omega_{M,\varepsilon}|^{\beta}\|\nabla u_0\|_{L^q(\Omega_{M,\varepsilon})}
+ \varepsilon^{\theta(\tau-1)}\|\nabla u_0\|_{L^q(\Omega)},
\end{equation}
where the range of $\tau$ is determined by
\begin{equation}\label{range:1}
1 <\tau <\frac{1-\gamma+\theta(p-1)+\gamma\beta}{1-\gamma+\theta(p-1)}.
\end{equation}
If we further consider the connection to the non-degenerated case $(p=2)$, there holds
\begin{equation}\label{pri:1.4}
\|\nabla u_\varepsilon - V\|_{L^p(\Omega)}
\lesssim  (1-I_{\{p=2\}})|\Omega_{M,\varepsilon}|^{\beta}\|\nabla u_0\|_{L^q(\Omega_{M,\varepsilon})}
+ \varepsilon^{\alpha\beta+\frac{(1-\gamma)(1-\tau)}{p-1}}\|\nabla u_0\|_{L^q(\Omega)}
\end{equation}
for any $1<p\leq 2$, where $I_{\{p=2\}}$ denotes the character function of the set $\{p=2\}$, and
it requires the range of $\tau$ to be
\begin{equation}\label{range:2}
\begin{aligned}
\frac{1-\gamma+\theta(p-1)+\gamma\beta}{1-\gamma+\theta(p-1)} < \tau
< \frac{1-\gamma+\gamma\beta}{1-\gamma}.
\end{aligned}
\end{equation}
  \item If $2<p<\infty$, then we have
\begin{equation}\label{pri:1.6}
\begin{aligned}
&\|\nabla u_\varepsilon - V\|_{L^p(\Omega)}\\
&\lesssim |\Omega_{M,\varepsilon}|^{\beta}\|\nabla u_0\|_{L^q(\Omega_{M,\varepsilon})}
+ \left\{\begin{aligned}
&\varepsilon^{(\frac{p-2}{p-1})[\beta+(1-\alpha)(1-\tau)]}\|\nabla u_0\|_{L^q(\Omega)}
&\text{~if~}& 2<p<3;\\
&\varepsilon^{(\frac{1}{p-1})[\beta+(1-\alpha)(1-\tau)]}\|\nabla u_0\|_{L^q(\Omega)}
&\text{~if~}& 3\leq p<\infty,
\end{aligned}\right.
\end{aligned}
\end{equation}
when $\tau$ is chosen such that
\begin{equation}\label{range:3}
\frac{1-\alpha+\theta+\beta}{1-\alpha+\theta}<\tau
<\frac{1-\alpha+\beta}{1-\alpha}.
\end{equation}
%  \item In terms of  $4\leq p<\infty$, there holds the following error estimate
%\begin{equation}\label{pri:1.7}
%\begin{aligned}
%\|\nabla u_\varepsilon - V\|_{L^p(\Omega)}
%&\lesssim |\Omega_{M,\varepsilon}|^{\beta}\|\nabla u_0\|_{L^q(\Omega_{M,\varepsilon})}
%+ \varepsilon^{(\frac{1}{p-1})[\beta+(1-\alpha)(1-\tau)]}
%\Bigg\{\Big(\int_{\Omega}|\nabla u_0|^q dx\Big)^{\frac{1}{q}} \\
%&+ \varepsilon^{(\frac{1}{p-1})[\beta+(1-\alpha)(1-\tau)]}
%\Big(\int_{\Omega}|\nabla u_0|^q dx\Big)^{\frac{2}{q(p-1)}}
%\Big(\int_{\Omega}|\nabla u_0|^{(\frac{p-2}{p-3})p}dx\Big)^{\frac{p-3}{p(p-1)}}\Bigg\},
%\end{aligned}
%\end{equation}
%provided $\tau$ satisfies the same range as in $\eqref{range:3}$.
\end{itemize}
Here $\lesssim$ means $\leq$ up to a multiplicative constant which depends on
$\mu_0,\mu_1,\mu_2,\tilde{\mu}_0,\tilde{\mu}_1, p,q,d$ and $r_0$.
\end{thm}

We really hope that the theorem can be written succinctly,
and as a compensation,
we make an effort to explain as much as possible, and some comments are as follows.

\vspace{0.1cm}
\noindent\emph{Comment 1}.
Let us begin from the explanation of the estimate $\eqref{pri:1.4}$.
For the special case $p=2$, it will become
\begin{equation*}
\|\nabla u_\varepsilon - V\|_{L^2(\Omega)}
\lesssim \varepsilon^{\beta}\|\nabla u_0\|_{L^{q}(\Omega)}
\end{equation*}
with any $\tau>(1+\beta/\theta)$, where $q=2(1+\delta)$. As mentioned before
the quantity $\nabla_x(\varepsilon N(x/\varepsilon,\varphi))$ exists in such the case, and so
the above estimate implies
\begin{equation*}
\|\nabla w_\varepsilon\|_{L^2(\Omega)}
\lesssim \varepsilon^{\beta}\|\nabla u_0\|_{L^2(\Omega)},
\end{equation*}
which is exactly what we have proved in \cite[Theorem 1.1]{WXZ1}.
So, this theorem may be roughly regarded as an extension of the previous one, and
our purpose is also to study
a uniform estimate for the solution of $(\text{D}_\varepsilon)$ through
the convergence rates (see Theorem $\ref{thm:1.2}$). For this consideration
the presence of the q-energy norm (i.e., $\|\nabla u_0\|_{L^q(\Omega)}$ which is a little stronger
than p-energy norm) in the right-hand side
of error estimates is quite convenient to us even when a higher rate of convergence could be derived by
a stronger norms.

\vspace{0.1cm}
\noindent\emph{Comment 2}.
The first term in the right-hand side of
the estimates $\eqref{pri:1.4}$, $\eqref{pri:1.5}$ and $\eqref{pri:1.6}$
has two possible directions:
\begin{equation}\label{pri:1.8}
|\Omega_{M,\varepsilon}|^{\beta}\|\nabla u_0\|_{L^p(\Omega_{M,\varepsilon})}
\leq \left\{
\begin{aligned}
 &|\Omega_{M,\varepsilon}|^{\beta}\|\nabla u_0\|_{L^p(\Omega)};\\
 &|\Omega_{M,\varepsilon}|^{\beta+\frac{1}{p}} M\varepsilon^{\vartheta}.
\end{aligned}\right.
\end{equation}
Although the first line of $\eqref{pri:1.8}$
does not clearly show any convergence rate, it is proved to be
useful and enlightening for investigating  a large scale estimate
(see Lemma $\ref{lemma:4.1}$ and Theorem $\ref{thm:1.2}$).
On the contrary, the second line of $\eqref{pri:1.8}$ gives the rate of convergence on
account of the assumption $\eqref{a:5}$.
However, it actually indicates that the process of
homogenization are not truely observable when fluctuations are limited in a small scale.

\vspace{0.1cm}
\noindent\emph{Comment 3}. Two reasons leads to the assumption $\eqref{a:5}$.
\emph{On the one hand}, the $C^{1,\vartheta}$ smoothness imposed to the effective solution
is reasonable under the precondition $\eqref{c:2}$
(although it usually admits such the regularity in a local sense), which ensures that
the calculation like $\eqref{f:1.1}$ conforms to the rule and ultimately
produces a quantity $o(\varepsilon)$. \emph{On the other hand}, the set
$\Omega_{M,\varepsilon}^c = \{x\in\Omega:|\nabla u_0(x)|>M\varepsilon^\vartheta\}$ is
in fact more important to us, since one may acquire a Harnack's inequality
\begin{equation*}
  \sup_{Y_{\varepsilon}^k} |\nabla u_0|
  \leq 2\inf_{Y_{\varepsilon}^k}|\nabla u_0|
\end{equation*}
for any $Y_{\varepsilon}^k\subset\Omega_{M,\varepsilon}^c$ with
$|Y_{\varepsilon}^k| = \varepsilon^d$ (where $\Omega_{M,\varepsilon}^c$ is an open set, and see
Subsection $\ref{subsec:1.1}$ for $Y_\varepsilon^k$),
and this together with the estimate $\eqref{pri:1.2}$ implies the following weighted inequality
\begin{equation}\label{pri:1.9}
  \Big(\int_{\Omega_{M,\varepsilon}^c} |\varphi-S_\varepsilon(\varphi)|^p |\nabla u_0|^r dx
\Big)^{1/p}
\lesssim \varepsilon\Big(\int_{\Omega}|\nabla\varphi|^p|\nabla u_0|^r dx\Big)^{1/p}
\end{equation}
for any $1<p,r<\infty$, where $|\nabla u_0|$ is referred to as a weight (see Remark $\ref{remark:2.3}$).
If familiar with the regularity theory of p-Laplace type equations, one may immediately realize
that the estimate $\eqref{pri:1.9}$ is in help of giving the quantity of the left-hand side below,
\begin{equation*}
\begin{aligned}
&\int_{\Sigma_{2r}} |\nabla u_0|^{p-2}|\nabla^2 u_0|^2 dx
\lesssim \frac{1}{r^2}\int_{\Sigma_{r}} |\nabla u_0|^{p} dx \quad\text{for~}p>2, \\
& \int_{\Sigma_{2r}}|\nabla^2 u_0|^p dx
\lesssim \frac{1}{r^p}\int_{\Sigma_{r}} |\nabla u_0|^{p} dx \quad\text{for~}1<p\leq 2,
\end{aligned}
\end{equation*}
(see Lemma $\ref{lemma:3.3}$). Thanks to them, the stated error estimates
$\eqref{pri:1.4}$, $\eqref{pri:1.5}$ and $\eqref{pri:1.6}$
do not involve the higher derivatives any more.
The above computations serve as a counterpart of those for $T_1$ in $\eqref{pri:1.1}$,
which will be addressed in Lemmas $\ref{lemma:3.6}, \ref{lemma:3.7}$.
Besides,
the set $\Omega_{M,\varepsilon}$ includes the degenerate information
of the homogenized equation $(\text{D}_0)$, and it seems that
this kind of information can not be known
in advance through the homogenization of the equation
$(\text{D}_\varepsilon)$, which eventually leads to
the separated two parts of the right-hand side in these error estimates.

%\vspace{0.1cm}
%\noindent\emph{Comment 4}. The main reason leads to three cases in the theorem is that
%under the assumptions $\eqref{a:1}$ and $\eqref{c:1}$, $(\text{D}_\varepsilon)$
%shows some properties with different forms in the cases $1<p < 2$ (singular case),
%$p=2$ (non-degenerated case)  and $2<p<\infty$ (degenerated case), respectively.
%In this sense, these error estimates may be placed in two classes:
%\begin{equation*}
%\begin{aligned}
%\underbrace{\text{the~estimates~}\eqref{pri:1.5},\eqref{pri:1.4};}_{
%\text{singular~}(+\text{~non-degenerated~class})}
%\qquad\text{and}\qquad
%\underbrace{\text{the~estimates~}\eqref{pri:1.6}}_{
%\text{degenerated~class}}
%\end{aligned}
%\end{equation*}
%The second reason is that the size of $\tau$ strongly relies on the rate of convergence
%(as well as, on the value of $p$), which may be observed from the estimates
%$\eqref{pri:1.5}$ and $\eqref{pri:1.4}$.
%In terms of $\eqref{pri:1.6}$, a clue may be found in Lemma $\ref{lemma:3.2}$.

\vspace{0.1cm}
\noindent\emph{Comment 4}. In both periodic and aperiodic settings,
error estimates have always been a hot topic in homogenization theory
(see for example \cite{ASC,AS,EMZ,G,GNF,GO,KLS1,P,S4,TS1,X3,WXZ1,ZVVPSE,ZVVPSE1} and
their references therein for more details), while there is few contributions in this field
concerning p-Laplace type equations, to the authors' best acknowledge. In fact,
there seems to be an interesting problem: except of the special cases in Section $\ref{section:5}$, what kind of
conditions can guarantee the homogenized operator $\mathcal{L}_0$ admits the structure
$\eqref{c:2}$.

\vspace{0.2cm}
As an application of Theorem $\ref{thm:1.1}$, we derive the following result.

\begin{thm}[large-scale H\"older estimates]\label{thm:1.2}
Let $B=B(0,1)\subset\Omega$ and $1<p<\infty$.
Assume the same conditions as in Theorem $\ref{thm:1.1}$.
Suppose that $u_\varepsilon\in W^{1,p}(B)$
satisfies $\mathcal{L}_\varepsilon u_\varepsilon = 0$ in $B$.
Then for any $\kappa\in(0,d)$, there holds an uniform estimate
\begin{equation}\label{pri:1.11}
  \int_{B(x,r)}|\nabla u_\varepsilon|^p dz
\lesssim \Big(\frac{r}{R}\Big)^{d-\kappa}\int_{B(x,R)} |\nabla u_\varepsilon|^p dz
\end{equation}
for any $\varepsilon\leq r\leq R\leq (1/2)$, where the up-to constant depends on
$\mu_1,\mu_2,\mu_3,p,d$ and $\kappa$.
\end{thm}

\noindent\emph{Comment 5}. Recently, the large-scale estimates received important developments
in quantitative homogenization theory, and without attempting to be exhaustive
we refer the reader to \cite{ASC,AS,GNF1,KLS,S5} for more details in
periodic and non-periodic settings, while the first systematic study in this area
went back to M. Avellaneda, F. Lin \cite{MAFHL}. Concerning the equation $(\text{D}_\varepsilon)$,
we believe that the estimate $\eqref{pri:1.11}$ is still true for $\kappa=0$
(i.e., large-scale Lipschitz estimates), and may further develop
a large-scale $L^q$ estimate with $q>p$ in an average sense. However,
we plan to investigate this topic in a separated work.
Since the method developed for large-scale estimates are based upon an iteration argument,
it is a kind of nonlinear approaches, which in fact is independent of the structure of equations.
Roughly speaking, the proof of Theorem $\ref{thm:1.2}$
follows the same strategy introduced in the previous work \cite{WXZ1}
(see \cite{S4,S5} for linear systems).
Here, we just mention that
the crucial step between Theorems $\ref{thm:1.1}$ and
$\ref{thm:1.2}$ is based upon the following corollary of Theorem $\ref{thm:1.1}$, which
will be concretely known as an approximating lemma (see Lemma $\ref{lemma:4.1}$).

\begin{corollary}\label{corollary:1.1}
Assume the same conditions as in Theorem $\ref{thm:1.1}$. If the left-hand side of
the estimates $\eqref{pri:1.5}$,  $\eqref{pri:1.4}$ and $\eqref{pri:1.6}$ is
replaced by $\|u_\varepsilon - u_0\|_{L^p(\Omega)}$, respectively.
Then, we have the same conclusions.
\end{corollary}

\noindent\emph{Comment 6}.
Here we just mention that Corollary $\ref{corollary:1.1}$ can not be derived from
Poincar\'e's inequality. Instead, it follows from the fact
$\|u_\varepsilon - u_0\|_{L^p(\Omega)}=  \|\nabla_i (u_\varepsilon - u_0)\|_{W^{-1,p}(\Omega)}$
with $i=1,\cdots,d$,
and the related details have been shown in the proof of Lemma $\ref{lemma:4.1}$.

\subsection{Outline of the proof of Theorem $\ref{thm:1.1}$}\label{subsec:1}

In this subsection, we show a strategy of the proof, which consists of four ingredients in order.

\noindent\textbf{Ingredient 1}. We manage to obtain the following formula,
which is the key ingredient, analogy to the estimate $\eqref{pri:1.1}$ for the non-degenerated case.
\begin{equation}\label{pri:1.10}
\begin{aligned}
\|\nabla u_\varepsilon -V\|_{L^p(\Omega)}^p
\leq  \bigg|\int_\Omega
&\Big[\underbrace{\widehat{A}(\nabla u_0)
- \widehat{A}(\varphi)}_{\text{P}_1:~\text{involving~the~estimate like $\eqref{pri:1.9}$}} \\
& +\underbrace{\widehat{A}(\varphi) - A(x/\varepsilon,\varphi
+\nabla_yN(x/\varepsilon,\varphi))}_{\text{P}_2:~\text{handed~by~the~flux~tensor}}\\
& +\underbrace{A(x/\varepsilon,\varphi+\nabla_yN(x/\varepsilon,\varphi))
- A(x/\varepsilon,V)}_{
\text{P}_3:~\text{producing~the remainder~terms}}\Big]
\cdot \big(\nabla u_\varepsilon-V\big) dx\bigg|
+ \underbrace{o(\varepsilon)}_{\text{a systematic error}}
\end{aligned}
\end{equation}
(where $y=x/\varepsilon$).

\vspace{0.1cm}
\noindent\textbf{Ingredient 2}. Estimating the term $P_1$ forced us to find the weighted type estimate $\eqref{pri:1.9}$,
which means that the shift argument still works for nonlinear equations,
at least in the small scales.

\vspace{0.1cm}
\noindent\textbf{Ingredient 3}. The core idea comes from dealing with the term $P_3$ of $\eqref{pri:1.10}$,
and we realized the importance of the following equality
\begin{equation}\label{eq:1.3}
D_i^{h}(\varepsilon N(y,\varphi)) =
\underbrace{\int_0^1 \partial_{i} N(y+t(h/\varepsilon)e_i,\varphi) dt}_{R_1}
+\underbrace{\varepsilon h^{\alpha-1}
\Phi(y,\varphi)\big|D_i^{h}\varphi\big|^\alpha}_{R_2},
\end{equation}
and it inspires us to discover that
there is an opportunity to observe the rate of convergence
by adjusting the size of $h$. A natural thinking is to set $h=\varepsilon^\tau$, and reduce
the problem to find out a suitable range of $\tau$, independent of $\varepsilon$.
Note that we have $R_1 = 0$ by periodicity of the corrector when $\tau =1$. Since we have to
estimate the quantity
\begin{equation}\label{eq:1.4}
 \|D_i^{h}(\varepsilon N(y,\varphi)) - \nabla_y N(y,\varphi)\|_{L^p(\Omega)}
\end{equation}
produced by the term $P_3$ in $\eqref{pri:1.10}$, the higher regularity of
the corrector turns to be crucial in these computations,
and it is one of roles that the imposed assumption $\eqref{a:4}$ plays.
Meanwhile, the related estimate of $\Phi$ in $R_2$ leads to most of complicated parts
in the paper.

\vspace{0.1cm}
\noindent\textbf{Ingredient 4}. Concerning $P_2$ in $\eqref{pri:1.10}$, we also employ the antisymmetry property of
the flux corrector (see Lemma $\ref{lemma:3.5}$). The difference quotient presented here
brings us more technical difficulties compared to the computation of
$T_2$ in $\eqref{pri:1.1}$
(see Lemma $\ref{lemma:3.5}$).
We finally mention that those arguments developed for the corrector in $\eqref{eq:1.3}$ and $\eqref{eq:1.4}$
also worked for the flux corrector, although
a little more efforts and carefulness have to be paid.

\subsection{Notation}\label{subsec:1.1}

Throughout the paper, we make use of the following notation:
\begin{itemize}
  \item $d\geq 1$ is the dimension, and $1<p<\infty$ is the characterization of the equation type,
and $\alpha,\beta,\gamma$ are defined in $\eqref{notation:1}$, and $\delta\in(0,1)$ is known as
a Meyer's index;
  \item $\Omega\subset\mathbb{R}^d$ is a bounded Lipschitz domain with $r_0$ being
the diameter of $\Omega$, and
  $\Sigma_{r}=\{x\in\Omega:\text{dist}(x,\partial\Omega)>r\}$ represents the
so-called ``co-layer'' part of $\Omega$, and its layer part is denoted
by $\Omega\setminus\Sigma_{r}$, and $\psi_r\in C_0^1(\Omega)$ is a cut-off function associated with
$\Sigma_r$,
satisfying
\begin{equation}\label{notation:2}
 \psi_r = 1 \quad\text{in}~\Sigma_{2r},
 \qquad
 \psi_r = 0 \quad\text{outside}~\Sigma_{2r},
 \qquad
 |\nabla \psi_r|\leq C/r;
\end{equation}
  \item $D^{h} = (D_1^h,\cdots,D_d^h)$, where the $i^{\text{th}}$ difference quotient of size $h$ is
defined by $D_i^h (f)(x) = \frac{f(x+he_i)-f(x)}{h}$ for any $x\in\Sigma_{\varepsilon}$ with
$0<h<\varepsilon$;
\item $\nabla = (\nabla_1,\cdots,\nabla_d)$ with
$\nabla_i = \nabla_{x_i} = \partial/\partial x_i$, and $\nabla_y$ (or $\text{div}_y$)
means we take a gradient (or divergence) with respective to the variable $y$,
when we need to emphasize that it is different from the default case $\nabla$.
The notation $\text{div}(F) =\nabla \cdot F = \sum_{i=1}^d\nabla_i F_i$,
and $\nabla^2 f = \nabla^2_{ij} f = \frac{\partial^2 f}{\partial x_i\partial x_j}$
denotes the divergence of $F$, and Hessian matrix of $f$, respectively. In general,
we denotes $\nabla_{z_i} N(x/\varepsilon+\varpi,\varphi)$
by $\partial_i N(x/\varepsilon+\varpi,\varphi)$ without confusion, where $z=x/\varepsilon+\varpi$;

\item $o(1)$ represents an infinitesimal quantity, and $o(\varepsilon)$ is
a higher order infinitesimal of $\varepsilon$ as $\varepsilon\to 0$;
\item $\big<A,B\big> = \sum_{i=1}^{d} A_i B_i$,
and $|\xi| = \big<\xi,\xi\big>^{\frac{1}{2}}$, and $|\Omega|$ represents the volume of $\Omega$,
and $\dashint_{\Omega} = \frac{1}{|\Omega|}\int_{\Omega}$,
and $Y_{\varepsilon}^k = \varepsilon(k + Y)$ for $k\in \mathbb{Z}^d$,
where $Y=(-\frac{1}{2},\frac{1}{2}]$;
\item $\lesssim$ and $\gtrsim$ stand for $\leq$ and $\geq$ up to a multiplicative constant which
may depend on any given constants such as $\mu_0$,$\mu_1$,$\mu_2$,$d$,$p$,$q$,
$\theta$,$\vartheta$, $\kappa$,
$\delta$,$r_0$, but independent of $\varepsilon$, and rescaling parameters such as $r,R\in(0,1)$,
and we use $\ll$ instead of $\lesssim$ to indicate that multiplicative constant
is quite close to zero (although it is still a positive number).
\end{itemize}

\subsection{Structure of the paper}
In Section $\ref{sec:2}$, we mainly addressed the correctors and flux correctors.
Concerned with the properties of the corrector $N(\cdot,\xi)$, we verified that
$N(\cdot,\cdot)\in C^{1,\theta}\times C^{\alpha}$, and in
particular, gave a proof of $\eqref{pri:1.3}$. As an application,
the condition $\eqref{c:2}$ was partially proved. Some new tricks were dedicated to
estimating the quantity $\|D^h(\varepsilon N(\cdot/\varepsilon,\varphi))\|_{L^p(\Omega)}$.
In terms of the flux corrector $E(\cdot,\xi)$, we similarly proved that
$E(\cdot,\cdot)\in C^{1,\rho}\times C^{\gamma}$, where $\rho=\min\{\theta,(p-1)\theta\}$,
while we merely demonstrated $D^h(\varepsilon E(\cdot/\varepsilon,\varphi))\in L^r(\Omega)$ with
$1<r\leq p/(p-1)$. In the end of this section, we defined the smoothing operator and stated its
properties. In Section $\ref{sec:3}$, the purpose was to show a proof of Theorem $\ref{thm:1.1}$.
The structure of this section was in fact consistent with that shown in Subsection $\ref{subsec:1}$.
In Section $\ref{sec:4}$, a main work was to show the so-called approximating lemma
(see Lemma $\ref{lemma:4.1}$), and then we gave a proof of Theorem $\ref{thm:1.2}$, as well as,
Caccioppolli's inequality (see Lemma $\ref{lemma:4.2}$). Finally, in Section $\ref{section:5}$
we constructed some special examples to satisfy the assumption $\eqref{c:2}$.

The interest of the present paper is
finding a new way (independent of numerical correctors) to obtain the rate of
convergence for nonlinear homogenization problems, which could be further applied
to studying large scale estimates. As a theory, it is more or less self-enclosed as we hoped.
However, it can not directly help us
understand the mechanism of homogenization for nonlinear problems
in a little larger classes of operators, since this method
does not tell us why the homogenized operator runs out of the class
 $\eqref{c:1}$ in most cases, and we will make an effort for this direction
in future.

\section{Preliminaries}\label{sec:2}

%We note that the condition $\eqref{c:1}$ implies the following conditions:
%\begin{enumerate}
%  \item if $p\geq 2$,
%  \begin{equation}\label{a:2}
%\begin{aligned}
%\big<A(y,\xi)-A(y,\xi^\prime),\xi-\xi^\prime\big>  &\geq \mu_0
%(|\xi|+|\xi^\prime|)^{p-2}|\xi-\xi^\prime|^2, \\
%\big|A(y,\xi)-A(y,\xi^\prime)\big| &\leq \mu_1(|\xi|+|\xi^\prime|)^{p-2}|\xi-\xi^\prime|;
%\end{aligned}
%\end{equation}
%  \item  if $1<p\leq 2$,
%\begin{equation}\label{a:3}
%\begin{aligned}
%\big<A(y,\xi)-A(y,\xi^\prime),\xi-\xi^\prime\big>  &\geq \mu_0(|\xi|+|\xi^\prime|)^{p-2}|\xi-\xi^\prime|^2, \\
%\big|A(y,\xi)-A(y,\xi^\prime)\big| &\leq \mu_1|\xi-\xi^\prime|^{p-1},
%\end{aligned}
%\end{equation}
%\end{enumerate}
%and they will make the following

\subsection{Corrector and its properties}

\begin{lemma}\label{lemma:2.1}
Assume that $A$ satisfies the conditions $\eqref{a:1}$, $\eqref{c:1}$.
Let $N(\cdot,\xi)\in W_{per}^{1,p}(Y)$ be the weak solution to the equation $\eqref{pde:1.2}$, and then
for any $\xi\in\mathbb{R}^d$, we
have the following estimate
\begin{equation}\label{pri:2.1}
 \dashint_{Y} |N(\cdot,\xi)|^p + \dashint_{Y} |\nabla N(\cdot,\xi)|^p \leq C|\xi|^p.
\end{equation}
Moreover, by setting
\begin{equation}\label{eq:2.2}
P(y,\xi) = \xi + \nabla_y N(y,\xi)
\end{equation}
one may derive
\begin{equation}\label{pri:2.2}
\Big(\dashint_{Y} |P(\cdot,\xi)-P(\cdot,\xi^\prime)|^p \Big)^{\frac{1}{p}}\leq
C\left\{\begin{aligned}
& |\xi-\xi^\prime|^{\frac{2}{p}}\big(|\xi|+|\xi^\prime|\big)^{1-\frac{2}{p}}
&~~~\text{if}~~& 2\leq p <\infty;\\
&|\xi-\xi^\prime|^{\frac{1}{3-p}}\big(|\xi|+|\xi^\prime|\big)^{\frac{2-p}{3-p}} &~~~\text{if}~~& 1<p\leq 2,
\end{aligned}\right.
\end{equation}
where $C$ depends only on $\mu_0,\mu_1,p$ and $d$.
\end{lemma}

\begin{proof}
These results have already been in \cite[Lemma 3.4]{MD}.
\end{proof}

\begin{lemma}\label{lemma:2.6}
Suppose that $A$ satisfies the conditions $\eqref{a:1}$ and $\eqref{c:1}$. Then we conclude that
$N(\cdot,\xi)$ is H\"older continuous with an upper bound estimate
\begin{equation}\label{pri:2.19}
  \|N(\cdot,\xi)\|_{L^\infty(Y)}\lesssim |\xi|
\end{equation}
for any $\xi\in\mathbb{R}^d$.
Moreover, if we additionally impose the smoothness condition $\eqref{a:4}$, then there exist
$\theta\in (0,1)$ and a constant $C$, depending on $\mu_0,\mu_1,\mu_2,p$ and $d$, such that
\begin{equation}\label{pri:2.9}
\|P(\cdot,\xi)\|_{L^\infty(Y)}
+[P(\cdot,\xi)]_{C^{0,\theta}(Y)}
\leq C\Big(\dashint_{Y}|P(\cdot,\xi)|^p\Big)^{1/p}
\leq C |\xi|
\end{equation}
holds for any $\xi\in\mathbb{R}^d$.
\end{lemma}

\begin{proof}
The estimate $\eqref{pri:2.19}$ have be shown in \cite[Theorem 4]{C} for the case $p=2$.
In terms of $p\not=2$, by setting $u = N(y,\xi)+y\cdot\xi$, it follows from \cite[Theorem 3.9]{MZ} that
\begin{equation}\label{f:2.21}
  \|u\|_{L^\infty(B(0,r))} \lesssim \Big(\dashint_{B(0,2r)}|u|^p\Big)^{1/p}
\end{equation}
for any $0<r<(1/2)$. Obviously, this implies the stated estimate $\eqref{pri:2.19}$.
Then we turn to address the estimate $\eqref{pri:2.9}$.  Under the smoothness assumption
$\eqref{a:4}$, it is known from \cite[Theorems 1,2]{D} that
\begin{equation*}
\begin{aligned}
\|\nabla u\|_{L^\infty(B(0,r))} + [\nabla u]_{C^\infty(B(0,r))}
&\leq C\big(r,\mu_0,\mu_1,d,p,\|u\|_{L^\infty(B(0,2/3))}\big) \\
&\leq C\big(r,\mu_0,\mu_1,d,p,\|u\|_{L^p(B(0,1))}\big).
\end{aligned}
\end{equation*}
where $r\in(0,1/2)$, and we use the estimate $\eqref{f:2.21}$ in the second inequality. Thus,
it is fine to assume $\|u\|_{L^p(B(0,1))} = 1$, and there holds
\begin{equation*}
\|\nabla u\|_{L^\infty(B(0,1/4))} + [\nabla u]_{C^\infty(B(0,1/4))}
\leq C\big(\mu_0,\mu_1,d,p\big).
\end{equation*}
Then, let $\bar{u} = u/\|u\|_{L^p(B(0,1))}$, and it follows from the homogeneity assumption
$\eqref{a:1}$ that $\text{div}A(y,\nabla\bar{u})=0$ in $B(0,1)$. Therefore,
\begin{equation*}
  \|\nabla u\|_{L^\infty(B(0,1/4))} + [\nabla u]_{C^\infty(B(0,1/4))}
\lesssim \|u\|_{L^p(B(0,1))}.
\end{equation*}
This together with a covering argument leads to the stated result $\eqref{pri:2.9}$.
\end{proof}

\begin{remark}\label{remark:2.1}
\emph{In view of the estimate $\eqref{pri:2.1}$, one may conclude that $N(y,0) = 0$ for any $y\in\mathbb{R}^d$.}
\end{remark}

\begin{lemma}\label{lemma:2.5}
Suppose $A$ satisfies the assumptions $\eqref{a:1}$ and $\eqref{c:1}$.
Let $\widehat{A}$ be given in $\eqref{eq:1.1}$.
Then the effective operator $\mathcal{L}_0$
satisfies $\widehat{A}(0) = 0$, and
\begin{equation}\label{pri:2.3}
\left\{\begin{aligned}
&\big<\widehat{A}(\xi)-\widehat{A}(\xi^\prime),\xi-\xi^\prime\big>\geq C_0|\xi-\xi^\prime|^p,\\
& |\widehat{A}(\xi)-\widehat{A}(\xi^\prime)|\leq C_1|\xi-\xi^\prime|
\big(|\xi|+|\xi^\prime|\big)^{p-2}
\end{aligned}\right.
\end{equation}
for $2\leq p<\infty$, and
\begin{equation}\label{pri:2.18}
\left\{\begin{aligned}
&\big<\widehat{A}(\xi)-\widehat{A}(\xi^\prime),\xi-\xi^\prime\big>\geq C_0\big(|\xi|+|\xi^\prime|\big)^{p-2}|\xi-\xi^\prime|^2,\\
& |\widehat{A}(\xi)-\widehat{A}(\xi^\prime)|\leq C_2|\xi-\xi^\prime|^{\frac{p-1}{3-p}}
\big(|\xi|+|\xi^\prime|\big)^{\frac{(p-1)(2-p)}{3-p}}
\end{aligned}\right.
\end{equation}
for $1<p\leq 2$, in which $C_0,C_1,C_2$ depend on $\mu_0,\mu_1,p$ and $d$.
\end{lemma}

\begin{proof}
The proof may be found in \cite[Remark 1.3]{MD} or \cite[Proposition 2.7]{FM}.
\end{proof}

\begin{remark}\label{remark:2.2}
\emph{In order to obtain some higher regularity
from homogenized equations, we have to impose the following conditions
\begin{equation}\label{c:4}
 \big<\widehat{A}(\xi)-\widehat{A}(\xi^\prime),\xi-\xi^\prime\big>
\geq \tilde{\mu}_0 |\xi-\xi^\prime|^2\big(|\xi|+|\xi^\prime|\big)^{p-2}
\end{equation}
in the case of $2<p<\infty$ (since there merely holds the first line of
$\eqref{pri:2.3}$ in general),  and
\begin{equation}\label{c:5}
 |\widehat{A}(\xi)-\widehat{A}(\xi^\prime)|
\leq \tilde{\mu}_1 |\xi-\xi^\prime|\big(|\xi|+|\xi^\prime|\big)^{p-2}
\end{equation}
for the case $1<p<2$, where $\xi,\xi^\prime\in\mathbb{R}^d\setminus\{0\}$. Obviously, the assumptions $\eqref{c:4}$ and $\eqref{c:5}$
implies the first line of $\eqref{pri:2.3}$ and the second one of $\eqref{pri:2.18}$, respectively.}
\end{remark}

\begin{lemma}\label{lemma:2.4}
Suppose that $A$ satisfies $\eqref{a:1}$ and $\eqref{c:1}$.
Let $\varphi\in C_0^1(\Sigma_{2\varepsilon};\mathbb{R}^d)$, and $0<h<\varepsilon\ll 1$.
Then, there exists $\Phi(\cdot,\varphi)\in L^r(Y)$ with $1<r\leq \infty$,  satisfying
\begin{equation}\label{pri:2.5}
\Big(\dashint_{Y}|\Phi(\cdot,\varphi)|^r\Big)^{\frac{1}{r}} \lesssim
\left\{\begin{aligned}
& |\varphi|^{\frac{p-2}{p}} &~~~\text{if}~~& 2\leq p<\infty;\\
& |\varphi|^{\frac{2-p}{3-p}} &~~~\text{if}~~& 1< p\leq 2,
\end{aligned}\right.
\end{equation}
such that
\begin{equation}\label{eq:2.3}
D_i^{h}(\varepsilon N(y,\varphi)) =
\int_0^1 \partial_i N(y+t(h/\varepsilon)e_i,\bar{\varphi}) dt
+\varepsilon h^{\alpha-1}\Phi(y,\varphi)\big|D_i^{h}\varphi\big|^\alpha,
\end{equation}
where $\bar{\varphi} = \varphi(\cdot+he_i)$, and $\alpha$ is defined in
$\eqref{notation:1}$. Moreover, let $h=\varepsilon^\tau$ with
$1<\tau<1/(1-\alpha)$ and $p\not=2$, we obtain
\begin{equation}\label{pri:2.8}
\int_\Omega |D_i^{h}(\varepsilon N(x/\varepsilon,\varphi))|^p dx
\lesssim \int_\Omega|\varphi|^p dx + \varepsilon^{p-p(1-\alpha)\tau}
\left\{\begin{aligned}
&\int_\Omega |\varphi|^{p-2}|D_i^h\varphi|^{2}dx &~\text{if}~& 2<p<\infty;\\
&\int_\Omega |\varphi|^{\frac{p(2-p)}{3-p}}|D_i^h\varphi|^{\frac{p}{3-p}}dx &~\text{if}~& 1<p<2,
\end{aligned}\right.
\end{equation}
in which the up to constant depends on $\mu_0,\mu_1,p,d$.
\end{lemma}

\begin{proof}
According to the definition of the difference quotients, we have
\begin{equation*}
\begin{aligned}
D_i^{h}(\varepsilon N(y,\varphi))
&= \frac{\varepsilon}{h}\Big[N(y+(h/\varepsilon)e_i,\bar{\varphi}) - N(y,\bar{\varphi})\Big]
+ \frac{\varepsilon N(y,\varphi(x+he_i)) - \varepsilon N(y,\varphi(x))}{h} \\
&= \frac{\varepsilon}{h}\int_{0}^{1}\frac{d}{dt} \Big(N(y+t(h/\varepsilon)e_i,\bar{\varphi})\Big)
+\varepsilon h^{\alpha-1}\Phi(y,\varphi)|D_i^{h}\varphi|^\alpha ,
\end{aligned}
\end{equation*}
where we note that $\bar{\varphi}(x) = \varphi(x+he_i)$ and we omit the bar
since the scale of $h$ is even smaller than $\varepsilon$, and it will not essentially
change the location of $\varphi$ even measured in $\varepsilon$ scales.
By letting
\begin{equation*}
  \Phi(y,\varphi)
  = \frac{N(y,\varphi(x+he_i)) - N(y,\varphi(x))}{
|\varphi(x+he_i)-\varphi(x)|^\alpha},
\end{equation*}
it is not hard to observe the equality $\eqref{eq:2.3}$. Then we proceed to verify
the estimate $\eqref{pri:2.5}$. To do so, we denoted by $\xi = \varphi(x)$ and $\xi^\prime = \varphi(x+h e_i)$, and it follows from Poincar\'e's inequality that
\begin{equation*}
\begin{aligned}
\Big(\dashint_{Y}|N(\cdot,\xi)-N(\cdot,\xi^\prime)|^p\Big)^{\frac{1}{p}}
&\leq \Big(\dashint_{Y}|\nabla N(\cdot,\xi)-\nabla N(\cdot,\xi^\prime)|^p\Big)^{\frac{1}{p}}\\
&\lesssim \Big(\dashint_{Y}| P(\cdot,\xi)- P(\cdot,\xi^\prime)|^p\Big)^{\frac{1}{p}} + |\xi-\xi^\prime|
\end{aligned}
\end{equation*}
This together with the estimate $\eqref{pri:2.2}$ leads to
\begin{equation*}
\Big(\dashint_{Y}|N(\cdot,\xi)-N(\cdot,\xi^\prime)|^p\Big)^{\frac{1}{p}}
\lesssim \left\{\begin{aligned}
&|\xi-\xi^\prime|^{\frac{2}{p}}|\xi|^{\frac{p-2}{p}},
&~\text{if}~& 2\leq p <\infty;\\
&|\xi-\xi^\prime|^{\frac{1}{3-p}}|\xi|^{\frac{2-p}{3-p}}, &~\text{if}~& 1<p\leq 2,
\end{aligned}\right.
\end{equation*}
where we use the fact that $|\xi^\prime|\leq 2|\xi|$ whenever $h$ is sufficient small. This
proved the case $1<r<\infty$ of $\eqref{pri:2.5}$. In the following, we proceed to show the case
$r=\infty$. Let $u(y,\xi) = N(y,\xi)+y\cdot\xi$ and
$\tilde{u}(y,\xi) = u(y,\xi)+\tilde{M}$, in which $\tilde{M}$ is such that
$\tilde{u}$ is positive in $B(0,1)$. According to the estimate $\eqref{pri:2.19}$, it is not
hard to see the existence of $\tilde{M}$ whenever $|\xi|$ is bounded, and such the boundedness
is in fact given by $\|\varphi\|_{L^\infty(\Omega)}$. Note that $\tilde{u}$ still satisfies
the same type equation
\begin{equation*}
 \text{div} A(y,\nabla u(y,\xi)) = 0 \quad \text{in}\quad B(0,1).
\end{equation*}
Thus, it follows from the local boundedness estimate and the weak Harnack inequality
(see for example \cite[Corollary 3.10,Theorem 3.13]{MZ}) that
\begin{equation}\label{f:3.18}
 \sup_{y\in B(0,2/3)}\tilde{u}(y,\xi) \lesssim \dashint_{B(0,4/5)} \tilde{u}(\cdot,\xi)
 \qquad\text{and}\qquad
 \inf_{y\in B(0,2/3)}\tilde{u}(y,\xi)
 \gtrsim \dashint_{B(0,4/5)}\tilde{u}(\cdot,\xi),
\end{equation}
respectively. Then for any $y\in Y$ such that $\tilde{u}(y,\xi)-\tilde{u}(y,\xi^\prime)>0$,
there holds
\begin{equation*}
\begin{aligned}
\tilde{u}(y,\xi) - \tilde{u}(y,\xi^\prime)
\leq \sup_{y\in Y}\tilde{u}(y,\xi)-\inf_{y\in Y}\tilde{u}(y,\xi^\prime)
\lesssim \dashint_{B(0,4/5)} \big|\tilde{u}(\cdot,\xi)-\tilde{u}(\cdot,\xi^\prime)\big|,
\end{aligned}
\end{equation*}
where we use the estimates $\eqref{f:3.18}$ in the second inequality. Similarly, for any
$y\in Y$ such that $\tilde{u}(y,\xi^\prime)- \tilde{u}(y,\xi) >0$,
one may derive
\begin{equation*}
\begin{aligned}
\tilde{u}(y,\xi^\prime) - \tilde{u}(y,\xi)
\leq \sup_{y\in Y}\tilde{u}(y,\xi^\prime)-\inf_{y\in Y}\tilde{u}(y,\xi)
\lesssim \dashint_{B(0,4/5)} \big|\tilde{u}(\cdot,\xi)-\tilde{u}(\cdot,\xi^\prime)\big|.
\end{aligned}
\end{equation*}
Then we actually obtain
\begin{equation*}
\begin{aligned}
|u(y,\xi) - u(y,\xi^\prime)|
&\lesssim\dashint_{B(0,4/5)} \big|u(\cdot,\xi)-u(\cdot,\xi^\prime)\big| \\
&\lesssim \dashint_{Y}|N(\cdot,\xi)-N(\cdot,\xi^\prime)| + |\xi-\xi^\prime|
\end{aligned}
\end{equation*}
for any $y\in Y$, where we use a covering argument in the second inequality. This further implies
\begin{equation*}
\|N(\cdot,\xi)-N(\cdot,\xi^\prime)\|_{L^\infty(Y)}
\lesssim \left\{\begin{aligned}
&|\xi-\xi^\prime|^{\frac{2}{p}}|\xi|^{\frac{p-2}{p}},
&~\text{if}~& 2\leq p <\infty;\\
&|\xi-\xi^\prime|^{\frac{1}{3-p}}|\xi|^{\frac{2-p}{3-p}}, &~\text{if}~& 1<p\leq 2,
\end{aligned}\right.
\end{equation*}
whenever $h$ is sufficiently small, which will show $\Phi(\cdot,\varphi)\in L^\infty(Y)$ with
\begin{equation}\label{f:2.19}
 \|\Phi(\cdot,\varphi(x))\|_{L^\infty(Y)}
 \lesssim \left\{\begin{aligned}
&|\varphi(x)|^{\frac{p-2}{p}},
&~\text{if}~& 2\leq p <\infty;\\
&|\varphi(x)|^{\frac{2-p}{3-p}}, &~\text{if}~& 1<p\leq 2,
\end{aligned}\right.
\end{equation}
for any $x\in\Omega$.

Recall $Y_{\varepsilon}^i = \varepsilon(i + Y)$ for $i\in \mathbb{Z}^d$, and we collect a family of $\{Y_\varepsilon^i\}$, whose index set is denoted by $I_\varepsilon$, such that
$\text{supp}\varphi\subset \cup_{i\in I_\varepsilon}Y_\varepsilon^i\subset\Omega$ and
$Y_\varepsilon^i\cap Y_\varepsilon^j = \emptyset$ if $i\not=j$. Thus the left-hand side of
$\eqref{pri:2.8}$ equals
\begin{equation*}
\begin{aligned}
 \sum_{k\in I_\varepsilon}\int_{Y_\varepsilon^k}
 \big|D_i^h(\varepsilon N(x/\varepsilon,\varphi))\big|^pdx
& \leq \sum_{k\in I_\varepsilon}\int_{Y_\varepsilon^k}
\Big|\int_0^1 \nabla_{y_i} N(y+t(h/\varepsilon)e_i,\varphi) dt\Big|^p dx\\
&+\varepsilon^{p-p(1-\alpha)\tau} \int_{\Omega}
\big|\Phi(y,\varphi)|D_i^{h}\varphi|^\alpha\big|^p dx.
\end{aligned}
\end{equation*}
The first term in the right-hand side above is controlled by
\begin{equation*}
\begin{aligned}
\sum_{k\in I_\varepsilon}\int_{Y_\varepsilon^k}\int_0^1 \big|\partial_i N(y+t(h/\varepsilon)e_i,\varphi)\big|^p dt dx
&= \sum_{k\in I_\varepsilon}\int_0^1 \int_{Y_\varepsilon^k}\big|\nabla_{y_i} N(y+t(h/\varepsilon)e_i,\varphi)\big|^p  dxdt \\
&\lesssim \sum_{k\in I_\varepsilon} |Y_\varepsilon^k|\int_0^1 \dashint_{Y_\varepsilon^k}
\big|\nabla_y N(y+t(h/\varepsilon)e_i,\varphi)\big|^p  dxdt \\
&\lesssim \sum_{k\in I_\varepsilon} |Y_\varepsilon^k|\int_0^1
\dashint_{Y+k+t(h/\varepsilon)e_i}\big|\nabla_z N(z,\tilde{\varphi})\big|^p  dz dt\\
&\lesssim \sum_{k\in I_\varepsilon} |Y_\varepsilon^k|
\dashint_{Y}\big|\nabla_z N(z,\tilde{\tilde{\varphi}})\big|^p  dz,
\end{aligned}
\end{equation*}
where $\tilde{\varphi}(z) = \varphi(\varepsilon z - the_i)$ with
$z\in (k+ the_i) + Y$, and
$\tilde{\tilde{\varphi}}(z) = \varphi(\varepsilon z
- the_i - k - t(h/\varepsilon)e_i)$ with  $z\in Y$.
Let $\hat{x} = \varepsilon z
- the_i - k - t(h/\varepsilon)e_i$, and one may conclude that $\hat{x}\in Y_{2\varepsilon}^k$.
Note that the periodicity of the corrector is employed in the last inequality. Therefore, it
follows from the estimate $\eqref{pri:2.1}$ that
\begin{equation*}
\sum_{k\in I_\varepsilon} |Y_\varepsilon^k|
\dashint_{Y}\big|\nabla N(z,\tilde{\tilde{\varphi}})\big|^p  dz
\lesssim \sum_{k\in I_\varepsilon} |Y_\varepsilon^k| \inf_{\hat{x}\in Y_{2\varepsilon}^k}
|\varphi(\hat{x})|^p
\lesssim \sum_{k\in I_\varepsilon} \int_{Y_\varepsilon^k}|\varphi|^p dx
\lesssim \int_{\Omega}|\varphi|^p dx.
\end{equation*}
Here we use Chebyshev's inequality in the second inequality. Thus, we arrive at
\begin{equation}\label{f:2.3}
\sum_{k\in I_\varepsilon}
\int_{Y_\varepsilon^k}\int_0^1 \big|\nabla_{y_i} N(y+t(h/\varepsilon)e_i,\varphi)\big|^p dt dx
\lesssim \int_{\Omega}|\varphi|^p dx.
\end{equation}

In terms of the estimate $\eqref{f:2.3}$ and the periodicity of $\Phi$ with respective to
the first variable, it is not hard to see that
\begin{equation*}
  |\Phi(x/\varepsilon,\varphi(x))|\lesssim |\varphi(x)|^\nu
  \qquad\forall x\in\Omega,
\end{equation*}
where $\nu = (p-2)/p$ if $2<p<\infty$, and
$\nu= (2-p)/(p-3)$ if $1<p<2$. Thus, we have
\begin{equation}\label{f:2.4}
\begin{aligned}
\int_{\Omega}
\big|\Phi(y,\varphi)|D_i^{h}\varphi|^\alpha\big|^p dx
\lesssim \int_{\Omega}|\varphi|^{\nu p}|(D_i^h\varphi)|^{\alpha p} dx.
\end{aligned}
\end{equation}

Finally, the estimates $\eqref{f:2.3}$ and
$\eqref{f:2.4}$ imply the desired estimate $\eqref{pri:2.8}$, and we have completed the proof.
\end{proof}

\begin{lemma}\label{lemma:2.9}
Let $\varphi\in C_0^1(\Sigma_{2\varepsilon};\mathbb{R}^d)$.
Let $\theta\in(0,1)$ be given in Lemma $\ref{lemma:2.6}$, and
assume the same conditions and notation as in Lemma $\ref{lemma:2.4}$.
Then, under the additional assumption $\eqref{a:4}$, there holds
the following estimate
\begin{equation}\label{pri:2.15}
\begin{aligned}
&\int_\Omega |D_i^{h}(\varepsilon N(x/\varepsilon,\varphi))
-\partial_i N(x/\varepsilon,\varphi)|^p dx\\
&\lesssim \varepsilon^{p\theta(\tau-1)}\int_{\Omega}|\varphi|^p dx
+ \varepsilon^{p-p(1-\alpha)\tau}
\left\{\begin{aligned}
&\int_\Omega |\varphi|^{p-2}|D_i^h\varphi|^{2}dx &~\text{if}~& 2<p<\infty;\\
&\int_\Omega |\varphi|^{\frac{p(2-p)}{3-p}}|D_i^h\varphi|^{\frac{p}{3-p}}dx &~\text{if}~& 1<p<2,
\end{aligned}\right.
\end{aligned}
\end{equation}
where the up to constant depends on $\mu_0,\mu_1,\mu_2,p,d$.
\end{lemma}

\begin{proof}
According to the expression $\eqref{eq:2.3}$ and the estimate $\eqref{pri:2.8}$, it
suffices to estimate the following quantity
\begin{equation}
\begin{aligned}
&\quad\int_{\Omega}\int_0^1
\big|\partial_i N(y
+t(h/\varepsilon)e_j,\bar{\varphi})
- \partial_i N(y,\varphi)\big|^p dt dx \\
&\lesssim \underbrace{\int_{\Omega}\int_0^1
\big|\partial_i N(y
+t(h/\varepsilon)e_j,\bar{\varphi})
- \partial_i N(y,\bar{\varphi})\big|^p dt dx }_{I_1}
+ \underbrace{\int_{\Omega}
\big|\partial_i N(y,\bar{\varphi}) - \partial_i N(y,\varphi)\big|^p dx}_{I_2}.
\end{aligned}
\end{equation}
Then, it follows from the estimate $\eqref{pri:2.9}$ that
\begin{equation}\label{f:2.20}
\begin{aligned}
I_1
\leq (h/\varepsilon)^{p\theta}
\int_{\Omega}\big[\nabla N(\cdot,\bar{\varphi})\big]_{C^{0,\theta}(\mathbb{R}^d)}^p dx
\int_0^1 t^{p\theta} dt
\lesssim \varepsilon^{p\theta(\tau-1)}\int_{\Omega}|\varphi|^pdx,
\end{aligned}
\end{equation}
where we note that the compact support of $\bar{\varphi}$ is still included in $\Omega$ since
$0<h<\varepsilon$. Moreover, we claim that $I_1$ will be the main term, since we have the following
computations. In view of $\eqref{pri:2.2}$,
\begin{equation}\label{f:2.22}
\begin{aligned}
I_2 &\leq \sum_{k\in I_\varepsilon}|Y_{\varepsilon}^k|
\dashint_{Y_\varepsilon^k}|\nabla_{y} N(y,\bar{\varphi}) - \nabla_{y} N(y,\varphi)|^pdx \\
&\lesssim h^{p\alpha}
\sum_{k\in I_\varepsilon}|Y_{\varepsilon}^k||\nabla \varphi(\hat{x})|^{p\alpha}
|\varphi(\hat{x})|^{p\nu}
\lesssim \varepsilon^{p\tau\alpha}
\left\{\begin{aligned}
& \int_{\Omega}|\nabla \varphi|^{2}|\varphi|^{p-2}dx,
&~\text{if}~& 2\leq p <\infty;\\
&\int_{\Omega}|\nabla\varphi|^{\frac{p}{3-p}}|\varphi|^{\frac{p(2-p)}{3-p}}dx, &~\text{if}~& 1<p\leq 2,
\end{aligned}\right.
\end{aligned}
\end{equation}
where the index $\nu$ is defined in $\eqref{f:2.4}$,
and $\hat{x}\in Y_{\varepsilon}^k$ is due to mean value theorem, and the last step follows from
the definition of Riemann integral
for $\varphi\in C_0^1(\Omega)$.
As in the previous lemma,
we set $h=\varepsilon^{\tau}$ with $1<\tau<1/(1-\alpha)$, and it is not hard to find that
the term $I_2$ will be eaten by the second term in the right-hand side of $\eqref{pri:2.8}$ in terms
of the order of $\varepsilon$. That is the reason why we say $I_1$ is main term here,
and the proof is complete.
\end{proof}

\begin{remark}
\emph{In terms of $p=2$,
on account of Lemma $\ref{lemma:2.4}$ and the above calculations,  there holds
\begin{equation*}
\begin{aligned}
\int_\Omega |D_i^{h}(\varepsilon N(x/\varepsilon,\varphi))
-\partial_i N(x/\varepsilon,\varphi)|^2 dx
\lesssim  \varepsilon^{2}
\int_\Omega |\nabla_i\varphi|^{2}dx
\end{aligned}
\end{equation*}
for any $0<h<\varepsilon$. If we further assume
$\varphi = \psi_{4\varepsilon}\nabla u_0$, then
layer and co-layer type estimates (see Lemma $\ref{lemma:3.7}$) leads to
\begin{equation}\label{pri:2.20}
\begin{aligned}
\Big(\int_\Omega |D_i^{h}(\varepsilon N(x/\varepsilon,\varphi))
-\partial_i N(x/\varepsilon,\varphi)|^2 dx\Big)^{1/2}
\lesssim  \varepsilon^{\beta}
\Big(\int_\Omega |\nabla u_0|^{q}dx\Big)^{1/q}
\end{aligned}
\end{equation}
where $q=2(1+\delta)$ and $\delta\in (0,1)$ is the so-called Meyer's index.}
\end{remark}

\subsection{Flux corrector and its properties}
\begin{lemma}[Flux correctors]\label{lemma:2.3}
Suppose $A$ satisfies $\eqref{a:1}$ and $\eqref{c:1}$.
Let $b(y,\xi) = A(y,\xi+\nabla N(y,\xi)) - \widehat{A}(\xi)$, where $y\in Y$
and $\xi\in \mathbb{R}^d$.
Then we have two properties: \emph{(i)} $\dashint_Y b(\cdot,\xi) = 0$;
\emph{(ii)} $\emph{div} b(\cdot,\xi) = 0$ in $Y$. Moreover, there
exists $\hat{p}\in(1,2)$, depending only on $p,d$, such that the so-called
flux corrector $E_{ji}(\cdot,\xi)\in W^{1,\hat{p}}_{per}(Y)\cap L^2_{per}(Y)$, and
\begin{equation}\label{eq:2.4}
 b_i(y,\xi) = \frac{\partial}{\partial y_j}\big\{E_{ji}(y,\xi)\big\}
 \qquad \text{and} \quad E_{ji} = -E_{ij},
\end{equation}
and  there holds
\begin{equation}\label{pri:2.10}
\Big(\dashint_Y |E_{ji}(\cdot,\xi)-E_{ji}(\cdot,\xi^\prime)|^2\Big)^{\frac{1}{2}}
\lesssim\left\{\begin{aligned}
&|\xi-\xi^\prime|^{\frac{2}{p}}\big(|\xi|+|\xi^\prime|\big)^{\frac{(p-2)(1+p)}{p}}
&~~~\text{if}~~& 2< p<\infty;\\
&|\xi-\xi^\prime|^{\frac{p-1}{3-p}}
\big(|\xi|+|\xi^\prime|\big)^{\frac{(2-p)(p-1)}{3-p}} &~~~\text{if}~~& 1< p\leq2.
\end{aligned}\right.
\end{equation}
Moreover, if we additionally assume the smoothness condition $\eqref{a:4}$, then we have
\begin{equation}\label{pri:2.14}
\|\nabla E_{ji}(\cdot,\xi)\|_{C^{0,\rho}(Y)}
\leq C |\xi|^{p-1},
\end{equation}
where $\rho = \min\{\theta,(p-1)\theta\}$, and $C$ depends only on $\mu_0,\mu_2,\mu_3,p$ and $d$.
\end{lemma}

\begin{proof}
The proof is quite similar to the linear case (see for example \cite{S4,ZVVPSE1}).
It is clear to see that (i) and (ii) follow from the formula $\eqref{eq:1.1}$
and the equation $\eqref{pde:1.2}$, respectively. To obtain more results, we first need to
verify that $b_i(\cdot,\xi)\in L^{\hat{q}}(Y)$
with the index $\hat{q}$ being as follows:
\begin{equation}
\hat{q} =^{(d=1,2)} p/(p-1),\qquad
\hat{q}=^{(d\geq 3)}\left\{\begin{aligned}
& 2d/(d+2) &~\text{if}~& 1<p\leq 2d/(2d-2);\\
& (1+\delta)p/(p-1) &~\text{if}~& 2d/(2d-2)<p<\infty.
\end{aligned}\right.
\end{equation}
Note that $\delta\in(0,1)$ is a small number given in Remark $\ref{remark:4.1}$,
and we just employ Meyer's estimate $\eqref{pri:4.3}$ (note that this estimate is also true for
$\varepsilon=1$.) in the case of $2d/(2d-2)<p<\infty$ with $d\geq 3$. In fact, a routine computation leads to
\begin{equation}\label{f:2.13}
\Big(\dashint_{Y} |b_i(\cdot,\xi)|^{\hat{q}}\Big)^{1/\hat{q}}
\lesssim |\xi|^{p-1},
\end{equation}
whose proof is similar to that given for $\eqref{f:2.14}$.
Then, it is not hard to see that $\hat{q}\geq q$ in each case,
where $q$ is shown in $\eqref{f:2.11}$. This together with the property (i) guarantees
the existence of the weak solution of
\begin{equation*}
\left\{\begin{aligned}
&\Delta f_i(\cdot,\xi) = b_i(\cdot,\xi) \quad\text{in}~ Y,\\
& f_i(\cdot,\xi)\in H^{1}_{per}(Y), \quad \dashint_{Y}f_i(\cdot,\xi) = 0
\end{aligned}\right.
\end{equation*}
for each $i$ and $\xi\in\mathbb{R}^d$. Also, we have the following energy estimate
\begin{equation}\label{f:2.12}
\begin{aligned}
\|E_{ji}\|_{L^2(Y)}\lesssim \|\nabla f_i\|_{L^2(Y)}
\lesssim \|b_i\|_{L^{\hat{q}}(Y)}\lesssim |\xi|^{p-1}.
\end{aligned}
\end{equation}
Let $E_{ji}(y,\xi)= \frac{\partial}{\partial y_j}\big\{f_{i}(y,\xi)\big\}
-\frac{\partial}{\partial y_i}\big\{f_{j}(y,\xi)\big\}$. Obviously, $E_{ji} = - E_{ij}$, and
one may derive the first expression in $\eqref{eq:2.4}$ from the fact $(\text{ii})$
in a weak sense. In fact, one may further show that $E_{ji}(\cdot,\xi)\in W^{1,\hat{p}}_{per}(Y)$ with
some $\hat{p}\in(0,2)$ given by
\begin{equation*}
  \hat{p} =\left\{ \begin{aligned}
  &\tilde{q}\in(1,2), &&\quad 1<p\leq 2, \text{~and~} d=1,2;\quad& \\
  &\hat{q}, &&\quad \text{other~cases}.\quad&
  \end{aligned}\right.
\end{equation*}
Let $Y\subset B(0,2/3)$, and
it follows from the Calder\'on-Zygmund theorem coupled with a localization
argument that,
\begin{equation}\label{f:2.10}
\begin{aligned}
\|\nabla^2 f(\cdot,\xi)\|_{L^{\hat{p}}(Y)}
&\lesssim \|\nabla f(\cdot,\xi)\|_{L^{\hat{p}}(B(0,1))}
+ \|b(\cdot,\xi)\|_{L^{\hat{p}}(B(0,1))}\\
&\lesssim \|\nabla f(\cdot,\xi)\|_{L^{2}(Y)}
+ \|b(\cdot,\xi)\|_{L^{\hat{q}}(Y)}\\
&\lesssim \|b(\cdot,\xi)\|_{L^{\hat{q}}(B(Y))},
\end{aligned}
\end{equation}
in which we use the periodicity of $f(\cdot,\xi)$ and a covering argument in the second inequality,
and the estimate $\eqref{f:2.12}$ in the last one.
Note that $1<\hat{p}\leq \hat{q}$ in each case. Collecting the estimates $\eqref{f:2.13}$ and
$\eqref{f:2.10}$, we obtain the estimate
\begin{equation*}
  \Big(\dashint_{Y}|\nabla E(\cdot,\xi)|^{\hat{p}}\Big)^{1/\hat{p}}
  \lesssim |\xi|^{p-1}.
\end{equation*}

Then, we turn to show the estimate $\eqref{pri:2.10}$.
For any $\xi,\xi^\prime\in\mathbb{R}^d$, note that
\begin{equation}\label{f:2.14}
\begin{aligned}
\Big(\dashint_{Y} |E_{ji}(y,\xi)- E_{ji}(y,\xi^\prime)|^2 dy\Big)^{\frac{1}{2}}
&\lesssim \Big(\dashint_{Y} \big|\nabla \big(f_{i}(y,\xi)-f_{i}(y,\xi^\prime)\big)\big|^2
dy\Big)^{\frac{1}{2}}\\
&\lesssim \dashint_{Y} \Big(\big|b_{i}(y,\xi)-b_{i}(y,\xi^\prime)\big|^q dy\Big)^{\frac{1}{q}},
\end{aligned}
\end{equation}
where we employ an $H^{1}$ estimate in the second step, and
\begin{equation}\label{f:2.11}
q = \left\{\begin{aligned}
&2d/(d+2) &\quad\text{if}\quad& d\geq 3;\\
&1<q<2 &\quad\text{if}\quad& d= 2;\\
&q=1&\quad\text{if}\quad& d=1.
\end{aligned}\right.
\end{equation}
Then, we carry out a computation in the case of $p>2$,
which in fact involves higher regularity of $P(\cdot,\xi)$. By definition,
the right-hand side of $\eqref{f:2.14}$
is controlled by
\begin{equation*}
\begin{aligned}
&\quad\Big(\dashint_{Y}
|A(\cdot,P(\cdot,\xi))-A(\cdot,P(\cdot,\xi^\prime))|^q\Big)^{1/q}
+ \big|\widehat{A}(\xi)-\widehat{A}(\xi^\prime)\big| \\
&\lesssim \underbrace{\Big(\dashint_{Y}
|P(\cdot,\xi)-P(\cdot,\xi^\prime)|^{q}
\big(|P(\cdot,\xi)|+|P(\cdot,\xi^\prime)|\big)^{q(p-2)}\Big)^{1/q}}_{I}
+ |\xi-\xi^\prime|\big(|\xi|+|\xi^\prime|\big)^{p-2}.
\end{aligned}
\end{equation*}
where we use the assumption $\eqref{c:1}$ and the estimate $\eqref{pri:2.3}$
in the inequality. Then, it follows from H\"older's inequality that
\begin{equation*}
\begin{aligned}
I &\leq  \Big(\dashint_{Y}
|P(\cdot,\xi)-P(\cdot,\xi^\prime)|^{p}\Big)^{\frac{1}{p}}
\Big(\|P(\cdot,\xi)\|_{L^\infty(Y)}+\|P(\cdot,\xi^\prime)\|_{L^\infty(Y)}\Big)^{p-2} \\
&\lesssim |\xi-\xi^\prime|^{\frac{2}{p}}\big(|\xi|+|\xi^\prime|\big)^{\frac{(p-2)(1+p)}{p}},
\end{aligned}
\end{equation*}
where we use the estimates $\eqref{pri:2.2}$, $\eqref{pri:2.9}$ in the second inequality.
This is the first line of $\eqref{pri:2.10}$ ,
and we continue to show the case $1<p\leq 2$. On account of $\eqref{pri:2.18}$,
it suffices to compute the following quantity, and by $\eqref{c:1}$ we have
\begin{equation*}
\begin{aligned}
\Big(\dashint_Y|A(\cdot,P(\cdot,\xi))-A(\cdot,P(\cdot,\xi^\prime))|^q\Big)^{1/q}
&\lesssim \Big(\dashint_Y|P(\cdot,\xi)-P(\cdot,\xi^\prime)|^{q(p-1)}\Big)^{1/q}\\
&\lesssim \Big(\dashint_Y|P(\cdot,\xi)-P(\cdot,\xi^\prime)|^{p}\Big)^{\frac{p-1}{p}}
\lesssim |\xi-\xi^\prime|^{\frac{p-1}{3-p}}
\big(|\xi|+|\xi^\prime|\big)^{\frac{(2-p)(p-1)}{3-p}},
\end{aligned}
\end{equation*}
where we use the fact that $\frac{p}{p-1}\geq 2$ in the second inequality (which gives $\frac{p}{q(p-1)}\geq 1$), and
$\eqref{pri:2.2}$ in the last one. This ends the proof of the stated estimate $\eqref{pri:2.10}$.

We now proceed to prove the estimate $\eqref{pri:2.14}$. On account of
Lemma $\ref{lemma:2.5}$, the assumptions $\eqref{a:1},\eqref{c:1}$, one may have
\begin{equation*}
 \|b(\cdot,\xi)\|_{L^\infty(\mathbb{R}^d)}
 \leq \|P(\cdot,\xi)\|_{L^\infty(\mathbb{R}^d)}^{p-1} + |\widehat{A}(\xi)|
 \lesssim |\xi|^{p-1}.
\end{equation*}
where we also use the estimate $\eqref{pri:2.9}$ in the second inequality.
For any
$y_1,y_2\in\mathbb{R}^d$ with $|y_1-y_2|\ll 1$, we first have
\begin{equation*}
\begin{aligned}
|b(y_1,\xi)-b(y_2,\xi)|
\lesssim\underbrace{\big|A(y_1,P(y_1,\xi))-A(y_2,P(y_1,\xi))\big|}_{I_1}
+\underbrace{\big|A(y_2,P(y_1,\xi))-A(y_2,P(y_2,\xi))\big|}_{I_2}.
\end{aligned}
\end{equation*}
Then, it follows from the assumption $\eqref{a:4}$ and the estimate $\eqref{pri:2.9}$ that
\begin{equation*}
\begin{aligned}
I_1 \leq \mu_3|y_1-y_2| \|P(\cdot,\xi)\|_{L^\infty(\mathbb{R}^d)}^{p-1}
\lesssim |\xi|^{p-1}|y_1-y_2|.
\end{aligned}
\end{equation*}
Similarly, by the assumption $\eqref{c:1}$, and the estimate $\eqref{pri:2.9}$ we derive that
\begin{equation*}
\begin{aligned}
I_2 \leq\mu_1\left\{\begin{aligned}
&\|P(\cdot,\xi)\|_{L^\infty(\mathbb{R}^d)}^{p-2}
[P(\cdot,\xi)]_{C^{0,\theta}(\mathbb{R}^d)}|y_1-y_2|^{\theta}\\
&[P(\cdot,\xi)]_{C^{0,\theta}(\mathbb{R}^d)}^{p-1}|y_1-y_2|^{(p-1)\theta}
\end{aligned}\right.
\lesssim\left\{\begin{aligned}
&|\xi|^{p-1}|y_1-y_2|^{\theta}
&\quad&2\leq p<\infty;\\
&|\xi|^{p-1}|y_1-y_2|^{(p-1)\theta}
&\quad&1<p<2.
\end{aligned}\right.
\end{aligned}
\end{equation*}
Hence, by setting $\rho=\min\{\theta,(p-1)\theta\}$, combining the above estimates leads to
\begin{equation}\label{f:2.16}
\big[b(\cdot,\xi)\big]_{C^{0,\rho}(\mathbb{R}^d)}
+ \|b(\cdot,\xi)\|_{L^{\infty}(\mathbb{R}^d)}
\lesssim |\xi|^{p-1}.
\end{equation}
This together with interior Schauder estimate (see for example \cite[Theorem 5.19]{MGLM}) leads to
\begin{equation*}
\big[\nabla^2 f_{i}\big]_{C^{0,\rho}(Y)}
\lesssim \|\nabla f_{i}\|_{L^2(B(0,2/3))}+\|b_{i}\|_{C^{0,\rho}(B(0,2/3))}
\lesssim |\xi|^{p-1},
\end{equation*}
in which we use the estimates $\eqref{f:2.12}$ and $\eqref{f:2.16}$ in the second inequality.
This implies the desired estimate $\eqref{pri:2.14}$, and
we have completed the whole proof.
\end{proof}

\begin{remark}\label{remark:2.4}
\emph{In fact, the estimate $\eqref{pri:2.10}$ is still true if the integrand
$|E_{ji}(y,\xi)-E_{ji}(y,\xi^\prime)|$ is replaced by
$|\nabla_yE_{ji}(y,\xi)-\nabla_yE_{ji}(y,\xi^\prime)|$. Since $E_{ji}(y,\xi)$ satisfies
\begin{equation}\label{eq:2.7}
\Delta E_{ji}(\cdot,\xi) = \big(\nabla \times b(\cdot,\xi)\big)_{ji} =:
\nabla_{j}b_i(\cdot,\xi) - \nabla_{i}b_j(\cdot,\xi)
\end{equation}
the result follows from the linearity and Caccioppoli's inequality, i.e.,
\begin{equation*}
\begin{aligned}
\Big(\dashint_{Y} |\nabla E_{ji}(\cdot,\xi)- \nabla E_{ji}(\cdot,\xi^\prime)|^2 \Big)^{1/2}
&\lesssim \Big(\dashint_{2Y} |E_{ji}(\cdot,\xi)- E_{ji}(\cdot,\xi^\prime)|^2 \Big)^{1/2}\\
&+  \Big(\dashint_{2Y}\big|b_{i}(\cdot,\xi)-b_{i}(\cdot,\xi^\prime)\big|^2 \Big)^{1/2}.
\end{aligned}
\end{equation*}
Thus, in view of $\eqref{pri:2.10}$, it suffices to estimate the last term above, while
it admits the same computations as in the proof of $\eqref{f:2.14}$. We do not repeat here.
Finally, we mention that the equation $\eqref{eq:2.7}$ may be referred to as the definition of
the flux corrector (see \cite[Lemma 1]{GNF1}).}
\end{remark}

\begin{lemma}\label{lemma:2.8}
Let $\varphi\in C_0^1(\Omega;\mathbb{R}^d)$.
Assume the same conditions as in Lemma $\ref{lemma:2.3}$.
Let $E_{ji}(\cdot,\xi)\in W^{1,\hat{p}}_{per}(Y)\cap L^2_{per}(Y)$ with $\hat{p}\in(1,2)$ be the flux corrector,
then there exists $\Psi_{ji}(\cdot,\varphi)\in L^2(Y)$ satisfying
\begin{equation}\label{pri:2.11}
\Big(\dashint_{Y}|\Psi_{ji}(\cdot,\varphi)|^2\Big)^{\frac{1}{2}} \lesssim
\left\{\begin{aligned}
& |\varphi|^{\frac{(p-2)(1+p)}{p}} &~~~\text{if}~~& 2\leq p<\infty;\\
& |\varphi|^{\frac{(2-p)(p-1)}{3-p}} &~~~\text{if}~~& 1< p \leq 2,
\end{aligned}\right.
\end{equation}
such that
\begin{itemize}
  \item in the case of $p=2$, we have
  \begin{equation}\label{eq:2.5}
    \varepsilon\frac{\partial}{\partial x_j}\big[E_{ji}(x/\varepsilon,\varphi)\big]
    = \frac{\partial}{\partial y_j}\big[E_{ji}(x/\varepsilon,\varphi)\big]
    + \varepsilon \frac{\partial}{\partial\xi_k}
    \big[E_{ji}(\cdot,\varphi)\big]\frac{\partial \varphi_k}{\partial x_j}
  \end{equation}
  with the estimate
  \begin{equation}\label{pri:2.12}
     \Big(\dashint_{Y} |\nabla_{\xi} E_{ji}(\cdot,\varphi)|^2\Big)^{1/2}
     \leq C(\mu_0,\mu_1,d);
  \end{equation}
  \item in the case of $p\in(1,2)\cup(2,\infty)$, there holds
  \begin{equation}\label{eq:2.6}
D_j^{h}(\varepsilon E_{ji}(y,\varphi)) =
\int_0^1 \partial_j E_{ji}(y+t(h/\varepsilon)e_j,\bar{\varphi}) dt
+\varepsilon h^{\gamma-1}\Psi_{ji}(y,\varphi)\big|D_j^{h}\varphi\big|^\gamma,
\end{equation}
where $\bar{\varphi} = \varphi(\cdot+he_j)$, and $\gamma$ is defined in $\eqref{notation:1}$.
\end{itemize}
Moreover, let $h=\varepsilon^\tau$ with
$1<\tau<1/(1-\gamma)$, and $1<r\leq p/(p-1)$, we obtain
\begin{equation}\label{pri:2.13}
\begin{aligned}
&\int_\Omega
|D_j^{h}(\varepsilon E_{ji}(x/\varepsilon,\varphi))
- \partial_j E_{ji}(x/\varepsilon,\varphi)|^{r} dx
\\
& \lesssim \varepsilon^{r\rho(\tau-1)}\int_\Omega |\varphi|^{r(p-1)} dx
+ \varepsilon^{r-r(1-\gamma)\tau}
\left\{\begin{aligned}
&\int_\Omega |\varphi|^{\frac{r(p-2)(1+p)}{p}}|D_j^h\varphi|^{\frac{2r}{p}}dx &~\text{if}~& 2<p<\infty;\\
&\int_\Omega |\varphi|^{\frac{r(2-p)(p-1)}{3-p}}|D_j^h\varphi|^{\frac{(p-1)r}{3-p}}dx &~\text{if}~& 1<p<2,
\end{aligned}\right.
\end{aligned}
\end{equation}
in which $\rho=\min\{\theta,\theta(p-1)\}$, and the up to constant depends on $\mu_0,\mu_1,\mu_2,p$ and $d$.
\end{lemma}

\begin{remark}
\emph{Although one can not
expect that the estimate
$\eqref{pri:2.11}$ is true in a $L^\infty$-norm,
the quantities
$\Psi(\cdot/\varepsilon,\varphi(\cdot))$ and
$\partial_j E_{ji}(\cdot/\varepsilon,\varphi(\cdot))$
are still continuous under the smoothness assumption $\eqref{a:4}$. In the end, we mention that
the estimate $\eqref{pri:2.13}$ is also valid for the operator $D_j^{-h}$.}
\end{remark}

\begin{proof}
Obviously, using the same argument as in the proof of Lemma $\ref{lemma:2.4}$, we may
easily carry out the proof for the estimate $\eqref{pri:2.11}$ and the formula
$\eqref{eq:2.6}$, which are all based upon the estimate $\eqref{pri:2.10}$.
Meanwhile, $\eqref{eq:2.5}$ and $\eqref{pri:2.12}$
have been shown in the previous work \cite{WXZ1}.
Here we just give the expression of
\begin{equation*}
  \Psi_{ji}(y,\varphi)
  = \frac{E_{ji}(y,\varphi(\cdot+he_i)) - E_{ji}(y,\varphi)}{
|\varphi(\cdot+he_i)-\varphi|^\gamma},
\end{equation*}
and note that $|\varphi(\cdot+he_i)|\lesssim |\varphi|$ as long as $h$ is sufficiently small, and
this inequality is independent of the variable of $\varphi$,
and this has shown the equality $\eqref{eq:2.6}$.

Then we turn to show the estimate $\eqref{pri:2.13}$, and use the same notation
as in the proof of Lemma $\ref{lemma:2.4}$. Let $Y_{\varepsilon}^i = \varepsilon(i + Y)$ for $i\in \mathbb{Z}^d$, and we collect a family of $\{Y_\varepsilon^i\}$, whose index set is denoted by $I_\varepsilon$, such that
$\text{supp}\varphi\subset \cup_{i\in I_\varepsilon}Y_\varepsilon^i\subset\Omega$ and
$Y_\varepsilon^i\cap Y_\varepsilon^j = \emptyset$ if $i\not=j$.
Thus, the left-hand side of
$\eqref{pri:2.13}$ may be controlled by
\begin{equation*}
\begin{aligned}
\sum_{k\in I_\varepsilon}\int_{Y_\varepsilon^k}
\big|D_j^h(\varepsilon E_{ji}(x/\varepsilon,\varphi))
&-\partial_j E_{ji}(x/\varepsilon,\varphi)\big|^rdx
\leq \varepsilon^{r-r(1-\gamma)\tau} \underbrace{\sum_{k\in I_\varepsilon}\int_{Y_\varepsilon^k}
\big|\Psi_{ji}(y,\varphi)\big|^r |D_j^{h}\varphi|^{r\gamma} dx}_{I_1} \\
& + \underbrace{\int_{\Omega}\int_0^1
\big|\partial_j E_{ji}(y+t(h/\varepsilon)e_j,\bar{\varphi})
- \partial_j E_{ji}(y,\varphi)\big|^rdt dx}_{I_2},
\end{aligned}
\end{equation*}
where $\varphi^k = \inf_{x\in Y_\varepsilon^k}|\varphi|$, and $\hat{x}\in Y_\varepsilon^k$.
Then, by a similar computation given in the proof of $\eqref{f:2.20}$,
we have
\begin{equation}\label{f:2.17}
\begin{aligned}
I_2
&\lesssim (h/\varepsilon)^{r\rho}\int_0^1
t^{r\rho}dt\int_\Omega\big[\nabla E(\cdot,\bar{\varphi})\big]_{C^{0,\rho}(\mathbb{R}^d)}^{r}dx
+ \underbrace{\int_{\Omega}
\big|\partial_j E_{ji}(y,\bar{\varphi})
- \partial_j E_{ji}(y,\varphi)\big|^r dx}_{I_3}\\
&\lesssim \varepsilon^{r\rho(\tau-1)}\int_\Omega|\varphi|^{r(p-1)}dx
+ \varepsilon^{r\tau\gamma}
\int_\Omega |\nabla\varphi|^{r\gamma}|\varphi|^{r\tilde{\beta}}dx,
\end{aligned}
\end{equation}
where we use the estimate $\eqref{pri:2.14}$ to derive the first term in the second inequality.
Note that the calculation for $I_3$ is totally similar to that given in $\eqref{f:2.22}$,
and a required estimate could be found in Remark $\ref{remark:2.4}$. The notation $\tilde{\beta}$
is given after the estimate $\eqref{f:2.18}$.
Also, an analogy computation leads to
\begin{equation}\label{f:2.18}
\begin{aligned}
I_1
&\lesssim \sum_{k\in I_\varepsilon}|Y_\varepsilon^k|\dashint_{Y_\varepsilon^k}
|\Psi(y,\varphi)|^r
dx|(D_j^h\varphi)(\hat{x})|^{r\gamma}\\
&\lesssim \sum_{k\in I_\varepsilon}|Y_\varepsilon^k||\varphi^k|^{r\tilde{\beta}}
|(D_j^h\varphi)(\hat{x})|^{r\gamma}
\lesssim \int_\Omega |\varphi|^{r\tilde{\beta}}|(D_j^h\varphi)|^{r\gamma} dx,
\end{aligned}
\end{equation}
where $\tilde{\beta} = (p-2)(1+p)/p$ if $2<p<\infty$, and
$\tilde{\beta}= (2-p)(p-1)/(3-p)$ if $1<p<2$. Here the point $\hat{x}\in Y_\varepsilon^k$ is such
that $|D_i^h(\varphi)(\hat{x})|=\sup_{x\in Y_\varepsilon^k}|D_i^h(\varphi)(x)|$, and
then we choose $\varphi^k = \varphi(\hat{x})$. In fact, we employ the estimate $\eqref{pri:2.11}$
in the second inequality, and the last one is due to the definition of Riemann integral
for $\varphi\in C_0^1(\Omega)$.
Finally, the estimates $\eqref{f:2.17}$ and
$\eqref{f:2.18}$ imply the desired estimate $\eqref{pri:2.13}$, and
the proof is complete.
\end{proof}

\subsection{Smoothing operator and its properties}

\begin{definition}\label{def:1}
Fix a nonnegative function $\zeta\in C_0^\infty(B(0,1/2))$, and $\int_{\mathbb{R}^d}\zeta(x)dx = 1$. Define the smoothing operator
\begin{equation}\label{def:2.1}
S_\varepsilon(f)(x) = f*\zeta_\varepsilon(x) = \int_{\mathbb{R}^d} f(x-y)\zeta_\varepsilon(y) dy,
\end{equation}
where $\zeta_\varepsilon=\varepsilon^{-d}\zeta(x/\varepsilon)$.
\end{definition}

\begin{lemma}
Let $f\in L^p(\Omega)$
for some $1\leq p<\infty$. Then for any $\varpi\in L_{per}^p(Y)$,
\begin{equation}\label{pri:2.6}
\big\|\varpi(\cdot/\varepsilon)S_\varepsilon(f)\big\|_{L^p(\Omega)}
\leq C\big\|\varpi\big\|_{L^p(Y)}\big\|f\big\|_{L^p(\Omega)},
\end{equation}
where $C$ depends only on $\zeta$ and $d$.
\end{lemma}

\begin{proof}
The proof may be find in \cite[Lemma 2.1]{S5} or \cite[Lemma 3.2]{X3}.
\end{proof}

\begin{lemma}\label{lemma:2.7}
Let $f\in W^{1,p}(\Sigma_{\varepsilon})$ for some $1<p<\infty$. Then we have
\begin{equation}\label{pri:2.7}
\big\|S_\varepsilon(f)-f\big\|_{L^p(\Sigma_{2\varepsilon})}
\leq C\varepsilon\big\|\nabla f\big\|_{L^p(\Sigma_{\varepsilon})},
\end{equation}
where $C$ depends only on $d$.
\end{lemma}

\begin{proof}
The proof may be found in \cite[Lemma 2.2]{S5} and \cite[Lemma 3.3]{X3}, which is based upon
the characterization of $W^{1,p}$ functions, i.e.,
$\|f(\cdot-\varepsilon y) - f(\cdot)\|_{L^p(\Sigma_{2\varepsilon})}
\leq \varepsilon\|\nabla f\|_{L^p(\Sigma_{\varepsilon})}$, where $|y|\leq 1$. This
together with Minkowski’s inequality consequently leads to the stated estimate $\eqref{pri:2.7}$.
\end{proof}

\begin{remark}\label{remark:2.3}
\emph{To obtain the inequality $\eqref{pri:1.9}$,
we need to apply the estimate $\eqref{pri:2.7}$ in each of small cells $Y_{\varepsilon}^k$ (see the proof
of $\eqref{f:3.20}$). So,
it is crucial in applications that the constant $C$ is independent of the size of domain.
Compared to \cite[Lemma 3.3]{X3}, it is not hard to see that $|\nabla u_0|^r$ with some $r>1$
in $\eqref{pri:1.9}$ plays a role as a weight.}
\end{remark}

\section{Convergence rates}\label{sec:3}

\subsection{Variational formula}

\begin{lemma}\label{lemma:3.1}
Let $\Omega\subset\mathbb{R}^d$ be a bounded Lipschitz domain.
Suppose that $u_\varepsilon, u_0\in W^{1,p}(\Omega)$
satisfy $\mathcal{L}_\varepsilon u_\varepsilon = \mathcal{L}_0 u_0$ in
$\Omega$ with $u_\varepsilon = u_0$ on $\partial\Omega$. For any $\varphi\in C^1_0(\Omega;\mathbb{R}^d)$,
the first-order approximating corrector is given by
$v_\varepsilon = u_0 + \varepsilon N(x/\varepsilon,\varphi)$, and let
$V=(V_i)\in L^p(\Omega;\mathbb{R}^d)$, with
\begin{equation}\label{eq:3.1}
 V_i(x) = \nabla_i u_0(x)  + D^{h}_{i}\big(\varepsilon N(x/\varepsilon,\varphi)\big)
\end{equation}
for some $0<h<\varepsilon$.
Assume that $A$ satisfies $\eqref{a:1}$ and $\eqref{c:1}$,
then we have
\begin{equation}\label{eq:3.2}
\begin{aligned}
&\int_{\Omega} \big(A(x/\varepsilon,\nabla u_\varepsilon) - A(x/\varepsilon,V)\big)
\cdot (\nabla u_\varepsilon - V) dx\\
&= \int_\Omega \big[\widehat{A}(\nabla u_0) - A(x/\varepsilon,V)\big] \cdot(\nabla u_\varepsilon - V) dx
- o(\varepsilon)\int_\Omega N(x/\varepsilon,\varphi) dx
\end{aligned}
\end{equation}
as $h$ is sufficiently small.
\end{lemma}

\begin{proof}
In fact, it suffices to show
\begin{equation*}
\int_{\Omega} A(x/\varepsilon,\nabla u_\varepsilon)
\cdot (\nabla u_\varepsilon - V) dx
=\int_\Omega \widehat{A}(\nabla u_0)\cdot(\nabla u_\varepsilon - V) dx
- o(\varepsilon)\int_\Omega N(x/\varepsilon,\varphi) dx.
\end{equation*}
In view of
$\mathcal{L}_\varepsilon u_\varepsilon = \mathcal{L}_0 u_0$ in $\Omega$, one may have
\begin{equation*}
 D_i^{-h} \big[A(x/\varepsilon,\nabla u_\varepsilon)-\widehat{A}(\nabla u_0)\big]
 = o(1),
\end{equation*}
as $h\to 0$.
This gives
\begin{equation*}
\int_\Omega D_i^{-h} \big[A_i(x/\varepsilon,\nabla u_\varepsilon)-\widehat{A}_i(\nabla u_0)\big] \varepsilon N(x/\varepsilon,\varphi) dx
= o(\varepsilon)\int_\Omega N(x/\varepsilon,\varphi) dx,
\end{equation*}
and using integration by parts for difference quotients,
\begin{equation*}
\int_\Omega A_i(x/\varepsilon,\nabla u_\varepsilon) D_i^{h} \big(\varepsilon N(x/\varepsilon,\varphi)\big) dx
= \int_\Omega \widehat{A}_i(\nabla u_0) D_i^{h} \big(\varepsilon N(x/\varepsilon,\varphi)\big) dx
-o(\varepsilon)\int_\Omega N(x/\varepsilon,\varphi) dx.
\end{equation*}
Note that $N(x/\varepsilon,\varphi) = 0$ on $\partial\Omega$ due to Remark $\eqref{remark:2.1}$.
On the other hand, we have
\begin{equation*}
\int_\Omega A(x/\varepsilon,\nabla u_\varepsilon) \cdot\big(\nabla u_\varepsilon - \nabla u_0\big) dx
= \int_{\Omega} \widehat{A}(\nabla u_0)\cdot\big(\nabla u_\varepsilon - \nabla u_0\big) dx,
\end{equation*}
where we also use the fact that $u_\varepsilon-u_0\in W_0^{1,p}(\Omega)$. Obviously, the
above two equalities imply the desired equality $\eqref{eq:3.2}$. We are done.
\end{proof}

\begin{lemma}\label{lemma:3.4}
Assume the same conditions as in Lemma $\ref{lemma:3.1}$.
Let $K = (K_i)$ be with
\begin{equation}\label{eq:3.4}
K_i = \nabla_i u_\varepsilon - \nabla_i u_0
- D_i^{h}\big(\varepsilon N(y,\varphi_\varepsilon)\big),
\qquad y = x/\varepsilon.
\end{equation}
Also assume $\varphi_\varepsilon = S_\varepsilon(\psi_{4\varepsilon}\nabla u_0)$, where
$\psi_{4\varepsilon}\in C_0^1(\Omega)$ is defined in $\eqref{notation:2}$.
Then we have
\begin{equation}\label{pri:3.16}
\begin{aligned}
\|K\|_{L^p(\Omega)}^p
\leq  \bigg|\int_\Omega
&\Big[\widehat{A}(\nabla u_0) - \widehat{A}(\varphi_\varepsilon) \\
& +\widehat{A}(\varphi_\varepsilon) - A(x/\varepsilon,P(x/\varepsilon,\varphi_\varepsilon)) \\
& +A(x/\varepsilon,P(x/\varepsilon,\varphi_\varepsilon))
- A(x/\varepsilon,V)\Big] \cdot K dx\bigg|,
\end{aligned}
\end{equation}
by ignoring a quantity $o(\varepsilon)$.
\end{lemma}

\begin{proof}
In view of the assumption $\eqref{c:1}$, the left-hand side of
the equality $\eqref{eq:3.2}$ gives
\begin{equation}\label{f:3.23}
\int_{\Omega} \big(A(x/\varepsilon,\nabla u_\varepsilon) - A(x/\varepsilon,V)\big)
\cdot K dx \geq \mu_0\|K\|_{L^p(\Omega)}^{p},
\end{equation}
while its right-hand side is equal to
\begin{equation}\label{f:3.24}
\begin{aligned}
\int_\Omega \Big[\widehat{A}(\nabla u_0) - \widehat{A}(\varphi_\varepsilon)
&+\widehat{A}(\varphi_\varepsilon) - A(x/\varepsilon,P(x/\varepsilon,\varphi_\varepsilon))\\
&+A(x/\varepsilon,P(x/\varepsilon,\varphi_\varepsilon)) - A(x/\varepsilon,V)\Big] \cdot K dx
+o(\varepsilon)\int_{\Omega}N(x/\varepsilon,\varphi_\varepsilon)dx.
\end{aligned}
\end{equation}
Note that
\begin{equation}\label{f:3.40}
\begin{aligned}
  \int_{\Omega}|N(x/\varepsilon,\varphi_\varepsilon)|dx
&\lesssim \sum_{k\in I_\varepsilon}|Y_\varepsilon^k|
\dashint_{Y_\varepsilon^k}|N(x/\varepsilon,\varphi_\varepsilon)|dx \\
&\lesssim \sum_{k\in I_\varepsilon}|Y_\varepsilon^k|\inf_{x\in Y_\varepsilon^k}|\varphi_\varepsilon|
\lesssim \int_{\Omega}|\varphi_\varepsilon|dx \lesssim
\int_{\Sigma_{4\varepsilon}}|\nabla u_0|dx
\end{aligned}
\end{equation}
where we use the estimate $\eqref{pri:2.1}$ in the second inequality, and
the last one follows from the estimate $\eqref{pri:2.6}$ by recalling
$\varphi_\varepsilon = S_\varepsilon(\psi_{4\varepsilon}\nabla u_0)$.
Thus, due to the estimate $\eqref{f:3.40}$ we may omit
the second term of $\eqref{f:3.24}$, otherwise there is nothing to prove.
Combining the estimates $\eqref{f:3.23}$ and $\eqref{f:3.24}$ leads to
our desired estimate $\eqref{pri:3.16}$, and
we have completed the proof.
\end{proof}

\subsection{Estimates involving a shift argument}

\begin{lemma}\label{lemma:3.6}
Assume the same conditions as in Lemma $\ref{lemma:3.4}$.
Let $K$ be defined in $\eqref{eq:3.4}$, and $u_0\in W^{1,p}(\Omega)$ satisfies
$\mathcal{L}_0(u_0) = \mathcal{L}_\varepsilon(u_\varepsilon)=0$ in $\Omega$
with the assumption $\eqref{a:5}$. Then we have
\begin{equation}\label{pri:3.6}
\begin{aligned}
&\quad\Big|\int_{\Omega} \big[\widehat{A}(\nabla u_0)
- \widehat{A}(\varphi_\varepsilon)\big]
\cdot \tilde{K}dx\Big|
\lesssim (1-I_{\{p=2\}})|\Omega_{M,\varepsilon}|^{\beta(p-1)}
\Big(\int_{\Omega_{M,\varepsilon}}|\nabla u_0|^qdx \Big)^{\frac{p-1}{q}}\\
&+
\left\{\begin{aligned}
&\Big(\int_{\Omega\setminus\Sigma_{8\varepsilon}}|\nabla u_0|^pdx\Big)^{1-\frac{1}{p}}
+ \varepsilon
\Big(\int_{\Sigma_{4\varepsilon}}|\nabla^2 u_0|^2|\nabla u_0|^{p-2}dx\Big)^{\frac{1}{2}}
\Big(\int_\Omega|\nabla u_0|^p\Big)^{\frac{p-2}{2p}}
&\text{if~}& p\in[2,\infty);\\
&\Big(\int_{\Omega\setminus\Sigma_{8\varepsilon}}|\nabla u_0|^pdx\Big)^{1-\frac{1}{p}}
+ \varepsilon^{\frac{p-1}{3-p}}
\Big(\int_{\Sigma_{4\varepsilon}}|\nabla^2 u_0|^pdx\Big)^{\frac{p-1}{p(3-p)}}
\Big(\int_\Omega|\nabla u_0|^p\Big)^{\frac{(2-p)(p-1)}{p(3-p)}}
&\text{if~}& p\in(1,2),
\end{aligned}\right.
\end{aligned}
\end{equation}
where the up to constant depends on $\mu_0,\mu_1,p$ and $d$.
\end{lemma}

\begin{proof}
Obviously, the proof should be divided into two parts. We first handle the case
$2\leq p <\infty$. It follows from the estimate $\eqref{pri:2.3}$ that
\begin{equation}\label{f:3.8}
\begin{aligned}
\Big|\int_{\Omega} \big[\widehat{A}(\nabla u_0)
- \widehat{A}(\varphi_\varepsilon)
\big]\cdot \tilde{K}dx\Big|
&\lesssim \int_{\Omega}|\nabla u_0 - \varphi_\varepsilon|
\big(|\nabla u_0| + |\varphi_\varepsilon|\big)^{p-2}|\tilde{K}|dx \\
&\lesssim \int_{\Omega\setminus\Sigma_{3\varepsilon}}|\nabla u_0|^{p-1}|\tilde{K}|dx
+ \int_{\Sigma_{3\varepsilon}}|\nabla u_0 - \varphi_\varepsilon|
|\nabla u_0|^{p-2}|\tilde{K}|dx \\
&\lesssim \underbrace{\int_{\Omega\setminus\Sigma_{8\varepsilon}}
|\nabla u_0|^{p-1}|\tilde{K}|dx}_{I_1}
+ \underbrace{\int_{\Sigma_{3\varepsilon}}|\psi_{4\varepsilon}\nabla u_0 - \varphi_\varepsilon|
|\nabla u_0|^{p-2}|\tilde{K}|dx}_{I_2}
\end{aligned}
\end{equation}
in which we recall that
$\varphi_\varepsilon = S_\varepsilon(\psi_{4\varepsilon}\nabla u_0)$, and employ the facts:
\begin{equation}\label{f:3.13}
\text{supp}S_\varepsilon(\psi_{4\varepsilon}\nabla u_0) \subset \Sigma_{3\varepsilon},
\qquad \text{and}\qquad
\big|S_\varepsilon(\psi_{4\varepsilon}\nabla u_0)(x)\big|\leq  2|\nabla u_0 (x)|
\qquad \forall x\in\Sigma_{3\varepsilon},
\end{equation}
whenever $\varepsilon$ is sufficiently small.
Also, the last step is due to the definition of the cut-off function.
The easier term is
\begin{equation}\label{f:3.9}
I_1 \leq \Big(\int_{\Omega\setminus\Sigma_{8\varepsilon}}
|\nabla u_0|^{p}dx\Big)^{1-\frac{1}{p}}.
\end{equation}
Then we show the estimate on $I_2$,
\begin{equation*}
\begin{aligned}
I_2 &\leq \Big(\int_{\Sigma_{3\varepsilon}}|\psi_{4\varepsilon}\nabla u_0
- \varphi_\varepsilon|^{\frac{p}{p-1}}|\nabla u_0|^{\frac{p(p-2)}{p-1}}dx\Big)^{1-\frac{1}{p}} \\
&\leq \Big(\underbrace{\int_{\Sigma_{3\varepsilon}\setminus
\Omega_{M,\varepsilon}}|\psi_{4\varepsilon}\nabla u_0
- \varphi_\varepsilon|^{\frac{p}{p-1}}|\nabla u_0|^{\frac{p(p-2)}{p-1}}dx}_{I_{21}}
\Big)^{1-\frac{1}{p}}
+ \Big(\underbrace{\int_{\Omega_{M,\varepsilon}}|\psi_{4\varepsilon}\nabla u_0
- \varphi_\varepsilon|^{\frac{p}{p-1}}|\nabla u_0|^{\frac{p(p-2)}{p-1}}dx}_{I_{22}}
\Big)^{1-\frac{1}{p}}.
\end{aligned}
\end{equation*}
Let
$\Sigma_{3\varepsilon}\setminus\Omega_{M,\varepsilon}
\subset \cup_{k\in I_\varepsilon^{\prime}} Y_{\varepsilon}^k$. Clearly,
$I_\varepsilon^{\prime}\subset I_\varepsilon$,
and the set $\Sigma_{3\varepsilon}\setminus\Omega_{M,\varepsilon}$ is open. Here
the notation $Y_{\varepsilon}^k$ and
$I_\varepsilon$ are given in Lemma $\ref{lemma:2.4}$.
Note that for any $k\in I_\varepsilon^\prime$, there holds
a Harnack's inequality
\begin{equation}\label{f:3.7}
  \sup_{Y_{2\varepsilon}^k} |\nabla u_0|
  \leq 4\inf_{Y_{2\varepsilon}^k}|\nabla u_0|
\end{equation}
in small scales, since $|\nabla u_0|>M\varepsilon^{\vartheta}$ in
$\Sigma_{3\varepsilon}\setminus\Omega_{M,\varepsilon}$ under the assumption $\eqref{a:5}$,
Thus,
where we use the estimate $\eqref{pri:2.7}$ in the small scales. Then,
one may have the following
computation:
\begin{equation}\label{f:3.20}
\begin{aligned}
I_{21}
&\leq \sum_{k\in I_\varepsilon^\prime}\Big(\int_{Y_\varepsilon^k}
|\psi_{4\varepsilon}\nabla u_0
- \varphi_\varepsilon|^{\frac{p}{p-1}}dx\Big)\sup_{Y_\varepsilon^k}|\nabla u_0|^{\frac{p(p-2)}{p-1}}
\\
&\lesssim \varepsilon^{\frac{p}{p-1}}\sum_{k\in I_\varepsilon^\prime}\Big(\int_{Y_{2\varepsilon}^k}
\big|\nabla(\psi_{4\varepsilon}\nabla u_0)\big|^{\frac{p}{p-1}}dx\Big)
\inf_{Y_{2\varepsilon}^k}|\nabla u_0|^{\frac{p(p-2)}{p-1}} \\
&\lesssim \varepsilon^{\frac{p}{p-1}}\sum_{k\in I_\varepsilon^\prime}\int_{Y_{2\varepsilon}^k}
\big|\nabla(\psi_{4\varepsilon}\nabla u_0)\big|^{\frac{p}{p-1}}
|\nabla u_0|^{\frac{p(p-2)}{p-1}}dx \\
&\lesssim \varepsilon^{\frac{p}{p-1}}\int_{\Omega}
\big|\nabla(\psi_{4\varepsilon}\nabla u_0)\big|^{\frac{p}{p-1}}
|\nabla u_0|^{\frac{p(p-2)}{p-1}}dx
\end{aligned}
\end{equation}
where we use the estimate $\eqref{pri:2.7}$ in the small scales,
as well as $\eqref{f:3.7}$ in the second inequality. Then, the last line above
will be immediately controlled by
\begin{equation*}
\int_{\Omega\setminus\Sigma_{8\varepsilon}}
|\nabla u_0|^{p}dx
+ \varepsilon^{\frac{p}{p-1}}\int_{\Sigma_{4\varepsilon}}
\big|\nabla^2 u_0\big|^{\frac{p}{p-1}}
|\nabla u_0|^{\frac{p(p-2)}{p-1}}dx.
\end{equation*}
Therefore, we arrive at
\begin{equation*}
\begin{aligned}
  I_{21}
&\lesssim
  \bigg\{\int_{\Omega\setminus\Sigma_{8\varepsilon}}
|\nabla u_0|^{p}dx
+ \varepsilon^{\frac{p}{p-1}}\Big(\int_{\Sigma_{4\varepsilon}}
\big|\nabla^2 u_0\big|^{2}
|\nabla u_0|^{p-2}dx\Big)^{\frac{p}{2(p-1)}}
\Big(\int_\Omega|\nabla u_0|^{p}dx\Big)^{\frac{p-2}{2(p-1)}}\bigg\},
\end{aligned}
\end{equation*}
where we just merely use H\"older's inequality and it asks for the condition $p\geq 2$.
We now proceed to show the estimate on $I_{22}$. Note that
it is fine to assume that there are finite connected components,
denoted by $\{\Omega_i\}_{i=1}^{k_0}$ such that
$\Omega_{M,\varepsilon}= \cup_{i=1}^{k_0}\Omega_i$ since the geometry
of $\Omega_{M,\varepsilon}$ is not bad in general on account of the assumption $\eqref{a:5}$.
\begin{equation}\label{f:3.36}
\begin{aligned}
I_{22} \leq \sum_{i=1}^{k_0} \int_{\Omega_{i}}|\nabla u_0|^p dx
\leq \sum_{i=1}^{k_0}|\Omega_i|^{1-\frac{p}{q}}
\Big(\int_{\Omega_i}|\nabla u_0|^qdx\Big)^{\frac{p}{q}}
\lesssim |\Omega_{M,\varepsilon}|^{1-\frac{p}{q}}
\Big(\int_{\Omega_{M,\varepsilon}}|\nabla u_0|^qdx\Big)^{\frac{p}{q}},
\end{aligned}
\end{equation}
where we just use H\"older's inequality.
Thus,
Plugging the above expression back into $I_2$, we have
\begin{equation*}
\begin{aligned}
  I_2
& \lesssim
|\Omega_{M,\varepsilon}|^{\beta(p-1)}
\Big(\int_{\Omega_{M,\varepsilon}}|\nabla u_0|^qdx\Big)^{\frac{p-1}{q}}
+\Big(\int_{\Omega\setminus\Sigma_{8\varepsilon}}
|\nabla u_0|^{p}dx\Big)^{1-\frac{1}{p}}\\
&+ \varepsilon\Big(\int_{\Sigma_{4\varepsilon}}
\big|\nabla^2 u_0\big|^{2}
|\nabla u_0|^{p-2}dx\Big)^{\frac{1}{2}}
\Big(\int_\Omega|\nabla u_0|^pdx\Big)^{\frac{p-2}{2p}},
\end{aligned}
\end{equation*}
where we use H\"older's inequality in the second line,
which asks for the condition $p\geq 2$. This together with
the estimates $\eqref{f:3.8}$ and $\eqref{f:3.9}$ consequently gives the
first line of the desired estimate $\eqref{pri:3.6}$.

Now, we turn to show the proof in the case of $1<p<2$. In view of the estimate $\eqref{pri:2.18}$,
\begin{equation*}
\begin{aligned}
\Big|\int_{\Omega} \big[\widehat{A}(\nabla u_0)
- \widehat{A}(\varphi_\varepsilon)
\big]\cdot \tilde{K}dx\Big|
&\lesssim \int_{\Omega}|\nabla u_0 - \varphi_\varepsilon|^{\frac{p-1}{3-p}}
\big(|\nabla u_0| + |\varphi_\varepsilon|\big)^{\frac{(p-1)(2-p)}{3-p}}|\tilde{K}|dx \\
&\lesssim \underbrace{\Big(\int_{\Omega}|\nabla u_0 - \varphi_\varepsilon|^{\frac{p}{3-p}}
\big(|\nabla u_0| + |\varphi_\varepsilon|\big)^{\frac{p(2-p)}{3-p}}dx\Big)^{1-\frac{1}{p}}}_{I_3}.
\end{aligned}
\end{equation*}
In fact, it suffices to estimate the term $I_3$, and we have
\begin{equation*}
\begin{aligned}
I_3 &\leq \Big(\int_{\Omega\setminus\Sigma_{3\varepsilon}}
|\nabla u_0|^{p}dx\Big)^{1-\frac{1}{p}}
+ \Big(\int_{\Sigma_{3\varepsilon}}|\nabla u_0 - \varphi_\varepsilon|^{\frac{p}{3-p}}
\big|\nabla u_0\big|^{\frac{p(2-p)}{3-p}}dx\Big)^{1-\frac{1}{p}}\\
&\lesssim \Big(\int_{\Omega\setminus\Sigma_{8\varepsilon}}
|\nabla u_0|^{p}dx\Big)^{1-\frac{1}{p}}
+ \Big(\int_{\Sigma_{4\varepsilon}}|\psi_{4\varepsilon}\nabla u_0
- \varphi_\varepsilon|^{\frac{p}{3-p}}
\big|\nabla u_0\big|^{\frac{p(2-p)}{3-p}}dx\Big)^{1-\frac{1}{p}}
\end{aligned}
\end{equation*}
by noting the facts $\eqref{f:3.13}$. Then using the same argument as in the proof of
the estimates $\eqref{f:3.20}$, $\eqref{f:3.36}$ we can derive that
\begin{equation*}
\begin{aligned}
I_3
&\lesssim |\Omega_{M,\varepsilon}|^{\beta(p-1)}
\Big(\int_{\Omega_{M,\varepsilon}}|\nabla u_0|^qdx\Big)^{\frac{p-1}{q}}
+ \Big(\int_{\Omega\setminus\Sigma_{8\varepsilon}}
|\nabla u_0|^{p}dx\Big)^{1-\frac{1}{p}} \\
&
+ \varepsilon^{\frac{p-1}{3-p}}
\Big(\int_{\Sigma_{4\varepsilon}}|\nabla^2 u_0|^{\frac{p}{3-p}} dx\Big)^{\frac{p-1}{p(3-p)}}
\Big(\int_{\Omega}|\nabla u_0|^pdx\Big)^{\frac{(p-1)(2-p)}{p(3-p)}},
\end{aligned}
\end{equation*}
where we also use the H\"older's inequality in the above estimate,
which requires the condition $1<p< 2$. This will give
the second line of $\eqref{pri:3.6}$,
and we have completed the whole proof.
\end{proof}

\subsection{Applications of the flux corrector}

\begin{lemma}\label{lemma:3.5}
Assume the same conditions as in Lemma $\ref{lemma:3.4}$.
Let $K$ be defined in $\eqref{eq:3.4}$, and $E_{ji}$ be the flux corrector
defined in Lemma $\ref{lemma:2.3}$.
Then we obtain
\begin{equation}\label{pri:3.5}
\begin{aligned}
\bigg|\int_{\Omega} \big[\widehat{A}(\varphi_\varepsilon)
&- A(y,P(y,\varphi_\varepsilon))\big]
\cdot K dx\bigg| \\
&\leq \left\{
\begin{aligned}
&\int_{\Omega} \big|
D_j^{-h}(\varepsilon E_{ji}(y,\varphi_\varepsilon))-\partial_j E_{ji}(y,\varphi_\varepsilon)
\big|\big|K_i\big|dx &\text{if}&\quad p\in(1,2)\cup(2,\infty);\\
&\int_{\Omega} \Big|\Big\{\frac{\partial}{\partial\xi_k}
\big[E_{ji}(x/\varepsilon,\varphi_\varepsilon)\big]
\frac{\partial}{\partial x_j}\big[(\varphi_\varepsilon)_k\big]\Big\}K_i\Big|dx
&\text{if}&\quad p=2;\\
\end{aligned}\right.
\end{aligned}
\end{equation}
by ignoring a systematic error
$o(\varepsilon)$, where the notation $(\varphi_\varepsilon)_k$ denotes the
$k^{\text{th}}$ component of $\varphi_\varepsilon$, and
$\varphi_\varepsilon$ is given in Lemma $\ref{lemma:3.4}$.
\end{lemma}

\begin{remark}
\emph{In the case $p=2$, there is in fact no real systematic error.}
%\begin{equation}
%\begin{aligned}
%\bigg|\int_{\Omega} \big[\widehat{A}(\varphi_\varepsilon)
%&- A(y,P(y,\varphi_\varepsilon))\big]
%\cdot K dx\bigg| \\
%&\leq
%\left\{\begin{aligned}
%&\int_{\Omega} \big|
%D_j^{-h}(\varepsilon E_{ji}(y,\varphi_\varepsilon))-\partial_j E_{ji}(y,\varphi_\varepsilon)
%\big|\big|K_i\big|dx &~\text{if}~& p\in (1,2)\cup (2,\infty);\\
%&\varepsilon\int_{\Omega} \Big|\Big\{\frac{\partial}{\partial \xi_k}
%\big[E_{ji}(x/\varepsilon,\varphi_\varepsilon)\big]\frac{\partial}{\partial x_j}
%\big[(\varphi_{\varepsilon})_k\big]+o(1)e_i\Big\}K_i \Big|dx
%&~\text{if}~& p=2,
%\end{aligned}\right.
%\end{aligned}
%\end{equation}
%where the notation
%$(\varphi_\varepsilon)_k$ denotes the $k^{\text{th}}$ component of
%$\varphi_\varepsilon$, and $\varphi_\varepsilon = S_\varepsilon(\psi_{4\varepsilon}\nabla u_0)$.
\end{remark}

\begin{proof}
The main idea is to use
the antisymmetry of the flux corrector $E$ as one did in the linear equations,
which play an important role in quantitative estimates.
We first show a proof of the first case of $\eqref{pri:3.5}$, and inspired by the equation
$\eqref{eq:2.6}$,
\begin{equation}\label{f:3.5}
\begin{aligned}
\int_{\Omega} \big[\widehat{A}(\varphi_\varepsilon)
&- A(y,P(y,\varphi_\varepsilon))\big]
\cdot K dx
= -\int_{\Omega}
\frac{\partial}{\partial y_j}
\big[E_{ji}(y,\varphi_\varepsilon)\big] K_i dx \\
& = \int_{\Omega}
\Big(D_j^{-h}\big[\varepsilon E_{ji}(y,\varphi_\varepsilon)\big]
-\frac{\partial}{\partial y_j}\big[E_{ji}(y,\varphi_\varepsilon)\big]\Big) K_i dx
+ \varepsilon\int_{\Omega} E_{ji}(y,\varphi_\varepsilon) D_j^{h}(K_i) dx,
\end{aligned}
\end{equation}
where we use ``integration by parts'' formula for difference quotients,
\begin{equation*}
-\int_\Omega D_j^{-h}\big[\varepsilon E_{ji}(y,\varphi_\varepsilon)\big] K_i dx
=\varepsilon\int_{\Omega} E_{ji}(y,\varphi_\varepsilon) D_j^{h}(K_i) dx.
\end{equation*}
Since $\nabla u_\varepsilon, \nabla u_0\in L^p(\Omega;\mathbb{R}^d)$ with
higher regularities, there may hold the following facts
\begin{equation}\label{eq:3.3}
K_i = D_i^{h} \big[\underbrace{u_\varepsilon + u_0 +
\varepsilon N(y,\varphi_\varepsilon)}_{f}\big] + o(1)e_i
\qquad\text{and}\qquad D_j^{h}D_i^{h} f = D_i^{h}D_j^{h} f
\end{equation}
in a weak sense, whenever $0<h\ll 1$. This together with the property $\eqref{eq:2.4}$ leads to
\begin{equation*}
\begin{aligned}
\Big|\varepsilon\int_{\Omega} E_{ji}(y,\varphi_\varepsilon) D_j^{h}(K_i) dx\Big|
\leq o(\varepsilon)\int_\Omega |E(y,\varphi_\varepsilon)| dx.
\end{aligned}
\end{equation*}
Note that
\begin{equation*}
\int_\Omega |E(y,\varphi_\varepsilon)| dx
\leq \sum_{k\in I_\varepsilon}|Y_\varepsilon^k|
\dashint_{Y_\varepsilon}|E(y,\varphi_\varepsilon)| dx
\lesssim
\int_{\Omega}|\varphi_\varepsilon| dx
\lesssim \int_{\Omega}|\nabla u_0| dx,
\end{equation*}
where we use the estimate $\eqref{f:2.12}$ in the second inequality.
Thus, ignoring the quantity
\begin{equation*}
o(\varepsilon)\int_{\Omega}|\nabla u_0| dx,
\end{equation*}
the estimate $\eqref{f:3.5}$ gives the first line of $\eqref{pri:3.5}$.
Then we turn to show a proof for its second line.

In terms of the equation $\eqref{eq:2.5}$,
there holds
\begin{equation}\label{f:3.19}
\begin{aligned}
\int_{\Omega} \big[\widehat{A}(\varphi_\varepsilon)
&- A(y,P(y,\varphi_\varepsilon))\big]
\cdot K dx
= -\int_{\Omega}
\frac{\partial}{\partial y_j}
\big[E_{ji}(y,\varphi_\varepsilon)\big] K_i dx \\
& = -\int_{\Omega}
\frac{\partial}{\partial x_j}\big[\varepsilon E_{ji}(y,\varphi_\varepsilon)\big] K_i dx
+ \varepsilon\int_{\Omega} \frac{\partial}{\partial \xi_k}
\big[E_{ji}(y,\varphi_\varepsilon)\big]\frac{\partial}{\partial x_j}
\big[(\varphi_\varepsilon)_k\big]K_i dx.
\end{aligned}
\end{equation}
Note that we also have
\begin{equation*}
 \frac{\partial}{\partial x_j}\big[E_{ji}(y,\varphi_\varepsilon)\big]
= D_j^{-h}\big[E_{ji}(y,\varphi_\varepsilon)\big] + o(1)e_i,
\end{equation*}
as $h$ is sufficiently small (in such the case, $h$ could be independent of $\varepsilon$).
Thus, the above formula together with
the formula $\eqref{eq:3.3}$ and the antisymmetry property $\eqref{eq:2.4}$ gives
\begin{equation*}
\begin{aligned}
 \int_{\Omega}
\frac{\partial}{\partial x_j}\big[\varepsilon E_{ji}(y,\varphi_\varepsilon)\big]
&K_i dx
= \varepsilon\int_{\Omega} D_j^{-h}
\big[E_{ji}(y,\varphi_\varepsilon)\big] K_i dx
+o(\varepsilon)e_i\int_\Omega K_i dx \\
&= \varepsilon\int_{\Omega} D_j^{-h}
\big[E_{ji}(y,\varphi_\varepsilon)\big] D_i^{h} f dx
+ o(\varepsilon)\int_\Omega D_j^{-h}
\big[E_{ji}(y,\varphi_\varepsilon)\big] dx
+ o(\varepsilon)e_i\int_\Omega K_i dx\\
&= o(\varepsilon)e_i\int_\Omega K_i dx.
\end{aligned}
\end{equation*}
Then, plugging this back into the formula $\eqref{f:3.19}$ leads to our desired estimate in
$\eqref{pri:3.5}$. We also mention that in the last step above,
we employ the following fact that
\begin{equation*}
\int_\Omega D_j^{-h}
\big[E_{ji}(y,\varphi_\varepsilon)\big] dx
= \int_\Omega D_j^{-h}
\big[E_{ji}(y,\varphi_\varepsilon)\big]\psi_{\varepsilon} dx
= - \int_\Omega
\big[E_{ji}(y,\varphi_\varepsilon)\big]D_j^{h}\psi_{\varepsilon} dx = 0,
\end{equation*}
where $\varphi_\varepsilon = S_\varepsilon(\psi_{4\varepsilon}\nabla u_0)$, and
both of $\psi_{\varepsilon}$ and $\psi_{4\varepsilon}$ are cut-off functions defined in $\eqref{notation:2}$.
This ends the proof.
\end{proof}

%\begin{remark}
%In the case of $p=2$, we will derive that
%\begin{equation}
%\begin{aligned}
%\bigg|\int_{\Omega} \big[\widehat{A}(\varphi_\varepsilon)
%&- A(y,P(y,\varphi_\varepsilon))\big]
%\cdot K dx\bigg| \\
%&\leq\int_{\Omega} \Big|\Big\{\frac{\partial}{\partial\xi_k}
%\big[E_{ji}(x/\varepsilon,\varphi_\varepsilon)\big]
%\frac{\partial}{\partial x_j}\big[(\varphi_\varepsilon)_k\big]\Big\}K_i\Big|dx
%\end{aligned}
%\end{equation}
%\end{remark}

\subsection{Estimates on remainder terms}

\begin{lemma}[$p> 2$]\label{lemma:3.2}
Assume the same conditions as in Lemma $\ref{lemma:3.4}$.
Let $K$ be defined in $\eqref{eq:3.4}$, and $u_0\in W^{1,p}(\Omega)$ be
with the assumption $\eqref{a:5}$.
Then by setting $\tilde{K} = K/\|K\|_{L^p(\Omega)}$, we obtain
\begin{equation}\label{pri:3.7}
\begin{aligned}
\bigg|\int_\Omega
\big[A(y,P(y,\varphi_\varepsilon))
&-A(y,V)\big]\cdot \tilde{K}dx\bigg|
\lesssim
\text{the~first~line~of~r.h.s.~of~}\eqref{pri:3.6} +  \Big(\int_{\Omega\setminus
\Sigma_{3\varepsilon}}|\nabla u_0|^{\frac{p}{p-1}}dx\Big)^{\frac{p-1}{p}}\\
&+\Big(\int_\Omega\big|D^h(\varepsilon N(y,\varphi_\varepsilon))
-\nabla_yN(y,\varphi_\varepsilon)\big|^pdx\Big)^{1-\frac{2}{p}}
\Big(\int_\Omega|\nabla u_0|^pdx\Big)^{\frac{1}{p}}\\
&+\Big(\int_\Omega\big|D^h(\varepsilon N(y,\varphi_\varepsilon))
-\nabla_yN(y,\varphi_\varepsilon)\big|^pdx\Big)^{\frac{1}{p}}
\Big(\int_\Omega|\nabla u_0|^pdx\Big)^{1-\frac{2}{p}}\\
&+\Big(\int_\Omega\big|D^h(\varepsilon N(y,\varphi_\varepsilon))
-\nabla_yN(y,\varphi_\varepsilon)\big|^pdx\Big)^{1-\frac{1}{p}},
\end{aligned}
\end{equation}
where the up to constant depends only on $\mu_0,\mu_1,p$ and $d$.
\end{lemma}

\begin{proof}
It follows from the assumption $\eqref{c:1}$ that
\begin{equation}\label{f:3.16}
\begin{aligned}
\bigg|\int_\Omega
\big[A(y,P(y,\varphi_\varepsilon))
-A(y,V)\big]\cdot \tilde{K}dx\bigg|
\lesssim
\int_{\Omega}|P(y,\varphi_\varepsilon)-V|
\big(|P(y,\varphi_\varepsilon)|+|V|\big)^{p-2}|\tilde{K}|dx.
\end{aligned}
\end{equation}
Recalling the expression $V$ in $\eqref{eq:3.1}$, the right-hand side above
will be controlled by
\begin{equation*}
\begin{aligned}
&\int_{\Omega}|\varphi_\varepsilon - \nabla u_0|
\big(|P(y,\varphi_\varepsilon)|+|\nabla u_0|\big)^{p-2}|\tilde{K}|dx &:=I_1&\\
&+\int_{\Omega}|\varphi_\varepsilon - \nabla u_0|
\big|D^h\big(\varepsilon N(y,\varphi_\varepsilon)\big)\big|^{p-2}|\tilde{K}|dx &:=I_2&\\
&+\int_{\Omega}\Big|\nabla_y N(y,\varphi_\varepsilon)
- D^h\big(\varepsilon N(y,\varphi_\varepsilon)\big)\Big|
\big(|P(y,\varphi_\varepsilon)|+|V|\big)^{p-2}|\tilde{K}|dx. &:=I_3&
\end{aligned}
\end{equation*}
Before proceeding further, we state the following facts: there holds
\begin{equation}\label{f:3.10}
|P(x/\varepsilon,\varphi_\varepsilon(x))|
\lesssim |\varphi_\varepsilon(x)|
\lesssim |\nabla u_0(x)|
\end{equation}
for any $x\in\Sigma_{3\varepsilon}$, according to
the estimates $\eqref{pri:2.9}, \eqref{f:3.13}$ and, Remark $\ref{remark:2.1}$ implies
\begin{equation}\label{f:3.11}
 P(x/\varepsilon,\varphi_\varepsilon(x)) = 0
 \qquad \forall x\in\Omega\setminus\Sigma_{3\varepsilon}.
\end{equation}
Hence, using $\eqref{f:3.10}, \eqref{f:3.11}$ one may have
\begin{equation}\label{f:3.15}
\begin{aligned}
I_1 &\lesssim \int_{\Sigma_{3\varepsilon}}
|\varphi_\varepsilon - \nabla u_0||\nabla u_0|^{p-2} |\tilde{K}| dx
+ \int_{\Omega\setminus\Sigma_{3\varepsilon}}|\nabla u_0|^{p-1}|\tilde{K}|dx \\
&\lesssim \text{r.h.s.~of~} \eqref{pri:3.6},
\end{aligned}
\end{equation}
where the second inequality is due to the estimate $\eqref{f:3.8}$. By the same token, we also
have
\begin{equation}\label{f:3.12}
\begin{aligned}
&\quad\int_{\Omega}|\varphi_\varepsilon - \nabla u_0|\big|\nabla_y N(y,\varphi_\varepsilon)\big|^{p-2}
|\tilde{K}|dx \\
&\lesssim \int_{\Sigma_{3\varepsilon}}
|\varphi_\varepsilon - \nabla u_0||\nabla u_0|^{p-2} |\tilde{K}| dx
+ \int_{\Omega\setminus\Sigma_{3\varepsilon}}|\nabla u_0|^{p-1}|\tilde{K}|dx
\lesssim \text{~r.h.s.~of~} \eqref{pri:3.6}.
\end{aligned}
\end{equation}
We now turn to show the estimate of $I_2$, and
\begin{equation*}
\begin{aligned}
I_2
&\leq \int_\Omega |\varphi_\varepsilon - \nabla u_0|
\Big|D^h\big(\varepsilon N(y,\varphi_\varepsilon)\big)
-\nabla_y N(y,\varphi_\varepsilon)\Big|^{p-2}
|\tilde{K}| dx \\
&+ \int_{\Omega} |\varphi_\varepsilon - \nabla u_0|
\big|\nabla_y N(y,\varphi_\varepsilon)\big|^{p-2}
|\tilde{K}| dx.
\end{aligned}
\end{equation*}
Note that the estimate of the second term in the right-hand side above has been shown
in $\eqref{f:3.12}$. Thus we merely handle its first term, and
\begin{equation*}
\begin{aligned}
&\quad\int_\Omega |\varphi_\varepsilon - \nabla u_0|
\Big|D^h\big(\varepsilon N(y,\varphi_\varepsilon)\big)
-\nabla_y N(y,\varphi_\varepsilon)\Big|^{p-2}
|\tilde{K}| dx  \\
&\lesssim \int_{\Omega\setminus\Sigma_{3\varepsilon}}
|\nabla u_0||\tilde{K}| dx
+ \int_{\Sigma_{3\varepsilon}} |\nabla u_0|
\Big|D^h\big(\varepsilon N(y,\varphi_\varepsilon)\big)
-\nabla_y N(y,\varphi_\varepsilon)\Big|^{p-2}
|\tilde{K}| dx \\
&\lesssim \Big(\int_{\Omega\setminus\Sigma_{3\varepsilon}}
|\nabla u_0|^{\frac{p}{p-1}} dx\Big)^{1-\frac{1}{p}}
+ \Big(\int_{\Sigma_{3\varepsilon}}
\big|D^h\big(\varepsilon N(y,\varphi_\varepsilon)\big)
-\nabla_y N(y,\varphi_\varepsilon)\big|^{p}dx\Big)^{1-\frac{2}{p}}
\Big(\int_{\Omega}|\nabla u_0|^pdx\Big)^{\frac{1}{p}}
\end{aligned}
\end{equation*}
by using the facts $\eqref{f:3.10}$ and $\eqref{f:3.11}$ in the first inequality,
and the last one follows from H\"older's inequality.
This together with the estimate $\eqref{f:3.12}$ gives
\begin{equation}\label{f:3.17}
\begin{aligned}
I_2
&\lesssim
\text{~r.h.s.~of~} \eqref{pri:3.6}
+\Big(\int_{\Omega\setminus\Sigma_{3\varepsilon}}
|\nabla u_0|^{\frac{p}{p-1}} dx\Big)^{1-\frac{1}{p}} \\
& + \Big(\int_{\Sigma_{3\varepsilon}}
\big|D^h\big(\varepsilon N(y,\varphi_\varepsilon)\big)
-\nabla_y N(y,\varphi_\varepsilon)\big|^{p}dx\Big)^{1-\frac{2}{p}}
\Big(\int_{\Omega}|\nabla u_0|^pdx\Big)^{\frac{1}{p}}.
\end{aligned}
\end{equation}

We proceed to estimate $I_3$, and it follows from $\eqref{f:3.10}$ that
\begin{equation*}
\begin{aligned}
I_3 &\lesssim \int_{\Sigma_{3\varepsilon}}
\Big|\nabla_y N(y,\varphi_\varepsilon)
- D^h\big(\varepsilon N(y,\varphi_\varepsilon)\big)\Big|
|\nabla u_0|^{p-2}|\tilde{K}|dx &:=I_{31}&\\
& + \int_{\Sigma_{3\varepsilon}}
\Big|\nabla_y N(y,\varphi_\varepsilon)
- D^h\big(\varepsilon N(y,\varphi_\varepsilon)\big)\Big|
\big|D^h\big(\varepsilon N(y,\varphi_\varepsilon)\big)\big|^{p-2}|\tilde{K}|dx.
&:=I_{32}&
\end{aligned}
\end{equation*}
Then we will compute one by one, and the easier one is
\begin{equation}\label{f:3.14}
\begin{aligned}
I_{31} \leq \Big(
\int_{\Sigma_{3\varepsilon}}
\big|\nabla_y N(y,\varphi_\varepsilon)
- D^h\big(\varepsilon N(y,\varphi_\varepsilon)\big)\big|^{p}dx\Big)^{\frac{1}{p}}
\Big(\int_{\Omega}|\nabla u_0|^{p}dx\Big)^{1-\frac{2}{p}}.
\end{aligned}
\end{equation}
By using triangle's inequality coupled with H\"older's inequality, we obtain
\begin{equation*}
\begin{aligned}
I_{32}
&\lesssim
\int_{\Sigma_{3\varepsilon}}
\Big|\nabla_y N(y,\varphi_\varepsilon)
- D^h\big(\varepsilon N(y,\varphi_\varepsilon)\big)\Big|^{p-1}|\tilde{K}|dx \\
&+\int_{\Sigma_{3\varepsilon}}
\Big|\nabla_y N(y,\varphi_\varepsilon)
- D^h\big(\varepsilon N(y,\varphi_\varepsilon)\big)\Big|
\big|\nabla_yN(y,\varphi_\varepsilon)\big|^{p-2}|\tilde{K}|dx \\
&\lesssim\Big(\int_{\Sigma_{3\varepsilon}}
\big|\nabla_y N(y,\varphi_\varepsilon)
- D^h\big(\varepsilon N(y,\varphi_\varepsilon)\big)\big|^{p}dx\Big)^{1-\frac{1}{p}}
+ \text{~r.h.s.~of~} \eqref{f:3.14},
\end{aligned}
\end{equation*}
where we also use the fact $\eqref{f:3.10}$ in the last step. Obviously,
$I_3\lesssim I_{32}$. Plugging it with $\eqref{f:3.15}, \eqref{f:3.17}$ back into
the estimate $\eqref{f:3.16}$, we will have the desired estimate $\eqref{pri:3.7}$, and have completed the
whole proof.
\end{proof}

\begin{lemma}[$1<p\leq 2$]
Assume the same conditions as in Lemma $\ref{lemma:3.2}$. Then we have
\begin{equation}\label{pri:3.8}
\begin{aligned}
\bigg|\int_\Omega
\big[A(y,P(y,\varphi_\varepsilon))
-A(y,V)\big]\cdot \tilde{K}dx\bigg|
&\lesssim
\Big(\int_{\Omega\setminus
\Sigma_{8\varepsilon}}|\nabla u_0|^{p}dx\Big)^{1-\frac{1}{p}}
+ \varepsilon^{p-1}\Big(\int_{\Sigma_{4\varepsilon}}|\nabla^2 u_0|^{p}dx\Big)^{1-\frac{1}{p}} \\
&+\Big(\int_\Omega\big|D^h(\varepsilon N(y,\varphi_\varepsilon))
-\nabla_yN(y,\varphi_\varepsilon)\big|^pdx\Big)^{1-\frac{1}{p}}.
\end{aligned}
\end{equation}
\end{lemma}

\begin{proof}
In terms of the assumption $\eqref{c:1}$, we have
\begin{equation*}
\begin{aligned}
\bigg|\int_\Omega
\big[A(y,P(y,\varphi_\varepsilon))
-A(y,V)\big]\cdot \tilde{K}dx\bigg|
&\lesssim \int_{\Omega}\big|P(y,\varphi_\varepsilon)-V\big|^{p-1}\big|\tilde{K}\big|dx \\
&\lesssim \Big(\int_{\Omega}\big|P(y,\varphi_\varepsilon)-V\big|^{p}dx\Big)^{1-\frac{1}{p}},
\end{aligned}
\end{equation*}
and the right-hand side is further controlled by
\begin{equation*}
\underbrace{\Big(\int_{\Omega}\big|\varphi_\varepsilon-\nabla u_0\big|^{p}dx\Big)^{1-\frac{1}{p}}}_{I}
+ \Big(\int_{\Omega}\big|\nabla_y N(y,\varphi_\varepsilon)
-D^h\big(\varepsilon N(y,\varphi_\varepsilon)\big)\big|^{p}dx\Big)^{1-\frac{1}{p}}.
\end{equation*}
Recall that $\varphi_\varepsilon = S_\varepsilon(\psi_{4\varepsilon}\nabla u_0)$.
By $\eqref{pri:2.7}$, one may derive that
\begin{equation*}
 I \lesssim \Big(\int_{\Omega\setminus\Sigma_{8\varepsilon}}|\nabla u_0|^p dx\Big)^{1-\frac{1}{p}}
 + \varepsilon^{p-1}
 \Big(\int_{\Sigma_{4\varepsilon}}|\nabla^2 u_0|^p dx\Big)^{1-\frac{1}{p}},
\end{equation*}
and this together with the above estimate shows the estimate $\eqref{pri:3.8}$. The proof
is complete.
\end{proof}

\subsection{Layer and co-layer type estimates}

\begin{lemma}[higher regularity estimates for $u_0$]\label{lemma:3.3}
Suppose that $\widehat{A}$ satisfies the structure assumption $\eqref{c:2}$.
Let $u_0\in W^{1,p}_{loc}(\Omega)$ be a weak solution of
$\mathcal{L}_0 u_0 = 0$ in $\Omega$, then there hold
\begin{equation}\label{pri:3.9}
\dashint_{B(0,r)}|\nabla^2 u_0|^{2}|\nabla u_0|^{p-2}
\lesssim \frac{C}{r^2} \dashint_{B(0,2r)}|\nabla u_0|^{p}
\end{equation}
for the case $2\leq p<\infty$, and
\begin{equation}\label{pri:3.10}
\Big(\dashint_{B(0,r)}|\nabla^2 u_0|^{p}\Big)^{1/p}
\lesssim \frac{C}{r} \Big(\dashint_{B(0,2r)}|\nabla u_0|^{p}\Big)^{1/p}
\end{equation}
for the case $1<p\leq 2$, where $B(0,2r)\subset\Omega$ for any $0<r<1$. Also, we have
\begin{equation}\label{pri:3.18}
 \sup_{B(0,r/2)}|\nabla u_0|
\lesssim \Big(\dashint_{B(0,r)} |\nabla u_0|^p\Big)^{1/p}.
\end{equation}
Here the up to constant depends on $\tilde{\mu}_0,\tilde{\mu}_1,p$ and $d$.
\end{lemma}

\begin{proof}
The proof of the estimates $\eqref{pri:3.9},\eqref{pri:3.10}$
may be found in \cite[pp.267-271]{GE}, and the estimate $\eqref{pri:3.18}$ is included in
\cite[Theorem 3.19]{MZ}.
\end{proof}

\begin{lemma}[layer and co-layer estimates]\label{lemma:3.7}
Let $u_0\in W^{1,p}_{loc}(\Omega)$ be a weak solution of
$\mathcal{L}_0 u_0 = 0$ in $\Omega$, with the construction assumption $\eqref{c:2}$.
Then there holds the layer type estimate
\begin{equation}\label{pri:3.12}
\Big(\int_{\Omega\setminus\Sigma_{8\varepsilon}}|\nabla u_0|^p dx\Big)^{1/p}
\lesssim \varepsilon^{1/p-1/q}\Big(\int_{\Omega}|\nabla u_0|^qdx\Big)^{1/q}
\end{equation}
for some $q>p$. Meanwhile, we have the co-layer type estimates
\begin{equation}\label{pri:3.13}
\Big(\int_{\Sigma_{3\varepsilon}}|\nabla^2 u_0|^{2}|\nabla u_0|^{p-2} dx\Big)^{1/p}
\lesssim \varepsilon^{-1/p-1/q} \Big(\int_{\Omega}|\nabla u_0|^{q} dx\Big)^{1/q}
\end{equation}
for the case $2\leq p<\infty$, and
\begin{equation}\label{pri:3.14}
\Big(\int_{\Sigma_{3\varepsilon}}|\nabla^2 u_0|^{p}dx\Big)^{1/p}
\lesssim \varepsilon^{1/p-1/q-1}\Big(\int_{\Omega}|\nabla u_0|^{q}dx\Big)^{1/q}
\end{equation}
for the case $1<p\leq 2$.
Here the up to constant depends on $\tilde{\mu}_0,\tilde{\mu}_1,p,d$ and $r_0$.
\end{lemma}

\begin{proof}
These estimates have already been appeared in the linear case, which are based upon
Mayer's estimate and Lemma $\ref{lemma:3.3}$. It follows from H\"older's inequality  that
\begin{equation*}
\int_{\Omega\setminus\Sigma_{8\varepsilon}}|\nabla u_0|^p dx
\lesssim |\Omega\setminus\Sigma_{8\varepsilon}|^{1-\frac{p}{q}}
\Big(\int_{\Omega}|\nabla u_0|^qdx\Big)^{\frac{p}{q}}
\end{equation*}
and this gives the estimate $\eqref{pri:3.12}$. Since $\delta(x)\approx\delta(y)$, there holds
\begin{equation}\label{f:3.39}
\begin{aligned}
\int_{\Sigma_{3\varepsilon}}|\nabla^2 u_0|^{2}|\nabla u_0|^{p-2} dx
&\lesssim \int_{\Sigma_{3\varepsilon}}
\dashint_{B(x,\delta(x)/4)}|\nabla^2 u_0|^{2}|\nabla u_0|^{p-2} dy dx
\lesssim \int_{\Sigma_{2\varepsilon}}|\nabla^2 u_0|^{2}|\nabla u_0|^{p-2} dx.
\end{aligned}
\end{equation}
Then, we have
\begin{equation*}
\begin{aligned}
\int_{\Sigma_{3\varepsilon}}|\nabla^2 u_0|^{2}|\nabla u_0|^{p-2} dx
&\lesssim \int_{\Sigma_{3\varepsilon}}
\dashint_{B(x,\delta(x)/4)}|\nabla^2 u_0|^{2}|\nabla u_0|^{p-2} dy dx \\
&\lesssim \int_{\Sigma_{3\varepsilon}} [\delta(x)]^{-2}
\dashint_{B(x,\delta(x)/2)}|\nabla u_0|^{p} dy dx
\lesssim \int_{\Sigma_{2\varepsilon}} [\delta(x)]^{-2}|\nabla u_0|^{p}dx\\
&\lesssim \varepsilon^{-1-\frac{p}{q}}\Big(\int_\Omega |\nabla u_0|^q dx\Big)^{\frac{p}{q}},
\end{aligned}
\end{equation*}
where we also use the estimate $\eqref{pri:3.9}$ in the second inequality,
and this will imply the estimate $\eqref{pri:3.13}$. Using the same argument
as in the proof above, we may obtain the estimate $\eqref{pri:3.14}$,
which is based upon $\eqref{pri:3.10}$. We have completed the proof.
\end{proof}

\begin{lemma}
Assume the same conditions as in Lemma $\ref{lemma:2.9}$.
Let $\varphi_\varepsilon = S_{\varepsilon}(\psi_{4\varepsilon}\nabla u_0)$,
and $\beta = 1/p-1/q$ with some $q>p$, then we have
\begin{equation}\label{pri:3.11}
\begin{aligned}
&\Big(\int_\Omega\big|D^h_i(\varepsilon N(y,\varphi_\varepsilon))
-\partial_iN(y,\varphi_\varepsilon)\big|^pdx\Big)^{\frac{1}{p}} \\
&\lesssim \left\{\begin{aligned}
&\big\{\varepsilon^{\theta(\tau-1)} %r_0^{d\beta}
+\varepsilon^{(1-\alpha)(1-\tau)+\beta}\big\}\Big(\int_{\Omega}|\nabla u_0|^qdx\Big)^{\frac{1}{q}}
&\text{if~~}& 2< p<\infty;\\
&\big\{\varepsilon^{\theta(\tau-1)} %r_0^{d\beta}
+\varepsilon^{(1-\alpha)(1-\tau)+\alpha\beta}\big\}  %r_0^{d\beta(p-2)\alpha}
\Big(\int_{\Omega}|\nabla u_0|^qdx\Big)^{\frac{1}{q}} &\text{if~~}& 1<p<2,
\end{aligned}\right.
\end{aligned}
\end{equation}
where $1<\tau<(1+\beta-\alpha)/(1-\alpha)$ if $2<p<\infty$, and
$1<\tau<(1+\alpha\beta-\alpha)/(1-\alpha)$ if $1<p<2$. Here the
up to multiplicative constant will additionally depend on $\tilde{\mu}_0, \tilde{\mu}_1$ and $r_0$.
\end{lemma}

\begin{proof}
We first prove the first line of the estimate $\eqref{pri:3.11}$, and it follows from
the estimate $\eqref{pri:2.15}$ that
\begin{equation*}
\begin{aligned}
&\quad\Big(\int_\Omega\big|D^h_i(\varepsilon N(y,\varphi_\varepsilon))
-\partial_iN(y,\varphi_\varepsilon)\big|^pdx\Big)^{\frac{1}{p}} \\
&\lesssim \varepsilon^{\theta(\tau-1)}
\Big(\int_{\Omega}|\nabla u_0|^p dx\Big)^{\frac{1}{p}}
+ \varepsilon^{1-(1-\alpha)\tau}
\Big(\int_{\Sigma_{4\varepsilon}}|\nabla u_0|^{p-2}|D_i^h(\psi_{4\varepsilon}\nabla u_0)|^2
dx\Big)^{\frac{1}{p}},
\end{aligned}
\end{equation*}
where we use the facts $\eqref{f:3.13}$, and
\begin{equation*}
 D_i^h\big[S_{\varepsilon}(\psi_{4\varepsilon}\nabla u_0)\big]
 = S_\varepsilon(D_i^h(\psi_{4\varepsilon}\nabla u_0)).
\end{equation*}
Then we turn to the second term in the right-hand side above, and
\begin{equation*}
\begin{aligned}
\int_{\Sigma_{4\varepsilon}}|\nabla u_0|^{p-2}|D_i^h(\psi_{4\varepsilon}\nabla u_0)|^2
dx
&= \int_{\Sigma_{4\varepsilon}}|\nabla u_0|^{p-2}
|D_i^h(\psi_{4\varepsilon})\nabla u_0
+ \psi_{4\varepsilon}(x+the_i)D_i^h(\nabla u_0)|^2dx\\
&\lesssim
\underbrace{\varepsilon^{-2} \int_{\Omega\setminus\Sigma_{9\varepsilon}}|\nabla u_0|^p dx}_{I_1}
+ \underbrace{\int_{\Sigma_{3\varepsilon}}|\nabla u_0|^2|D_i^h(\nabla u_0)|^2 dx}_{I_2}
\end{aligned}
\end{equation*}
by noting that
\begin{equation*}
|D_i^h(\psi_{4\varepsilon})(x)| \leq \int_0^1|\nabla\psi_{4\varepsilon}(x+the_i)|dt
\leq C/\varepsilon,
\qquad\text{and}\qquad \text{supp}\big(D_i^h(\psi_{4\varepsilon})\big)
\subset \Sigma_{3\varepsilon}\setminus\Sigma_{9\varepsilon}.
\end{equation*}

Using the same argument as in the proof of Lemma $\ref{lemma:3.7}$, we can easily carry out
the proof this theorem, and we just mention that the estimates $\eqref{pri:3.9}$ and
$\eqref{pri:3.10}$ are in fact proved by an analysis of certain difference quotients
(see for example \cite[Theorem 8.1]{GE}).
In terms of $I_1$, it follows from H\"older's inequality that
\begin{equation*}
\begin{aligned}
I_1 &\leq \varepsilon^{-2}|\Omega\setminus\Sigma_{9\varepsilon}|^{1-\frac{p}{q}}
\Big(\int_{\Omega}|\nabla u_0|^q\Big)^{\frac{p}{q}}
\lesssim \varepsilon^{-1-\frac{p}{q}}\Big(\int_{\Omega}|\nabla u_0|^q\Big)^{\frac{p}{q}},
\end{aligned}
\end{equation*}
and we proceed to show
\begin{equation*}
\begin{aligned}
I_2
\lesssim \int_{\Sigma_{3\varepsilon}}
\dashint_{B(x,\delta(x)/4)}|\nabla u_0|^{p-2}|D_i^h(\nabla u_0)|^2 dy dx
&\lesssim \int_{\Sigma_{3\varepsilon}}[\delta(x)]^{-2}
\dashint_{B(x,\delta(x)/2)}|\nabla u_0|^p dy dx\\
&\lesssim \int_{\Sigma_{2\varepsilon}}[\delta(x)]^{-2}|\nabla u_0|^p dx
\lesssim \varepsilon^{-1-\frac{p}{q}}\Big(\int_{\Omega}|\nabla u_0|^q\Big)^{\frac{p}{q}}
\end{aligned}
\end{equation*}
where the first and third inequalities are due to the fact $\eqref{f:3.39}$, and the second one
is based upon the estimate $\eqref{pri:3.9}$. Clearly, the above two estimates give
\begin{equation}\label{f:3.21}
\begin{aligned}
\Big(\int_{\Sigma_{4\varepsilon}}|\nabla u_0|^{p-2}|D_i^h(\psi_{4\varepsilon}\nabla u_0)|^2
dx\Big)^{\frac{1}{p}}
\lesssim \varepsilon^{-\frac{1}{p}-\frac{1}{q}}\Big(\int_\Omega|\nabla u_0|^qdx\Big)^{1/q}.
\end{aligned}
\end{equation}
Then by noting that $-(1/p+1/q) = -\alpha + \beta$, we have proved
\begin{equation*}
\begin{aligned}
\Big(\int_\Omega\big|D^h_i(\varepsilon N(y,\varphi_\varepsilon))
-\partial_iN(y,\varphi_\varepsilon)\big|^pdx\Big)^{\frac{1}{p}}
\lesssim \big\{\varepsilon^{\theta(\tau-1)}
+ \varepsilon^{(1-\alpha)(1-\tau)+\beta}\big\}
\Big(\int_{\Omega}|\nabla u_0|^q dx\Big)^{\frac{1}{q}}
\end{aligned}
\end{equation*}
in the case of $2\leq p<\infty$. This requires the condition
$1<\tau<(1+\beta-\alpha)/(1-\alpha)$.
We proceed to show the second line of $\eqref{pri:3.11}$, and
it follows from the estimate $\eqref{pri:2.15}$ that
\begin{equation*}
\begin{aligned}
&\Big(\int_\Omega\big|D^h_i(\varepsilon N(y,\varphi_\varepsilon))
-\partial_iN(y,\varphi_\varepsilon)\big|^pdx\Big)^{\frac{1}{p}} \\
&\lesssim \varepsilon^{\theta(\tau-1)}
\Big(\int_{\Omega}|\nabla u_0|^p dx\Big)^{\frac{1}{p}}
+ \varepsilon^{1-(1-\alpha)\tau}
\Big(\int_{\Sigma_{4\varepsilon}}|\nabla u_0|^{\frac{p(2-p)}{3-p}}
|D_i^h(\psi_{4\varepsilon}\nabla u_0)|^{\frac{p}{3-p}}
dx\Big)^{\frac{1}{p}}\\
&\lesssim \varepsilon^{\theta(\tau-1)}
\Big(\int_{\Omega}|\nabla u_0|^q dx\Big)^{\frac{1}{q}}
+ \varepsilon^{1-(1-\alpha)\tau}
\Big(\int_{\Omega}|\nabla u_0|^{p}dx\Big)^{\frac{2-p}{p(3-p)}}
\Big(\int_{\Sigma_{3\varepsilon}}
|D_i^h(\psi_{4\varepsilon}\nabla u_0)|^{p}
dx\Big)^{\frac{1}{p(3-p)}}.
\end{aligned}
\end{equation*}
By a similar argument, there holds
\begin{equation}\label{f:3.22}
\begin{aligned}
\Big(\int_{\Sigma_{3\varepsilon}}
|D_i^h(\psi_{4\varepsilon}\nabla u_0)|^{p}
dx\Big)^{\frac{1}{p(3-p)}}
&\lesssim \bigg\{
\varepsilon^{-\frac{1}{3-p}}
\Big(\int_{\Omega\setminus\Sigma_{9\varepsilon}}|\nabla u_0|^{p}dx\Big)^{\frac{1}{p(3-p)}}
+\Big(\int_{\Sigma_{3\varepsilon}}|D_i^h(\nabla u_0)|^pdx\Big)^{\frac{1}{p(3-p)}}\bigg\}\\
&\lesssim \varepsilon^{(1/p-1/q-1)/(3-p)}
\Big(\int_{\Omega}|\nabla u_0|^qdx\Big)^{\frac{1}{q(3-p)}},
\end{aligned}
\end{equation}
where we use the estimate $\eqref{pri:3.10}$ in the second inequality.
Noting that $(1/p-1/q-1)/(3-p) = \alpha(\beta-1)$, we consequently obtain
\begin{equation*}
\Big(\int_\Omega\big|D^h_i(\varepsilon N(y,\varphi_\varepsilon))
-\partial_iN(y,\varphi_\varepsilon)\big|^pdx\Big)^{\frac{1}{p}}
\lesssim \big\{\varepsilon^{\theta(\tau-1)}
+\varepsilon^{(1-\alpha)(1-\tau)+\alpha\beta}\big\}
\Big(\int_{\Omega}|\nabla u_0|^qdx\Big)^{\frac{1}{q}}
\end{equation*}
with $1<\tau<\frac{1-\alpha+\alpha\beta}{1-\alpha}$, where $\beta=1/p-1/q$.
We have completed the proof.
\end{proof}

\begin{lemma}
Assume the same conditions as in Lemma $\ref{lemma:2.8}$.
Let $\varphi_\varepsilon = S_{\varepsilon}(\psi_{4\varepsilon}\nabla u_0)$,
and $\beta = 1/p-1/q$ with some $q>p$, then we have
\begin{equation}\label{pri:3.15}
\begin{aligned}
&\quad\Big(\int_\Omega\big|D^{-h}_j(\varepsilon E_{ji}(y,\varphi_\varepsilon))
-\partial_jE_{ji}(y,\varphi_\varepsilon)\big|^{\frac{p}{p-1}}dx\Big)^{\frac{p-1}{p}} \\
&\lesssim \left\{\begin{aligned}
&\big\{\varepsilon^{\rho(\tau-1)} + \varepsilon^{(1-\gamma)(1-\tau)+\beta}\big\}
\Big(\int_{\Omega}|\nabla u_0|^qdx\Big)^{\frac{p-1}{q}}
&\text{if~}& 2< p<\infty;\\
&\big\{\varepsilon^{\rho(\tau-1)}+\varepsilon^{(1-\gamma)(1-\tau)+\gamma\beta}\big\}
\Big(\int_{\Omega}|\nabla u_0|^qdx\Big)^{\frac{p-1}{q}} &\text{if~}& 1<p<2,
\end{aligned}\right.
\end{aligned}
\end{equation}
where $1<\tau<(1-\gamma+\beta)/(1-\gamma)$ if $2<p<\infty$, and
$1<\tau<(1-\gamma+\gamma\beta)/(1-\gamma)$ if $1<p<2$. Here the up to
constant depends on $\tilde{\mu}_0, \tilde{\mu}_1, \mu_0, \mu_1, \mu_2, p, d$ and $r_0$.
\end{lemma}

\begin{proof}
The proof is similar to that given in the previous lemma. and we omit it.
\end{proof}

\begin{remark}
\emph{Let $p=2$,
and $\beta = 1/2-1/q$ with some $q>2$. Then we have
\begin{equation}\label{pri:3.17}
\begin{aligned}
\int_{\Omega}\Big|\frac{\partial}{\partial\xi_k}
\big[E_{ji}(x/\varepsilon,\varphi_\varepsilon)\big]\frac{\partial}{\partial x_j}
\big[(\varphi_{\varepsilon})_k\big]K_i\Big|dx
\lesssim
\varepsilon^{-\frac{1}{2}(1+\frac{2}{q})} \Big(\int_{\Omega}|\nabla u_0|^qdx\Big)^{1/q}
\Big(\int_{\Omega}|K|^2dx\Big)^{1/2},
\end{aligned}
\end{equation}
where we just note that $|\nabla_{\xi_k} E_{ji}(y,\cdot)|\leq C$ for a.e. $y\in\mathbb{R}^d$
by the estimate $\eqref{pri:2.5}$.
The same result might be found in \cite[Theorem 3.4]{WXZ1}.}
\end{remark}
\subsection{Proof of Theorem $\ref{thm:1.1}$}
The estimate $\eqref{pri:1.3}$ has already been proved in Lemma $\ref{lemma:2.6}$.
Here we just focus on the error estimates stated in the theorem.
For the ease of the statement. Let $K = \nabla u_\varepsilon - V$ as in $\eqref{eq:3.4}$, and
$\tilde{K} = K/\|K\|_{L^p(\Omega)}$.
The proof is divided into two parts. One is to handle the singular case ($1<p<2$). The other is
to deal with the degenerated case ($2<p<\infty$).
%Before carrying out a proof,
%we mention that the following four parts that we will address in fact obey
%an analogy argument as we will do for Part 1.
%In order to make the proof concise,
%we just focus on its differences in the proof.
%The important thing is to figure out the concrete range of $\tau$ in different cases.

Part 1 ($1<p<2$). We show the estimate $\eqref{pri:1.5}$, and the proof mainly consists of three steps
according to Lemma $\ref{lemma:3.4}$.
The first step is to give the estimate
\begin{equation}\label{f:3.25}
\begin{aligned}
&\quad\int_{\Omega}|\widehat{A}(\nabla u_0) - \widehat{A}(\varphi_\varepsilon)||\tilde{K}|dx \\
&\lesssim |\Omega_{M,\varepsilon}|^{\beta(p-1)}\|\nabla u_0\|_{L^q(\Omega_{M,\varepsilon})}^{p-1}
+ \Big\{\varepsilon^{\beta(p-1)} + \varepsilon^{\beta\gamma}\Big\}
\|\nabla u_0\|_{L^q(\Omega)}^{p-1}\\
&\lesssim |\Omega_{M,\varepsilon}|^{\beta(p-1)}
\|\nabla u_0\|_{L^q(\Omega_{M,\varepsilon})}^{p-1}
+ \varepsilon^{\gamma\beta}\|\nabla u_0\|_{L^q(\Omega)}^{p-1},
\end{aligned}
\end{equation}
whose first inequality follows from the second line of $\eqref{pri:3.6}$ coupled with
the layer and co-layer estimates $\eqref{pri:3.12}$ and
$\eqref{pri:3.14}$.

Then the second step is to show
\begin{equation}\label{f:3.26}
\begin{aligned}
\bigg|\int_{\Omega}\big[\widehat{A}(\varphi_\varepsilon)
-A(y,P(y,\varphi_\varepsilon))\big]\cdot \tilde{K}dx\bigg|
&\lesssim \Big\{\varepsilon^{\rho(\tau-1)}
+\varepsilon^{(1-\gamma)(1-\tau)+\gamma\beta}\Big\}
\|\nabla u_0\|_{L^q(\Omega)}^{p-1}\\
&\lesssim \max \Big\{\varepsilon^{\rho(\tau-1)},
\varepsilon^{(1-\gamma)(1-\tau)+\gamma\beta}\Big\}
\|\nabla u_0\|_{L^q(\Omega)}^{p-1}
\end{aligned}
\end{equation}
where $\rho = \theta(p-1)$ (defined in Lemma $\ref{lemma:2.3}$), and we employ the estimates $\eqref{pri:3.5}$ and $\eqref{pri:3.15}$.

In the third step,
combining the estimates $\eqref{pri:3.8}$ and $\eqref{pri:3.11}$ together with
$\eqref{pri:3.12}$ and $\eqref{pri:3.14}$ leads to
\begin{equation}\label{f:3.27}
\begin{aligned}
&\bigg|\int_\Omega
\big[A(y,P(y,\varphi_\varepsilon))
-A(y,V)\big]\cdot \tilde{K}dx\bigg|\\
&\lesssim
\Big\{\varepsilon^{\beta(p-1)} + \varepsilon^{\theta(\tau-1)(p-1)}
+\varepsilon^{(1-\alpha)(1-\tau)(p-1)+\gamma\beta}\Big\}
\|\nabla u_0\|_{L^q(\Omega)}^{p-1} \\
&\lesssim
\max\Big\{\varepsilon^{(1-\alpha)(1-\tau)(p-1)+\gamma\beta},
\varepsilon^{\theta(\tau-1)(p-1)}\Big\}
\|\nabla u_0\|_{L^q(\Omega)}^{p-1},
\end{aligned}
\end{equation}
in which
the second inequality is due to the fact $\beta(p-1) > (1-\alpha)(1-\tau)(p-1)
+\gamma\beta$.

Hence, in view of the estimates $\eqref{f:3.26}$ and $\eqref{f:3.27}$,
the exact range of $\tau$ might be figured out by the following three cases:
\begin{enumerate}
  \item[(a)] \quad$0<\rho(\tau-1)< (1-\gamma)(1-\tau)+\gamma\beta< (1-\alpha)(1-\tau)(p-1)+\gamma\beta$;
  \item[(b)] \quad$0<(1-\gamma)(1-\tau)+\gamma\beta<\rho(\tau-1)< (1-\alpha)(1-\tau)(p-1)+\gamma\beta$;
  \item[(c)] \quad$0<(1-\gamma)(1-\tau)+\gamma\beta< (1-\alpha)(1-\tau)(p-1)+\gamma\beta<\rho(\tau-1)$.
\end{enumerate}
In the case (a), it follows from $\eqref{f:3.25}$, $\eqref{f:3.26}$,
$\eqref{f:3.27}$ and $\eqref{pri:3.16}$ that
\begin{equation*}
 \|K\|_{L^p(\Omega)}^{p-1}
\leq |\Omega_{M,\varepsilon}|^{\beta(p-1)}
\|\nabla u_0\|_{L^q(\Omega_{M,\varepsilon})}^{p-1}
+ \varepsilon^{\theta(1-\tau)(p-1)}\|\nabla u_0\|_{L^q(\Omega)}^{p-1},
\end{equation*}
and this gives the stated estimate $\eqref{pri:1.5}$. Meanwhile, it is not hard to see
that the range $\eqref{range:1}$ follows from the relationship (a).
By the same token, we can derive the estimate $\eqref{pri:1.4}$ on account of the cases (b) and (c).
However, as stated in $\eqref{range:2}$,
we have to pay more careful on the computations for the corresponding range of $\tau$.
From (b), it follows that
\begin{equation}\label{f:3.37}
%\left\{\begin{aligned}
%&\frac{1-\gamma + \gamma\beta + \rho}{1-\gamma+\rho}
%< \tau < \frac{1-\gamma+\gamma\beta}{1-\gamma};\\
%&1<\tau <\frac{1-\alpha+\alpha\beta+\theta}
%{1-\alpha+\theta},
%\end{aligned}\right.
%\Rightarrow
\frac{1-\gamma + \gamma\beta + \rho}{1-\gamma+\rho}
< \tau< \min\Big\{\frac{1-\alpha+\alpha\beta+\theta}
{1-\alpha+\theta},\frac{1-\gamma+\gamma\beta}{1-\gamma}\Big\},
\end{equation}
and keep the above inequality for a while. In terms of (c), we may have
\begin{equation}\label{f:3.38}
\left\{\begin{aligned}
&1
< \tau < \frac{1-\gamma+\gamma\beta}{1-\gamma};\\
&\frac{1-\alpha+\alpha\beta+\theta}
{1-\alpha+\theta}<\tau <\frac{1-\alpha+\alpha\beta}
{1-\alpha},
\end{aligned}\right.
\quad\Longrightarrow\quad
\frac{1-\alpha+\alpha\beta+\theta}
{1-\alpha+\theta}
< \tau< \frac{1-\gamma+\gamma\beta}{1-\gamma},
\end{equation}
under the condition $\frac{2+\delta}{1+\delta}<p<2$, since such the condition guarantees
the following inequality
\begin{equation*}
\frac{1-\alpha+\alpha\beta+\theta}
{1-\alpha+\theta}
< \frac{1-\gamma+\gamma\beta}{1-\gamma}.
\end{equation*}
Recalling $\eqref{f:3.37}$, we have the range
\begin{equation*}
\frac{1-\gamma + \gamma\beta + \rho}{1-\gamma+\rho}
< \tau<  \frac{1-\alpha+\alpha\beta+\theta}
{1-\alpha+\theta}.
\end{equation*}
under the same condition, and this together with $\eqref{f:3.38}$ gives
$\eqref{range:2}$. For the case $1<p\leq \frac{2+\delta}{1+\delta}$,
the stated range $\eqref{range:2}$ also follows from $\eqref{f:3.37}$, and in fact
(c) can not happen in such the case. The proof of Part 1 is done.

%\begin{equation*}
%\begin{aligned}
%0<\rho(\tau-1)\leq (1-\gamma)(1-\tau)+\gamma\beta  ~&\Longleftrightarrow~
%1< \tau < \frac{1-\gamma + \gamma\beta + \rho}{1-\gamma+\rho};\\
%0<\theta(\tau-1)\leq (1-\alpha)(1-\tau)+\alpha\beta ~&\Longleftrightarrow~
%1<\tau
%<\frac{1-\alpha+\alpha\beta+\theta}
%{1-\alpha+\theta},
%\end{aligned}
%\end{equation*}
%and
%
%\begin{equation*}
%\begin{aligned}
% 0<(1-\gamma)(1-\tau)+\gamma\beta <\rho(\tau-1) ~&\Longleftrightarrow~
%\frac{1-\gamma + \gamma\beta + \rho}{1-\gamma+\rho}
%< \tau < \frac{1-\gamma + \gamma\beta}{1-\gamma}\\
% 0<(1-\alpha)(1-\tau)+\alpha\beta <\theta(\tau-1) ~&\Longleftrightarrow~
%\frac{1-\alpha+\alpha\beta+\theta}
%{1-\alpha+\theta}<\tau
%<\frac{1-\alpha+\alpha\beta}{1-\alpha}
%\end{aligned}
%\end{equation*}

Part 2. In this part ($2<p<\infty$), we will address the estimate $\eqref{pri:1.6}$ and
the range $\eqref{range:3}$ of $\tau$.  We first show that
\begin{equation}\label{f:3.29}
\begin{aligned}
 \int_{\Omega}|\widehat{A}(\nabla u_0) - \widehat{A}(\varphi_\varepsilon)||\tilde{K}|dx
&\lesssim |\Omega_{M,\varepsilon}|^{\beta(p-1)}
\|\nabla u_0\|_{L^q(\Omega_{M,\varepsilon})}^{p-1}
+ \big\{\varepsilon^{\beta(p-1)}+\varepsilon^{\beta/\alpha}\big\}
\|\nabla u_0\|_{L^q(\Omega)}^{p-1}\\
&\lesssim |\Omega_{M,\varepsilon}|^{\beta(p-1)}
\|\nabla u_0\|_{L^q(\Omega_{M,\varepsilon})}^{p-1}
+ \varepsilon^{\beta/\alpha}
\|\nabla u_0\|_{L^q(\Omega)}^{p-1}
\end{aligned}
\end{equation}
where we use the first line of $\eqref{pri:3.6}$ coupled with the layer and co-layer type estimates
$\eqref{pri:3.12}$,$\eqref{pri:3.13}$.
Then, we employ
the first line of $\eqref{pri:3.5}$ and  $\eqref{pri:3.15}$ to obtain
\begin{equation}\label{f:3.41}
\begin{aligned}
&\quad\bigg|\int_{\Omega}\big[\widehat{A}(\varphi_\varepsilon)
-A(y,P(y,\varphi_\varepsilon))\big]\cdot \tilde{K}dx\bigg|\\
&\lesssim \varepsilon^{\theta(\tau-1)}
\Big(\int_{\Omega}|\nabla u_0|^qdx\Big)^{\frac{p-1}{q}}
+\varepsilon^{(1-\gamma)(1-\tau)+\beta}\Big(\int_{\Omega}|\nabla u_0|^qdx\Big)^{\frac{p-1}{q}}\\
&\lesssim \varepsilon^{(1-\gamma)(1-\tau)+\beta}
\Big(\int_{\Omega}|\nabla u_0|^qdx\Big)^{\frac{p-1}{q}}
\end{aligned}
\end{equation}
under the condition $\eqref{f:3.32}$, where we note that $\gamma = \alpha$ in the case of $p>2$.
Furthermore, let
\begin{equation*}
 R_{M,\varepsilon} := |\Omega_{M,\varepsilon}|^{\beta(p-1)}
\|\nabla u_0\|_{L^q(\Omega_{M,\varepsilon})}^{p-1},
\end{equation*}
and it follows from the estimates $\eqref{pri:3.7}$  and $\eqref{pri:3.11}$ that
\begin{equation}\label{f:3.28}
\begin{aligned}
&\quad\bigg|\int_\Omega
\big[A(y,P(y,\varphi_\varepsilon))
-A(y,V)\big]\cdot \tilde{K}dx\bigg|\\
&\lesssim  \Big\{\varepsilon^{\beta/\alpha}+ \varepsilon^{1-\alpha+\beta}
+ \varepsilon^{(p-2)\theta(\tau-1)}
+\varepsilon^{(p-2)[(1-\alpha)(1-\tau)+\beta]}\Big\}
\|\nabla u_0\|_{L^q(\Omega)}^{p-1} + R_{M,\varepsilon}\\
&\lesssim \Big\{\varepsilon^{\beta/\alpha} + \varepsilon^{1-\alpha+\beta}
+\varepsilon^{(p-2)[(1-\alpha)(1-\tau)+\beta]}\Big\}
\|\nabla u_0\|_{L^q(\Omega)}^{p-1} +R_{M,\varepsilon}\\
&\lesssim
\varepsilon^{[(1-\alpha)(1-\tau)+\beta](p-2)}
\|\nabla u_0\|_{L^q(\Omega)}^{p-1} + R_{M,\varepsilon},
\end{aligned}
\end{equation}
where we also use the estimates $\eqref{pri:3.12}$,$\eqref{pri:3.13}$ in the first inequality, and
the fact that
\begin{equation}\label{f:3.32}
  0<(1-\alpha)(1-\tau)+\beta <\theta(\tau-1) \quad\Longleftrightarrow\quad
\frac{1-\alpha+\beta+\theta}{1-\alpha+\theta}<\tau
<\frac{1-\alpha+\beta}{1-\alpha}
\end{equation}
in the second one.
The last inequality of $\eqref{f:3.28}$ is due to
\begin{equation*}
[(1-\alpha)(1-\tau)+\beta](p-2)<\beta/\alpha < 1-\alpha+\beta
\end{equation*}
whose second inequality follows from the fact $0<\beta<\alpha$. Here we also point out that
the estimate $\eqref{f:3.28}$ merely holds for the case of $2<p<3$. Thus, the estimate
$\eqref{pri:3.16}$ together with the estimates $\eqref{f:3.29}$, $\eqref{f:3.41}$ and $\eqref{f:3.28}$
implies that
\begin{equation*}
 \|K\|_{L^p(\Omega)}
\leq \varepsilon^{(\frac{p-2}{p-1})[(1-\alpha)(1-\tau)+\beta]}
\|\nabla u_0\|_{L^q(\Omega)} + |\Omega_{M,\varepsilon}|^{\beta}
\|\nabla u_0\|_{L^q(\Omega_{M,\varepsilon})},
\end{equation*}
which is exactly the first line of $\eqref{pri:1.6}$. Concerning the case $p\geq 3$,
we merely make a few modifications on the estimate $\eqref{f:3.28}$,
which is
\begin{equation}\label{f:3.33}
\begin{aligned}
&\quad\bigg|\int_\Omega
\big[A(y,P(y,\varphi_\varepsilon))
-A(y,V)\big]\cdot \tilde{K}dx\bigg|\\
&\lesssim
\Big\{\varepsilon^{\beta/\alpha} + \varepsilon^{1-\alpha+\beta}
+\varepsilon^{\theta(\tau-1)}
+\varepsilon^{(1-\alpha)(1-\tau)+\beta}\Big\}
\|\nabla u_0\|_{L^q(\Omega)}^{p-1} + R_{M,\varepsilon} \\
&\lesssim\varepsilon^{(1-\alpha)(1-\tau)+\beta}
\|\nabla u_0\|_{L^q(\Omega)}^{p-1} + R_{M,\varepsilon},
\end{aligned}
\end{equation}
where we use the estimates $\eqref{pri:3.7}$ and $\eqref{pri:3.11}$ in the case of
$p\geq 3$, coupled with the layer and co-layer type estimates $\eqref{pri:3.12}$ and
$\eqref{pri:3.13}$. Also, the condition $\eqref{f:3.32}$ and the fact $0<\beta<\alpha$ guarantee
the second inequality of $\eqref{f:3.33}$. Thus, combining the estimates
$\eqref{pri:3.16}$, $\eqref{f:3.33}$ and
$\eqref{f:3.29}$ leads to the second line of the stated estimate $\eqref{pri:1.6}$. We end
this part by mention that $\eqref{f:3.32}$ is exactly the stated range $\eqref{range:3}$.

In the end, we address the special case $p=2$, and proceed as in the proof of Part 1.
The first step is given by
\begin{equation*}
\begin{aligned}
\int_{\Omega}|\widehat{A}(\nabla u_0) - \widehat{A}(\varphi_\varepsilon)||\tilde{K}|dx
\lesssim
\varepsilon^{\beta}\|\nabla u_0\|_{L^q(\Omega)},
\end{aligned}
\end{equation*}
where we use the first line of the estimate $\eqref{pri:3.6}$ coupled with
the estimates $\eqref{pri:3.12}$ and $\eqref{pri:3.13}$. Then, it follows from
the estimates $\eqref{pri:3.5}$ and $\eqref{pri:3.17}$ that
\begin{equation*}
\bigg|\int_{\Omega} \big[\widehat{A}(\varphi_\varepsilon)
- A(y,P(y,\varphi_\varepsilon))\big]
\cdot \tilde{K} dx\bigg| \lesssim \varepsilon^{\beta}\|\nabla u_0\|_{L^q(\Omega)}.
\end{equation*}
Also, in view of the estimates $\eqref{pri:3.8}$ and $\eqref{pri:2.20}$, we consequently have
\begin{equation*}
\begin{aligned}
\quad\bigg|\int_\Omega
\big[A(y,P(y,\varphi_\varepsilon))
-A(y,V)\big]\cdot \tilde{K}dx\bigg|
\lesssim\varepsilon^{\beta}
\|\nabla u_0\|_{L^q(\Omega)}^{p-1},
\end{aligned}
\end{equation*}
and the above three estimates together with Lemma $\ref{lemma:3.4}$ gives the desired estimate.
We end the proof by mention that, in such the case $p=2$, the size of $h$
might be independent of $\varepsilon$.
\qed

%\begin{corollary}
%Assume the same conditions as in Theorem $\ref{thm:1.1}$.
%Let $r_0$ be the radius of $\Omega$ with $0<r_0<1$. Then we have
%\begin{equation}
%\|u_\varepsilon - u_0\|_{L^p(\Omega)}
%\lesssim r_0^{1+d\beta}\Big\{|\Omega_{M,{\varepsilon/r_0}}|^\beta
%+ (\varepsilon/r_0)^\eta\Big\}\|\nabla u_0\|_{L^q(\Omega)},
%\end{equation}
%where the H\"older index $\eta$ satisfies
%\begin{equation}
%\eta=\left\{\begin{aligned}
%&\alpha\beta+(1-\gamma)(1-\tau)/(p-1) &\text{~if~~}& 1<p\leq 2;\\
%&(p-2)[\beta+(1-\alpha)(1-\tau)]/(p-1) &\text{~if~~}& 2<p\leq 3;\\
%&[\beta+(1-\alpha)(1-\tau)]/(p-1)  &\text{~if~~}& 3\leq p< 4;
%\end{aligned}\right.
%\end{equation}
%and $\tau$ will be fixed in $\eqref{range:2}$ and $\eqref{range:3}$, respectively.
%\end{corollary}
%
%\begin{proof}
%\begin{equation}
%\|u_\varepsilon - u_0\|_{L^p(\Omega)}
%\leq \|\nabla u_\varepsilon - V\|_{L^p(\Omega)}
%+ \|D^h[\varepsilon N(x/\varepsilon,\varphi_\varepsilon)]\|_{W^{-1,p}(\Omega)}.
%\end{equation}
%\end{proof}

\section{Large scale estimates}\label{sec:4}

\subsection{An approximating argument}

\begin{lemma}[Caccioppoli's inequality]\label{lemma:4.2}
Let $1<p<\infty$. Suppose $A$ satisfies $\eqref{a:1}$ and $\eqref{c:1}$, and
$u_\varepsilon\in W^{1,p}_{loc}(\Omega)$ is a weak solution of
$\mathcal{L}_\varepsilon u_\varepsilon = 0$ in $\Omega$, then for any
$B(x,r)\subset\Omega$, there holds
\begin{equation}\label{pri:4.2}
 \Big(\dashint_{B(x,r)}|\nabla u_\varepsilon|^p\Big)^{1/p}
\leq \frac{C}{r}\Big(\dashint_{B(x,2r)}|u_\varepsilon-c|^p\Big)^{1/p}
\end{equation}
for any constant $c\in\mathbb{R}$, where $C$ depends on $\mu_0,\mu_1,p$ and $d$.
\end{lemma}

\begin{proof}
On account of the assumption $\eqref{a:1}$ and $\eqref{c:1}$, there hold
\begin{equation}\label{f:4.7}
 \big<A(y,\xi),\xi\big>\geq \mu_0|\xi|^p;
\qquad
\big|A(y,\xi)\big|\leq \mu_1|\xi|^{p-1}
\qquad \forall y,\xi\in\mathbb{R}^d.
\end{equation}
By definition of the weak solution, we have
\begin{equation*}
  \int_{\Omega}A(x/\varepsilon,\nabla u_\varepsilon)\nabla\varphi dx = 0
\qquad\quad\forall \varphi\in W_0^{1,p}(\Omega).
\end{equation*}
Let $\varphi = \psi_{r}^p(u_\varepsilon-c)$, where $\psi_r\in C_0^1(\Omega)$ is a cut-off function
satisfying
$\psi_r = 1$ in $B(x,r)$ and $\psi_r = 0$ outside $B(x,2r)$ with $|\nabla \psi_r|\leq C/r$.
Thus, a routine computation leads to
\begin{equation*}
\begin{aligned}
0&=\int_{\Omega}\psi_r^p A(x/\varepsilon,\nabla u_\varepsilon)\nabla u_\varepsilon dx
+ p\int_{\Omega}\psi_r^{p-1}(u_\varepsilon-c)A(x/\varepsilon,\nabla u_\varepsilon)\nabla\psi_r dx\\
&\geq \mu_0\int_{\Omega}\psi_r^p|\nabla u_\varepsilon|^p dx
- \mu_1p\Big(\int_{\Omega}\psi_r^p|\nabla u_\varepsilon|^pdx\Big)^{1-\frac{1}{p}}
\Big(\int_{\Omega}|\nabla\psi_r|^p|u_\varepsilon-c|^pdx\Big)^{\frac{1}{p}}\\
&\geq \frac{\mu_0}{2}\int_{\Omega}\psi_r^p|\nabla u_\varepsilon|^p dx
- C(\mu_0,\mu_1,p,d)\int_{\Omega}|\nabla\psi_r|^p|u_\varepsilon-c|^pdx,
\end{aligned}
\end{equation*}
where we use Young's inequality in the last step.
This implies the stated estimate $\eqref{pri:4.2}$ and we may end the proof here.
\end{proof}

\begin{remark}\label{remark:4.1}
\emph{The above estimate in fact implies the so-called Meyer's estimate.
There exists a small number $\delta\in (0,1)$, such that
\begin{equation}\label{pri:4.3}
 \Big(\dashint_{B(x,r)}|\nabla u_\varepsilon|^{p(1+\delta)}\Big)^{\frac{1}{p(1+\delta)}}
\lesssim \Big(\dashint_{B(x,2r)}|\nabla u_\varepsilon|^{p}\Big)^{\frac{1}{p}}.
\end{equation}
It is also known as a self-improvement property, and relies on the reverse H\"older inequality
developed by M. Giaquinta and G. Modica (see for example \cite[Theorem 6.38]{MGLM}).
We usually choose $q=(1+\delta)p$ in the paper.}

\emph{In particular, it has been known from Lemma $\ref{lemma:2.5}$ that $\widehat{A}$ admits the same property as $\eqref{f:4.7}$.
Given $g\in W^{1-1/q,q}(\partial\Omega)$, suppose
that $u_0\in W^{1,p}(\Omega)$ is a weak solution of $\mathcal{L}_0 u_0 = 0$ in $\Omega$
and $u_0 = g$ on $\partial\Omega$. Then one may have $|\nabla u_0|\in L^q(\Omega)$, and
\begin{equation}\label{f:4.8}
  \|\nabla u_0\|_{L^q(\Omega)}\leq C\|g\|_{W^{1-1/q,q}(\partial\Omega)},
\end{equation}
where $C$ depends on $\mu_0,\mu_1,p,d$ and $\Omega$. We refer the reader
to \cite[Theorem 2.13]{WXZ1} or \cite[Lemma 4.1]{F} for
an analogy argument for the proof, which reduces the problem to establish the corresponding
boundary Caccioppoli's inequality, and the details is left to the reader.}
\end{remark}

\begin{lemma}[approximating lemma]\label{lemma:4.1}
Let $1<p<\infty$, and $B=B(0,R)\subset\Omega$ with $\varepsilon\leq R\leq 1$. Assume the same conditions as in Theorem $\ref{thm:1.2}$.
Suppose that $u_\varepsilon\in W^{1,p}(2B)$
satisfies $\mathcal{L}_\varepsilon u_\varepsilon = 0$ in $2B$.
Then there exist a function $v_R\in W^{1,p}(B)$ satisfying
$\|\nabla v_R\|_{L^{\infty}(B)}\leq M$, and
$\eta\in(0,1)$,
such that
\begin{equation}\label{pri:4.1}
  \Big(\dashint_{B}|u_\varepsilon - v_R|^{p} \Big)^{1/p}
\lesssim \bigg\{\Big(\frac{\varepsilon}{R}\Big)^\eta +
\big|\{x\in B:|\nabla v_R|\leq M(\varepsilon/R)^{\vartheta}\}\big|^{\beta}\bigg\}\Big(\dashint_{2B}|u_\varepsilon|^p\Big)^{1/p},
\end{equation}
where the up-to constant is independent of $\varepsilon$ and $R$.
\end{lemma}

\begin{proof}
The proof is based upon the estimate $\eqref{pri:1.4}$, and there are two steps to complete
the whole argument. Here we set $\eta = \alpha\beta+\frac{(1-\gamma)(1-\tau)}{p-1}$, and
construct $v\in W^{1,p}(B)$ satisfying $\mathcal{L}_0 v_R
= \mathcal{L}_\varepsilon u_\varepsilon =0$ in $B$ with
$v_R=u_\varepsilon$ on $\partial B$. Since $u_\varepsilon\in C^{1,\theta}_{loc}(2B)$, it is
fine to assume that $\|\nabla v_R\|_{L^{\infty}(B)}\leq M$.

Step 1. In the special case of $R=1$, we may show the estimate $\eqref{pri:4.1}$.
For any fixed $i=1,\cdots,d$, there holds
\begin{equation*}
\begin{aligned}
\|u_\varepsilon - v_1\|_{L^p(B)}
&=  \|\nabla_i (u_\varepsilon - v_1)\|_{W^{-1,p}(B)} \\
&\leq \underbrace{\|\nabla_i u_\varepsilon - \nabla_i v_1
- D_i^h(\varepsilon N(\cdot/\varepsilon,\varphi_\varepsilon))\|_{W^{-1,p}(B)}}_{I_1}
+ \underbrace{\|D_i^h(\varepsilon N(y,\varphi_\varepsilon))\|_{W^{-1,p}(B)}}_{I_2},
\end{aligned}
\end{equation*}
where $h = \varepsilon^\tau$ with $\tau$ fixed in $\eqref{range:2}$, and
$\varphi_\varepsilon = S_\varepsilon(\psi_{4\varepsilon}\nabla v_1)$ with $\psi_{4\varepsilon}$
being the corresponding cut-off function.
For the ease of the statement, we denote by
\begin{equation}\label{def:4.1}
\Omega_{M,\varepsilon}^R:=\{x\in B(0,R):|\nabla v_R|\leq M\varepsilon\}
\end{equation}
Thus, it follows from the estimate $\eqref{pri:1.4}$ that
\begin{equation*}
\begin{aligned}
I_1
&\leq \|\nabla u_\varepsilon - \nabla v_1
- D^h(\varepsilon N(\cdot/\varepsilon,\varphi_\varepsilon))\|_{L^{p}(B)} \\
&\lesssim \big\{\varepsilon^\eta + |\Omega_{M,\varepsilon}^1|\big\}\|\nabla v_1\|_{L^q(B)} \\
&\lesssim \big\{\varepsilon^\eta + |\Omega_{M,\varepsilon}^1|\big\}\|u_\varepsilon\|_{W^{1,q}(B)}
\qquad\Big(~\lesssim\big\{\varepsilon^\eta + |\Omega_{M,\varepsilon}^1|\big\}
\|\nabla u_\varepsilon\|_{L^{q}(B)}~\Big) \\
&\lesssim \big\{\varepsilon^\eta + |\Omega_{M,\varepsilon}^1|\big\}
\|\nabla u_\varepsilon\|_{L^{p}(\frac{3}{2}B)}
\end{aligned}
\end{equation*}
where we employ the $L^q$ estimate $\eqref{f:4.8}$ coupled with the trace theorem in the third line.
Note that $u_\varepsilon$ up to a constant is still a solution to
$\mathcal{L}_\varepsilon u_\varepsilon = 0$, which leads to the
estimate in the brace. The fourth line comes from Meyer's estimate $\eqref{pri:4.3}$.
Then we turn to show a computation on $I_2$.
By definition,
\begin{equation*}
\begin{aligned}
I_2 &= \sup_{\|\phi\|_{W^{1,p^\prime}_0(\Omega)}\leq 1}
\Big|\int_{B}\varepsilon N(\cdot/\varepsilon,\varphi_\varepsilon)D^{-h}(\phi)dx\Big| \\
&\leq C\varepsilon\Big(\int_{B}|N(\cdot/\varepsilon,\varphi_\varepsilon)|^pdx\Big)^{1/p}
\|\nabla \phi\|_{L^{p^\prime}(B)}\lesssim \varepsilon\|\nabla v_1\|_{L^p(B)}
\lesssim \varepsilon\|\nabla u_\varepsilon\|_{L^p(B)},
\end{aligned}
\end{equation*}
in which we note that $|D^h(\varepsilon N(\cdot/\varepsilon,\varphi_\varepsilon))| $ is an
integrable function due to the estimate $\eqref{pri:2.8}$ under our assumption on the size of $h$.
The second inequality might be derived from the estimate $\eqref{f:3.40}$.

Hence, combining the above estimates, we have
\begin{equation}\label{f:4.3}
\|u_\varepsilon - v_1\|_{L^p(B)} \lesssim
\big\{\varepsilon^\eta + |\Omega_{M,\varepsilon}^1| + \varepsilon\big\}
\|\nabla u_\varepsilon\|_{L^{p}(\frac{3}{2}B)}
\lesssim
\big\{\varepsilon^\eta + |\Omega_{M,\varepsilon}^1| \big\}
\|u_\varepsilon\|_{L^{p}(2B)}
\end{equation}
where the second step follows from Caccioppoli's inequality $\eqref{pri:4.2}$.

Step 2. We show that the estimate $\eqref{pri:4.1}$ is also valid for $\varepsilon\leq R<1$,
by a rescaling argument. Let $\tilde{u}_\varepsilon(y)
= u_\varepsilon(Ry) = u_\varepsilon(x)$, where $y\in B(0,2)$ and $x\in B(0,2R)$. On account of
the assumption $\eqref{a:1}$, we have
\begin{equation*}
  \text{div}_y A(y/\varepsilon^\prime,\nabla_y \tilde{u}_\varepsilon) = 0 \quad \text{in~~} B(0,2),
\qquad \varepsilon^\prime = \varepsilon/R.
\end{equation*}
We set $v_R(x) = v_1(x/R)$ for any $x\in B(0,R)$, and denote by
$\widehat{B}(\xi) = R^{p-1} \widehat{A}(R^{-1}\xi)$ for any $\xi\in\mathbb{R}^d$. Obviously,
the operator $\widehat{B}$ still satisfies the structure $\eqref{c:2}$. Then we have
\begin{equation*}
 \text{div}_y \widehat{B}(\nabla_y v_1) =
\text{div}_y A(y/\varepsilon^\prime,\nabla_y \tilde{u}_\varepsilon) = 0 \quad \text{in~~} B(0,1)
\qquad v_1 = \tilde{u}_\varepsilon \quad\text{on~}\partial B(0,1),
\end{equation*}
with the condition $\|\nabla_y v_1\|_{L^\infty(B(0,1))}\leq RM$.
By last step,  it follows from the estimate $\eqref{f:4.3}$ that
\begin{equation*}
 \|\tilde{u}_\varepsilon - v_1\|_{L^p(B(0,1))}
\lesssim \Big\{(\varepsilon^\prime)^\eta +
|\Omega_{RM,\varepsilon^\prime}^1|^\beta \Big\}
\|\tilde{u}_\varepsilon\|_{L^{p}(B(0,2))},
\end{equation*}
and this implies
\begin{equation*}
\|u_\varepsilon - v_R\|_{L^p(B(0,R))}
\lesssim \Big\{(\varepsilon/R)^\eta +
\big|\Omega_{M,(\varepsilon/R)}^R\big|^\beta \Big\}
\|u_\varepsilon\|_{L^{p}(B(0,2R))},
\end{equation*}
where we just mention that
$\Omega_{M,(\varepsilon/R)}^R = \{x\in B(0,R):|\nabla v_R|\leq M(\varepsilon/R)^\vartheta\}$ by definition.
Multiplying $R^{-\frac{d}{p}}$ on the both sides above leads to the stated estimate $\eqref{pri:4.1}$, and we have completed
the proof.
\end{proof}

%\begin{lemma}[approximating lemma II]\label{lemma:4.1}
%Let $1<p\leq 2$, and $B=B(0,R)\subset\Omega$ with $\varepsilon\leq R\leq 1$. Assume the same conditions as in Theorem $\ref{thm:1.2}$.
%Suppose that $u_\varepsilon\in W^{1,p}(2B)$
%satisfies $\mathcal{L}_\varepsilon u_\varepsilon = 0$ in $2B$.
%Then there exist a function $v_R\in W^{1,p}(B)$ satisfying
%$\|\nabla v_R\|_{L^{\infty}(B)}\leq M$, and
%$\eta\in(0,1)$,
%such that
%\begin{equation}\label{pri:4.2}
%  \Big(\dashint_{B}|\nabla u_\varepsilon - \nabla v_R - \nabla N(y,\varphi_{\varepsilon})|^{p} \Big)^{1/p}
%\lesssim \bigg\{\Big(\frac{\varepsilon}{R}\Big)^\eta +
%\big|\Omega^R_{M,(\varepsilon/R)}\big|^{\beta}\bigg\}
%\Big(\dashint_{2B}|\nabla u_\varepsilon|^p\Big)^{1/p},
%\end{equation}
%where the up-to constant is independent of $\varepsilon$ and $R$.
%\end{lemma}

%\noindent\textbf{Proof of Theorem $\ref{thm:1.2}$}.
\subsection{Proof of Theorem $\ref{thm:1.2}$}
The main idea may be found in \cite[Theorem 5.2]{S5} or \cite[Theorem 5.4]{WXZ1}, and we make
a modification of the original proof to fit the situation here.
To do so, we set $\sigma\in (0,1)$, and
\begin{equation*}
H(r,\sigma,v) = r^{-\sigma}\bigg\{\Big(\dashint_{B(0,r)}|v|^p\Big)^{1/p}
\bigg\}.
\end{equation*}
For each $r\in[\varepsilon,1/2]$,
let $v_r$ be the function given in Lemma $\ref{lemma:4.1}$. Then,
in terms of the interior Lipschitz  estimate $\eqref{pri:3.18}$, there holds
\begin{equation}\label{f:4.4}
H(\theta r,\sigma,v_r)
\leq C\theta^{1-\sigma} H(r,\sigma,v_r)
\end{equation}
for any $\theta\in(0,1/4)$, where we just use Poincar\'e's inequality and Caccioppoli's inequality
in the computation. Recall the notation
$\Omega_{M,(\varepsilon/r)}^r$ in $\eqref{def:4.1}$, and we note that
\begin{equation*}
  |\Omega_{M,(\varepsilon/r)}^r|\lesssim r^d.
\end{equation*}
Then, by using the estimate $\eqref{pri:4.1}$ we acquire that
\begin{equation}
\begin{aligned}
H(\theta r,\sigma,u_\varepsilon)
&\leq (\theta r)^{-\sigma}\Big(\dashint_{B(0,\theta r)}
|u_\varepsilon - v_r|^p dx\Big)^{\frac{1}{p}}
+ H(\theta r,\sigma,v_r) \\
&\leq C_1\theta^{-\sigma-\frac{d}{p}}\Big\{(\varepsilon/r)^{\eta}+r^{d\beta}\Big\}
H(2r,\sigma,u_\varepsilon)
+ C_2\theta^{1-\sigma} H(r,\sigma,v_r)\\
&\leq \max\bigg\{C_1C_2\theta^{-\sigma-\frac{d}{p}}(\varepsilon/r)^{\eta},
C_1C_2\theta^{-\sigma-\frac{d}{p}}r^{d\beta},
C_2\theta^{1-\sigma}\bigg\}H(2r,\sigma,u_\varepsilon),
\end{aligned}
\end{equation}
where  we use the estimate $\eqref{f:4.4}$
in the second step.
Now, we first fix $\theta\in(0,1/4)$ such that $C_2\theta^{\sigma}\leq 1/4$, and
then let $N\varepsilon\leq r< r_0$. By choosing $N>1$ large and $r_0\in(0,1)$ sufficiently small,
we also obtain
\begin{equation*}
\max\Big\{C_1C_2\theta^{-\sigma-\frac{d}{p}}N^{-\eta},
C_1C_2\theta^{-\sigma-\frac{d}{p}}r_0^{d\beta}\Big\}\leq 1/4.
\end{equation*}
Obviously, we have
\begin{equation}\label{f:4.5}
H(\theta r,\sigma,u_\varepsilon)\leq \frac{1}{4} H(2r,\sigma,u_\varepsilon).
\end{equation}
Moreover, multiplying $r^{-1}$ on the both sides of $\eqref{f:4.5}$ and then integrating
with respect to $r$ from $N\varepsilon$ to $r_0$, we have
\begin{equation*}
\int_{\theta N\varepsilon}^{\theta r_0}H(r,\sigma,u_\varepsilon)\frac{dr}{r}
\leq \frac{1}{4} \int_{2N\varepsilon}^{2r_0}H(r,\sigma,u_\varepsilon)\frac{dr}{r}
\end{equation*}
and this implies that
\begin{equation*}
\int_{\theta N\varepsilon}^{2r_0}H(s,\sigma,u_\varepsilon)\frac{ds}{s}
\leq \frac{4}{3} \int_{\theta r_0}^{2r_0}H(s,\sigma,u_\varepsilon)\frac{ds}{s}
\leq CH(2r_0,\sigma,u_\varepsilon).
\end{equation*}
Thus we deduce from the above estimate that
\begin{equation}\label{f:4.6}
H(r,\sigma,u_\varepsilon)\leq CH(r_0,\sigma,u_\varepsilon)
\end{equation}
holds for any $\varepsilon\leq r< r_0$. We mention that due to the relationship
$N\gtrsim \theta^{-\frac{\sigma}{\eta}-\frac{d}{d\eta}}$, it is trivial to prove
the stated result $\eqref{f:4.6}$ in the case of $\varepsilon\leq r\leq \theta N\varepsilon$,
as well as, of $r_0\leq r\leq 1$.
Then it is clear to see that
\begin{equation*}
\begin{aligned}
\Big(\dashint_{B(0,r)}|\nabla u_\varepsilon|^p \Big)^{\frac{1}{p}}
&\lesssim r^{\sigma-1}H(2r,\alpha,u_\varepsilon)\\
&\lesssim r^{\sigma-1}H(R,\alpha,u_\varepsilon)
\lesssim \Big(\frac{r}{R}\Big)^{\sigma-1}
\Big(\dashint_{B(0,R)}|\nabla u_\varepsilon|^p \Big)^{\frac{1}{p}}
\end{aligned}
\end{equation*}
where we note that $u_\varepsilon$ up to a constant is still a solution to
$\mathcal{L}_\varepsilon u_\varepsilon = 0$ in $B(0,1)$.
The first step follows from Caccioppoli's inequality $\eqref{pri:4.2}$, and second one
comes from the estimate $\eqref{f:4.6}$,
while we use Poincar\'e's inequality in the last step. Then, by setting $\kappa=p(1-\sigma)$,
the above estimate consequently leads to the desired result $\eqref{pri:1.11}$, and we have
completed the whole proof.
\qed

\section{Examples on the assumption $\eqref{c:2}$}\label{section:5}

We take $A(y,\xi) = a(y)|\xi|^{p-2}\xi$ as an example, with the assumption that
\begin{equation*}
a(y+z) = a(y) \qquad\text{and}
\qquad \mu_0 \leq a(y) \leq \mu_1 \qquad \forall y\in\mathbb{R}^d, z\in\mathbb{Z}^d,
\end{equation*}
which is known as weighted p-Laplace equation.
In the case of $d=1$, the homogenized
equation could be figured out, and
\begin{equation*}
  \widehat{A}(\xi) = \widehat{a}|\xi|^{p-2}\xi
\end{equation*}
with
\begin{equation*}
\widehat{a} = \Big(\int_0^1\big[a(y)\big]^{\frac{1}{1-p}}dy\Big)^{1-p}
\end{equation*}
(see \cite[Example 25.4]{P}).

Then we turn to the case $d\geq 2$. For the higher dimensions, we assume that
\begin{equation}\label{a:6}
 \frac{\partial}{\partial y_i}\big[a(y)\big] = \frac{\partial}{\partial y_j}\big[a(y)\big],
\qquad i\not= j \text{~and~} i,j =1,2,\cdots,d.
\end{equation}
In such the case, one may construct the first order corrector in the form of
\begin{equation}\label{f:5.1}
  N(y,\xi) = \underbrace{\Big[\chi(y) + \sum_{k=1}^d y_k\Big]}_{T_1}
\underbrace{\sum_{k=1}^d\xi_k}_{T_2} - y\cdot\xi,
\qquad y\in Y,
\end{equation}
where the corrector $\chi$ satisfies
\begin{equation}\label{pde:5.2}
\left\{\begin{aligned}
&\text{div} \Big[a(y)\big|\nabla\chi +1\big|^{p-2}(\nabla\chi+1)\Big] = 0 \quad\text{in}~ Y,\\
& \chi\in W^{1,p}_{per}(Y), \quad \dashint_{Y}\chi = 0.
\end{aligned}\right.
\end{equation}
It is known that there exists the unique solution of the above cell problem, and
this implies that $N(\cdot,\xi)\in W^{1,p}_{per}(Y)$ satisfies the equation $\eqref{pde:1.2}$.
The homogenized equation is given by
\begin{equation}\label{pde:5.1}
-\widehat{a}\cdot\nabla \Big(|\sum_{k=1}^d\nabla_k u_0|^{p-2}\sum_{k=1}^d\nabla_k u_0\Big) = F
\qquad\text{in~}\Omega,
\end{equation}
where $\widehat{a}\in\mathbb{R}^d$ is a vector, and its component is given by
\begin{equation*}
  \widehat{a}_i =\int_{Y}\Big[a(y)\big|\nabla \chi(y)+1\big|^{p-2}
\big(\nabla_i\chi(y)+1\big)\Big]dy.
\end{equation*}
Note that the equation $\eqref{pde:5.1}$
is solvable according to $\widehat{a}$ direction,
which means if we take $F(x) = F(\widehat{a}\cdot x)$ and
$u_0(x) = u_0(\widehat{a}\cdot x)$ for any $x\in\Omega$, then
the equation $\eqref{pde:5.1}$ is transformed  into
\begin{equation*}
-\frac{\partial}{\partial r}
\bigg(\Big|\frac{\partial u_0}{\partial r}\Big|^{p-2}
\frac{\partial u_0}{\partial r}\bigg)
= \Big(|\widehat{a}|^2\big|\sum_{k=1}^d\widehat{a}_k\big|^{p-2}
\sum_{k=1}^d\widehat{a}_k\Big)^{-1}F(r)
\qquad\text{and}\quad  r=\widehat{a}\cdot x,
\end{equation*}
provided
\begin{equation}\label{c:6}
  \sum_{k=1}^d\widehat{a}_k\not=0.
\end{equation}
Thus, we go back to the equation $\eqref{pde:5.2}$, and it is possible to find a radial
solution (i.e., $\chi(y)=\chi(|y|)$) to meet the condition $\eqref{c:6}$, since
it is fine to assume $a(y) = a(|y|)$ under the assumption
$\eqref{a:6}$. In such the case,
we may have $\widehat{a}_i = \widehat{a}_j$ for any $i\not=j$.
\begin{remark}
\emph{In fact, the term $T_1$ of $\eqref{f:5.1}$ leads to the equation $\eqref{pde:5.2}$, while
the form of $T_2$ of $\eqref{f:5.1}$ determines the formula of
homogenized operator in $\eqref{pde:5.1}$, separately.
Thus, if the corrector $\chi$ is a radial solution, then we may further simplify $T_2$ into
$\sum_{k=1}^{d}\xi_k = d\xi_1$ for example, which gives
\begin{equation}\label{f:5.3}
 |\sum_{k=1}^d\xi_k|\geq |\xi|/\sqrt{d}.
\end{equation}
We mention that the above inequality will be true if there exists an integer $i_0\in[1,d]$
such that $\max_{i\not=i_0}|\xi_i|\leq |\xi_j|$ and $\sum_{i\not=i_0}\xi_i\geq 0$. In fact,
to verify the conditions $\eqref{c:4}$ and $\eqref{c:5}$, the inequality $\eqref{f:5.3}$ will
be used in later computations.}
\end{remark}

Then we plan to show the reasonableness of $\eqref{f:5.1}$, at least, in the case of $p=2$.
According to the asymptotic expansion $u_\varepsilon = u_0(x) + \varepsilon u_1(x,y) +
\varepsilon^2 u_2(x,y)+\cdots,$ one may derive that
\begin{equation}\label{eq:5.1}
 \frac{\partial}{\partial y_i}\Big(a(y)\frac{\partial u_1}{\partial y_i}\Big) =
- \frac{\partial}{\partial y_i}\big[a(y)\big]\frac{\partial u_0}{\partial x_i}
\end{equation}
(see for example \cite{S5}), where $y=x/\varepsilon$. On the one hand, let
$u_1(x,y) = N(y,\nabla u_0)$, and it is not hard to verify that
$N(y,\nabla u_0)$ in the form of $\eqref{f:5.1}$
in fact satisfies the above equation under the condition $\eqref{a:6}$
coupled with the equation $\eqref{pde:5.2}$. Firstly, in view of $\eqref{f:5.1}$,
\begin{equation*}
 \nabla_{y_i} N(y,\nabla u_0) = \Big(\nabla_{y_i}\chi + e_i\Big)
\Big(\sum_{k=1}^d\nabla_{x_k}u_0\Big) - \nabla_{x_i}u_0.
\end{equation*}
Then
\begin{equation*}
\begin{aligned}
\nabla_{y_i}\big\{a(y)\nabla_{y_i} N(y,\nabla u_0)\big\}
&= \nabla_{y_i}\Big\{a(y)(\nabla_{y_i}\chi + 1)\Big\}\Big(\sum_{k=1}^d\nabla_{x_k}u_0\Big)
- \nabla_{y_i}\big[a(y)\big]\nabla_{x_i}u_0\\
&= - \nabla_{y_i}\big[a(y)\big]\nabla_{x_i}u_0,
\end{aligned}
\end{equation*}
where we use the equation $\eqref{pde:5.2}$ in the case $p=2$.
In this sense, the form of $\eqref{f:5.1}$ is acceptable. On the other hand,
we say that the form of $\eqref{f:5.1}$ leads to the same equation that the corrector $\chi$
ought to satisfy. From $\eqref{eq:5.1}$, we have
\begin{equation*}
 \frac{\partial}{\partial y_i}\Big(a(y)\frac{\partial u_1}{\partial y_i}\Big) =
- \frac{\partial}{\partial y_i}\big[a(y)\big]
\Big(\sum_{k=1}^d\frac{\partial u_0}{\partial x_k}\Big),
\end{equation*}
that is due to the assumption $\eqref{a:6}$.
Let $u_1(x,y) = \chi(y)(\sum_{k=1}^d \partial u_0/ \partial x_k)$, it is clear to see that
$\chi$ admits a solution of
\begin{equation*}
\nabla_{y_i}\Big\{a(y)\big(\nabla_{y_i}\chi + 1\big)\Big\} = 0
\quad\text{in}\quad Y,
\end{equation*}
which is exactly the form of the equation $\eqref{pde:5.2}$ in the case of $p=2$.
Thus, in terms of the equation that the corrector obeys,
the form of $\eqref{f:5.1}$ plays the same role as the above decomposition of
$u_1$ in the linear case.

Now, we can show the desired estimates $\eqref{c:4}$ and $\eqref{c:5}$. Set
\begin{equation}\label{f:5.2}
 P(y,\xi)  = \big(\nabla\chi + 1\big)\sum_{k=1}^d\xi_k
\end{equation}
with the quantity $\sum_{k=1}^d\xi_k$ admitting the inequality $\eqref{f:5.3}$.
Then, it follows that
\begin{equation}\label{f:5.4}
\begin{aligned}
|P(y,\xi)| &\geq |\xi||\nabla \chi + 1|/\sqrt{d},\\
|P(y,\xi)-P(y,\xi^\prime)|
&\geq |\xi-\xi^\prime|\Big|\big|\nabla\chi+1\big|/\sqrt{d}-1\Big|, \\
 |P(y,\xi)-P(y,\xi^\prime)|&\leq |\xi-\xi^\prime|(|\nabla \chi| + 2).
\end{aligned}
\end{equation}
In the case $2<p<\infty$, we have
\begin{equation*}
 \begin{aligned}
&\quad \int_{Y}|P(y,\xi)|^{p-2}|P(y,\xi)-P(\cdot,\xi^\prime)|^2 dy\\
&\gtrsim |\xi|^{p-2}|\xi-\xi^\prime|^2\int_Y
|\nabla \chi + 1|^{p-2}\big|\big|\nabla\chi+1\big|/\sqrt{d}-1\big|^2dy,
\end{aligned}
\end{equation*}
and this together with
\begin{equation*}
\begin{aligned}
\big<\widehat{A}(\xi)-\widehat{A}(\xi^\prime),\xi-\xi^\prime\big>
&\geq \mu_0\int_Y |P(y,\xi) - P(y,\xi^\prime)|^2(|P(y,\xi)| + |P(y,\xi^\prime)|)^{p-2}dy\\
&\gtrsim  \int_Y
|P(y,\xi) - P(y,\xi^\prime)|^2 \big(|P(y,\xi)|^{p-2}+|P(y,\xi^\prime)|^{p-2}\big)dy\\
\end{aligned}
\end{equation*}
gives the desired estimate $\eqref{c:4}$. Then we proceed to show $\eqref{c:5}$,
and on account of the assumption $\eqref{c:1}$ in the case of $1<p<2$, we  arrive at
\begin{equation*}
\begin{aligned}
|\widehat{A}(\xi) - \widehat{A}(\xi^\prime)|
&\leq \mu_1\int_Y
|P(y,\xi) - P(y,\xi^\prime)| \big(|P(y,\xi)|+|P(y,\xi^\prime)|\big)^{p-2}dy \\
&\lesssim|\xi-\xi^\prime|\big(|\xi|+|\xi^\prime|\big)^{p-2}\int_Y
|\nabla\chi+2| |\nabla\chi+1|^{p-2}dy \\
&\lesssim |\xi-\xi^\prime|\big(|\xi|+|\xi^\prime|\big)^{p-2}
\end{aligned}
\end{equation*}
and this gives $\eqref{c:5}$. In the above estimate, we use $\eqref{f:5.4}$
in the second inequality. Here the up to constant also depend on
the corrector $\chi$ given by $\eqref{pde:5.2}$. In the end, we mention that, by the analogy idea
one may construct other examples, such as $A_i(y,\xi) = a_{ij}(y)|\xi|^{p-2}\xi_j$ under
a similar assumption to $\eqref{a:6}$, and we omit these details here.

\begin{center}
\textbf{Acknowledgements}
\end{center}
The second author appreciated Prof. Zhongwei Shen and Prof. Felix Otto
for their illuminating instructions and encourages as he studied the homogenization theory. The first and third authors were supported by the National Natural Science Foundation of China
(Grant NO. 11471147).

\end{document}